\documentclass[11pt]{amsart}
\pdfoutput=1

\usepackage{microtype}
\usepackage[pdftex,colorlinks,citecolor=black,linkcolor=black,urlcolor=black,bookmarks=false]{hyperref}
\usepackage[numeric,initials,nobysame]{amsrefs}
\usepackage[margin=1in]{geometry}

\theoremstyle{plain}
\newtheorem{thm}{Theorem}[section]
\newtheorem{lem}[thm]{Lemma}

\newtheorem{prop}[thm]{Proposition}

\theoremstyle{definition}
\newtheorem{definition}[thm]{Definition}
\newtheorem{ex}[thm]{Example}
\theoremstyle{remark}
\newtheorem{rem}[thm]{Remark}

\numberwithin{equation}{section}
\numberwithin{thm}{section}

\newcommand{\paren}[1]{{\left( {#1} \right)}}
\newcommand{\norm}[1]{{\left\| {#1} \right\|}}
\newcommand{\scprod}[2]{{\left\langle {#1},{#2} \right\rangle}}

\newcommand{\LP}{\textup{LP}}
\newcommand{\tail}{{\textup{tail}}}
\newcommand{\Bern}{\mathop{\textup{Bern}}}
\newcommand{\R}{\mathbb R}
\newcommand{\Z}{\mathbb Z}
\newcommand{\E}{\mathbb E}
\newcommand{\N}{\mathbb N}
\newcommand{\eps}{\varepsilon}
\newcommand{\cA}{\mathcal{A}}
\newcommand{\cB}{\mathcal{B}}
\newcommand{\cF}{\mathcal{F}}
\newcommand{\cH}{\mathcal{H}}
\newcommand{\cM}{\mathcal{M}}
\newcommand{\cP}{\mathcal{P}}
\newcommand{\pp}{\mathbf{p}}

\newcommand{\WW}[1]{\mathsf{W}[#1]}
\newcommand{\argmax}{\operatornamewithlimits{argmax}}
\newcommand{\argmin}{\operatornamewithlimits{argmin}}
\newcommand{\symmdiff}{\mathbin{\triangle}}

\newcommand{\HH}{{H}}
\newcommand{\Q}{{Q}}
\newcommand{\Hold}[2]{{\cH_{{#1},{#2}}}}
\newcommand{\HOLD}{\Hold{C}{\alpha}}

\title[Consistent Nonparametric Estimation\dots]{Consistent Nonparametric Estimation\\
for Heavy-tailed Sparse Graphs}

\author[C.\ Borgs]{Christian Borgs}
\address{Microsoft Research\\
One Memorial Drive\\
Cambridge, MA 02142} \email{borgs@microsoft.com}

\author[J.\ T.\ Chayes]{Jennifer T.\ Chayes}
\address{Microsoft Research\\
One Memorial Drive\\
Cambridge, MA 02142} \email{jchayes@microsoft.com}

\author[H.\ Cohn]{Henry Cohn}
\address{Microsoft Research\\
One Memorial Drive\\
Cambridge, MA 02142} \email{cohn@microsoft.com}

\author[S.\ Ganguly]{Shirshendu Ganguly}
\address{Department of Mathematics\\
University of Washington\\
Seattle, WA 98195}
\email{sganguly@math.washington.edu}
\thanks{Ganguly was supported by an internship at Microsoft Research New England.}

\begin{document}

\begin{abstract}
We study graphons as a non-parametric generalization of stochastic block
models, and show how to obtain compactly represented estimators for sparse
networks in this framework.  Our algorithms and analysis go beyond
previous work in several ways. First, we relax the usual boundedness
assumption for the generating graphon and instead treat arbitrary
integrable graphons, so that we can handle networks with long tails in
their degree distributions. Second, again motivated by real-world
applications, we relax the usual assumption that the graphon is defined on
the unit interval, to allow latent position graphs where the latent
positions live in a more general space, and we characterize
identifiability for these graphons and their underlying position spaces.

We analyze three algorithms. The first is a least squares algorithm, which
gives an approximation we prove to be consistent for all square-integrable
graphons, with errors expressed in terms of the best possible stochastic
block model approximation to the generating graphon.  Next, we analyze a
generalization based on the cut norm, which works for any integrable
graphon (not necessarily square-integrable).  Finally, we show that
clustering based on degrees works whenever the underlying degree
distribution is atomless. Unlike the previous two algorithms, this third
one runs in polynomial time.
\end{abstract}

\maketitle

\tableofcontents

\section{Introduction}

Motivated by numerous real-world technological, social, and biological
networks, the study of large networks has become increasingly important.
Much work in the statistics and machine learning communities has focused on
the questions of modeling and estimation for these networks.

\subsection{Stochastic block models and $W$-random graphs} Many previous papers have
described these networks in terms of parametric models, one of the most
popular being the stochastic block model, introduced in \cite{HLL83}. These
models can be characterized by a vector of probabilities $\mathbf p=(p_i)$
on a finite set of  communities and a matrix $B=(\beta_{ij})$  of
``affinities.''  Given these parameters, one then generates a graph on
labeled $n$ nodes by assigning a community to each vertex, independently at
random according to the probability distribution $\mathbf p$, and then
connecting vertices belonging to communities $i$ and $j$ with probability
$\beta_{ij}$. Hence the stochastic block model for $k$ groups is determined
by $(k-1) + k(k+1)/2$ parameters.  Such a model is often considered a
reasonable approximation of a small social network characterized by a
limited number of communities.

More recently, motivated by extremely large networks, researchers have begun
to consider non-parametric stochastic block models, for which there is a
continuous family of communities, i.e., for which the $k \times k$ matrix of
edge probabilities is replaced by a two-dimensional function. The
non-parametric models we study in this paper are usually referred to as
$W$-random graphs or latent position graphs. In the most general setup, such
a model is defined in terms of a probability space\footnote{As usual, a full
specification of the probability space $(\Omega,\pi)$ requires the
specification of a $\sigma$-algebra $\mathcal F$ in addition to the
underlying space $\Omega$ and measure $\pi$.  We will discuss
measure-theoretic technicalities only when they seem important or could
potentially cause confusion.} $(\Omega,\pi)$ (the space of latent positions
or features) and a \emph{graphon} $W$ over $(\Omega,\pi)$, defined as an
integrable, non-negative function on $\Omega\times\Omega$ that is symmetric
in the sense that $W(x,y)=W(y,x)$ for all $x,y\in \Omega$. To generate a
graph on $n$ nodes, one then chooses $n$ ``positions'' $x_1,\dots,x_n$
i.i.d.\ at random from $(\Omega,\pi)$ and, conditioned on these, chooses
edges independently, with the probability of an edge between vertices $i$
and $j$ given by $W(x_i,x_j)$.  The resulting graph is called a
\emph{$W$-random graph}.

As originally proposed in \cite{HRH02}, the space of latent positions
$\Omega$ comes equipped with a metric and the probability of connection is a
function of distance, but the more general setting we have described is
commonly studied. Note that in the dense setting, this model is quite
natural, since it can be shown \cites{H79,A81,DJ08} that if a random graph
$G$ is the restriction of (an ergodic component of) an infinite,
exchangeable random graph, then $G$ must be an instance of a $W$-random
graph for some function $W$ with values in $[0,1]$. Due to this connection,
$W$-random graph models are often called exchangeable graph models.

\subsection{Dense and sparse graphs}
\label{sec:intro-dense-sparse}
To model sparse graphs in this
non-parametric setup, one uses connection probabilities which are given by
symmetric function $W$ times a \emph{target density} $\rho$, leading to the
model of ``inhomogeneous random graphs'' defined in \cite{BJR07}, with nodes
$i$ and $j$ now being connected with probability $\min\{1,\rho
W(x_i,x_j)\}$. For both dense  and sparse graphs,  this kind of model is
related to the theory of convergent graph sequences
\cite{BCLSV06,BCLSV08,BR,BCLSV12,BCCZ14a,BCCZ14b}.  In the setting of dense
graph limits, $W$-random graphs were first explicitly proposed in
\cite{LS06}, although they can be implicitly traced back to the much earlier
work of \cite{H79} and \cite{A81} mentioned above.  The term graphon
originated in \cite{BCLSV08}.

While for dense graphs one only needs to consider bounded graphons (indeed,
the results of \cite{H79,A81} imply that it is enough to consider graphons
that take values in $[0,1]$), this boundedness assumption is not very
natural for sparse graphs.  Indeed, it is not hard to see that for bounded
graphons $W$, all degrees in a $W$-random graph are of the same order
(except in very sparse settings, where the maximum degree might differ from
the average degree by a logarithmic factor).  While this is no problem for
dense graphs, since here the average degree is of the same order as the
number of vertices, and hence automatically of the same order as the maximal
degree, it is a serious restriction for sparse graphs.  Indeed, many
real-world networks have long-tailed degree distributions. For applications,
one would therefore want to consider unbounded graphons $W$.

\subsection{Estimation and previous literature}
How can we estimate a graphon $W$ given a sample $G$ of a $W$-random graph?
This problem encapsulates the idea of inferring the underlying structure in
a random network.

For the special case where $W$ is a stochastic block model, the estimation
problem is closely related to the problem of graph partitioning and has been
intensely studied in the literature \cite{WBB76,HLL83,FMW85}, using methods
that range from maximum likelihood estimates \cite{WW87} and Gibbs sampling
\cite{SN97} or simulated annealing \cite{JS98} to spectral clustering
\cite{B87,McSherry01,DHKM06,C-O10,CCT12} and tensor algebra \cite{AGHK14}.
Proving consistency of these methods is often not hard in the dense regime,
but it becomes more difficult for sparse graphs. See, for example,
\cite{leirin,LeiZhu14} for a proof of consistency for spectral clustering
when the average degree is as small as $\log n$, and \cite{AS15,AS15b} for
an effective algorithm that is provably consistent as long as the average
degree diverges.

Estimating graphons that are not block models is more challenging. This
problem is implicit in \cite{Kal99}, but the first explicit discussion of
the non-parametric problem we are aware of was
given in \cite{BC09},
even though the actual consistency proof there is still limited to stochastic
block models with a fixed number of blocks. The restriction to a fixed
number of blocks was relaxed in
\cite{RCY11}
and \cite{CWA12}.
The full non-parametric model was studied in
\cite{BCL11},
under the assumption that none of the eigenfunctions of the operator
associated with the kernel $W$ is orthogonal to the constant function $1$
and the eigenvalues are distinct.

Many further papers have been written on graphon estimation, including
\cite{RCY11,CSX12,
  LOGR12,
  ACC13,
  ACBL13,
  FSTVP13,
  LR13,
  M13,
  QR13,
  TSP13,
  WO13,
  ABH14,
  CAF14,
  CA14,
  CW14,
  YangHA14,
  GaoLZ14,
  HWX14,
  OW14,
  V14,
  YP14,YP14b,
  Chatterjee15,
  CRV15, HWX15, BCS15, klopp}. Each paper makes different assumptions about the
density and the underlying graphon.  Strong results are known for dense
graphs: \cite{Chatterjee15} shows how to approximate arbitrary measurable
graphons $W$ with values in $[0,1]$ given a dense $W$-random graph, and
\cite{GaoLZ14} attains an optimal rate for  least squares estimators of
both
stochastic block models and H\"older-continuous graphons from a dense graph.
For sparse graphs, \cite{WO13}
proves convergence of a maximum likelihood estimator under the assumption
that $W$ is bounded, bounded away from $0$, and H\"older-continuous.  Most
recently, \cite{BCS15} introduces a modified version of the least squares
algorithm that optimizes over block models with bounded $L^\infty$ norm;
this algorithm achieves consistency for arbitrary bounded graphons and
arbitrary densities, as long as the average degree diverges with the number
of vertices.  The same paper also gives a differentially private version of
the least squares algorithm which works again for arbitrary bounded
graphons, now requiring that the average degree grows at least like a
logarithm of the number of vertices.
Independently, \cite{klopp} proposes and analyzes the
modified (non-private) algorithm and proves  matching upper and lower
bounds for the rates achieved by this algorithm.

But more important than some of the technical assumptions used by previous
authors is the fact that \emph{all} the previous results we are aware of
require $W$ to be \emph{bounded}. As pointed out before, this assumption,
while natural for dense graphs, rules out most degree distributions observed
in real-world networks. Our goal here is  to remove this assumption.

\subsection{Identifiability} Before summarizing our contributions, we need to
discuss the fact that in general, $W$ cannot be uniquely determined from the
observation of even the full sequence $(G_n)_{n\geq 1}$, a problem called
the identifiability issue in the literature; see, for example,
\cite{BC09,CA14}. To discuss this, consider two graphons $W$ and $W'$ over
two probability spaces $(\Omega,\pi)$ and $(\Omega',\pi')$, as well as a
measure-preserving map $\phi\colon(\Omega,\pi)\to(\Omega',\pi')$. Define the
pullback of $W'$ to $\Omega$ as the graphon $(W')^\phi$ defined by
$(W')^\phi(x,y)=W'(\phi(x),\phi(y))$. It is  not hard to see that then the
sequences of random graphs generated from $W$ and $W'$ have the same
distribution if $W=(W')^\phi$. While it was stated in some of the early
literature on graphon estimation that the converse is true as well,
that turns out to be false; see, for example, Example~\ref{ex:equiv-46}
below for a counterexample. To formulate the correct statement, we define
$W$ and $W'$ to be \emph{equivalent} if there exists a third graphon $U$
over a probability space $(\Omega'',\pi'')$ such that $W=U^\phi$ and
$W'=U^\psi$ for two measure-preserving maps $\phi,\psi$ from $\Omega$ and
$\Omega'$ to $\Omega''$; see Section~\ref{sec:identify} for more details.

With this definition, we are now ready to characterize the full extent to
which $W$ is not identifiable: \emph{The sequences generated from  two
graphons $W$ and $W'$ are identically distributed if and only if $W$ and
$W'$ are equivalent.} In the dense case, this was proved in \cite{DJ08} for
the case where $W$ and $W'$ are defined over $[0,1]$ equipped with the
uniform distribution,  and for the case of general probability spaces it
follows from the results of \cite{BCL10} by a simple argument involving
subgraph counts.  But for the sparse case, and general integrable (rather
than bounded) graphons, this is a new result, proved in this paper
(Theorem~\ref{thm:equiv} in Section~\ref{sec:identify}). Thus, both the
feature space $(\Omega,\pi)$ and the graphon $W$ are unobservable in
general, and even if we fix the feature space there is no ``canonical
graphon'' an estimation procedure can output.  The best we can hope for is a
representative from an equivalence class of graphons.

In light of these facts, the natural way of dealing with the identification
problem is to admit that there is nothing canonical about any particular
representative $W$, and to define consistency as consistency with respect to
a metric between equivalence classes, rather than between graphons
themselves.  The papers \cite{WO13} and \cite{klopp} follow this strategy, by
using a variant of the $L^2$ metric which is a metric over equivalence classes.
Most other papers either avoid the identifiability problem altogether, by redefining
the problem as the problem of finding an approximation $\hat \HH$ to the
matrix $\HH_n(W)=(W(x_i,x_j))_{1\le i,j\le n}$ (see, for example,
\cite{Chatterjee15} or \cite{GaoLZ14}), or by making additional assumptions
which guarantee the existence of a canonical representative, e.g., by
postulating that $W$ is defined over the interval $[0,1]$ and assuming that
after a measure-preserving transformation, the ``degrees'' $W_x=\int_0^1
W(x,y)\,dy$ are strictly monotone in $x$, in which case there \emph{is} a
canonical representative of the graphon $W$.

\subsection{Goals} In this paper, we follow the spirit of \cite{WO13} and
define consistency with respect to a metric on equivalence classes of
graphons, but in contrast to \cite{WO13}, we allow for more general spaces
than just the uniform distribution over the unit interval.\footnote{Note
that from a purely measure theoretic approach to $W$-random graphs, one can
restrict oneself to graphons over the unit interval without any loss of
generality, since every integrable graphon $W$ is equivalent to a graphon
$W'$ defined over $[0,1]$ equipped with the uniform distribution; see
Theorem~\ref{thm:equiv-over-01} in Section~\ref{sec:identify}.  However,
when $W$ is given in an application, it is often a continuous function over
a higher dimensional space, and while $W'$ leads to the same distribution of
$W$-random graphs, the transformation from $W$ to $W'$ ruins continuity,
which is often needed to prove good approximation bounds. For applications,
the general setup is therefore more natural.} To define our notion of
distance, we recall that a \emph{coupling} between two probability measures
$\pi,\pi'$ is a measure $\nu$ on the product space such that the projections
of $\nu$ to the two coordinates are equal to $\pi$ and $\pi'$, respectively.
Given $p\geq 1$ and two $L^p$ graphons $W$ over $(\Omega,\pi)$ and $W'$ over
$(\Omega',\pi')$ (i.e., graphons such that $\int_\Omega |W|^p \, d\pi < \infty$
and $\int_{\Omega'} |W'|^p \, d\pi' < \infty$), we then define the distance
$\delta_p(W,W')$ by
\begin{equation}
\label{del-p}
\delta_p(W,W')=\inf_{\nu}\paren{\int \Bigl|W(x,y)-W'(x',y')\Bigr|^p \,d\nu(x,x')\,d\nu(y,y')}^{1/p},
\end{equation}
where the infimum is over all couplings $\nu$ of $\pi$ and $\pi'$.

Having defined a metric on equivalence classes of graphons, we can now
formulate the estimation problem considered in this paper: \emph{Given a single
instance} of a $W$-random graph defined on an unobserved probability space
$(\Omega,\pi)$, find an algorithm that (a) outputs an estimator $\widehat W$
such that $\widehat W$ has a \emph{concise representation} whose size grows
only slowly with $n$; (b) estimates $W$ consistently \emph{assuming just
integrability conditions}; (c) works for \emph{arbitrary target densities},
as long as the graph is not too sparse (say has divergent average degree);
and (d) runs in \emph{polynomial time.}

While efficiency (property (d)) is clearly important for practical
applications, our main focus in this paper will be the fundamental problem
of consistent estimation under as few restrictions on $W$ as possible, i.e.,
algorithms achieving properties (a)--(c). Indeed, none of the three
algorithms we study in this paper achieves all four properties.  Two of them
achieve (a)--(c), and hence solve the desired problem of consistent
estimation, but do not run in polynomial time.  The third achieves (a), (c),
and (d), and hence is efficient, but requires an additional condition to
ensure consistency.

\subsection{Summary of results} In this paper, our estimator $\widehat W$
will be given in terms of a block model, with a number of blocks that grows
slowly with the number of vertices of the input graph. Given this framework,
it is natural to compare the performance of our algorithm to the best
possible block model in a suitable class of block models.  Here we  consider
the class $\cB_{\geq\kappa}=\{(\mathbf p, B)\colon \min_i p_i\geq\kappa\}$
of all block models with minimal block size at least $\kappa$. For an
approximation outputting a block model in $\cB_{\geq\kappa}$, the best error
we could achieve is
\begin{equation}
\label{eps_kappa}
\eps_{\geq\kappa}^{(p)}(W)=\inf_{W'\in\cB_{\geq \kappa}}\delta_p(W,W').
\end{equation}
We often refer to this benchmark as an \emph{oracle error}, since it is the
best an oracle with access to the unknown $W$ could do.  Our goal is to prove
\emph{oracle inequalities} that bound the estimation error in terms of the oracle
error, as well as a few additional terms that account for variance and the visibility
of heavy tails at finite scale.

When establishing
the estimation error for $W$, we usually first prove a bound on the
estimation error for the intermediate matrix $\Q_n=(\min\{1, \rho
W(x_i,x_j)\})_{i,j\in [n]}$, which will be expressed in terms of an oracle
error for $\Q_n$ plus a concentration error stemming from the fact that, even
after conditioning on $\Q_n$, the observed graph $G_n$ is random; see
Theorem~\ref{thm:L2-H} and Theorem~\ref{thm:cut-H} below.  In a second
step, we then prove consistency for the original estimation
error, given bounds that estimate the difference between $\widehat W$ and
$W$. Note that part of the literature stops at the first step, effectively
avoiding the identifiability issue discussed above.

In this paper, we consider three algorithms for producing a block model
approximation to $W$ from a single instance of a $W$-random graph $G$: two
inefficient ones and one whose running time is polynomial in $n$.
\begin{enumerate}
\item The well-known \emph{least squares algorithm}, which has been
    analyzed under various additional assumptions on $W$, until recently
    \cite{BCS15} not even covering arbitrary bounded graphons.  Here we
    will prove consistency of this algorithm in the metric $\delta_2$ for
    arbitrary $L^2$ graphons.
\item A \emph{least cut norm algorithm}, which we prove to be consistent
    under the cut norm for arbitrary $L^1$ graphons.
\item A \emph{degree sorting algorithm}, which we show is consistent
    whenever the degree distribution of $W$ is atomless.  (Graphons with
    this property are equivalent to graphons over $[0,1]$ such that
    $W_x=\int_0^1 W(x,y)\, dy$ is strictly monotone in $x$.)  This
    algorithm runs in polynomial time.
\end{enumerate}

To state our results, we need a few definitions.  As usual, $[n]$ denotes
the set $\{1,\dots,n\}$. Given an $n\times n$ matrix $A$, we use $\|A\|_p$
to denote its $L^p$ norm, defined by $\|A\|_p^p=\frac
1{n^2}\sum_{i,j}|A_{ij}|^p$. Given a graph $G$ on $[n]$, we use $A(G)$ to
denote the adjacency matrix of $G$, and $\rho(G)=\|A(G)\|_1$ to denote its
density. We identify partitions of $[n]$ into $k$ classes (some of which can
be empty) with maps $\pi\colon [n]\to[k]$, where
$V_i=V_i(\pi)=\pi^{-1}(\{i\})$ is the $i^{\text{th}}$ class of the
partition. Given such a map and a $k\times k$ matrix $B$, we will use
$B^\pi$ for the $n\times n$ matrix with entries $B_{\pi(i),\pi(j)}$.
Finally, for an $n\times n$ matrix $A$, we use $A_\pi$ to denote the  matrix
where for each $(x,y)\in V_i\times V_i$, the matrix element $A_{xy}$ is
replaced by the average over $V_i\times V_j$, and $A/\pi$ to denote the
$k\times k$ matrix of block averages
\[
(A/\pi)_{ij}=\frac 1{|V_i|\,|V_j|}\sum_{(u,v)\in V_i\times V_j}A_{uv},
\]
defined to be $0$ if either $V_i$ or $V_j$ is empty; note that the two are
related by $A_\pi=(A/\pi)^\pi$.

Throughout this paper, we will assume that the graph is sparse (in the sense
that $\rho\to 0$), but that it has divergent average degree (i.e., we assume
that $n\rho\to\infty$).  Under these assumptions we will prove the following
results.

\emph{Least squares algorithm.} Given an input graph $G$ on $n$ vertices and
a parameter $\kappa\in (0,1]$ such that $\kappa n\geq 1$, let
\begin{equation}
\label{L2-algorithm}
(\hat\pi,\hat B)\in \argmin_{\pi,B}\|A(G)-B^\pi\|_2,
\end{equation}
where the optimization is over all $k\times k$ matrices $B$ and all
partitions $\pi\colon [n]\to [k]$ such that all non-empty classes of $\pi$
have size at least $\lfloor \kappa n\rfloor$, with $k$ chosen such that it
can accommodate all such partitions, say $k=\lceil \frac n{\lfloor
n\kappa\rfloor}\rceil$. Setting $\hat p_i$ to be the relative size of the
$i^{\text{th}}$ partition class of $\hat\pi$, i.e.,
\[
\hat p_i=\frac 1n|V_i(\hat\pi)|,
\]
the least squares algorithm then outputs the block model $\widehat
W=(\hat{\mathbf p},\hat B)$. Note that the above minimization problem is
slightly helped by the fact that we minimize the $L^2$ norm.  For a given
$\pi$, the minimizer $\hat B$ can therefore be obtained by averaging $A(G)$
over the classes of $\pi$, showing that ${\hat B}$ is of the form
$A(G)/{\pi}$. Nevertheless the algorithm is inefficient, since we still need
to minimize over partitions $\pi\colon [n]\to [k]$.

Our main result concerning this algorithm is that if $G$ is a $W$-random
graph at target density $\rho$ and $W\in L^2$, then the algorithm is
consistent in the sense that
\[
\delta_2\left(\frac 1\rho\widehat W,\frac 1{\|W\|_1}W\right)\to 0
\]
with probability one as $n \to \infty$, as long as $\kappa\to 0$ and
$\kappa^{-2}\log(1/\kappa)=o(n\rho)$.  If instead of almost sure convergence
we content ourselves with convergence in probability, then for $\kappa\in
(n^{-1},1]$ and $\frac{1+\log (1/\kappa)}{\kappa^2}=O(\rho n)$, we have the
oracle inequality
\[
  \delta_2\paren{\frac 1\rho\widehat W, W}
  =
   O_p\paren{\eps_{\geq\kappa}^{(2)}(W)+\sqrt[4]{\frac{1+\log (1/\kappa)}{\kappa^2\rho n}}
  +\sqrt[4]{\frac {\log n}{\kappa n}}
  +\tail_{\rho}^{(2)}(W)},
  \]
when $\|W\|_1=1$, where $\tail_{\rho}^{(2)}(W)$ is a term which measures the
difference between $W$ and $\frac 1\rho\min\{1,\rho W\}$ in the $L^2$ norm;
see Theorem~\ref{thm:L2} for the details.

The four error terms above arise for
different reasons: first, when estimating the $L^2$ distance between the
matrix of probabilities $\Q_n$ and the estimator $\widehat W$, one
encounters an oracle error for $\Q_n$ and a concentration error, the latter
being the second term in the above bound. Second, one encounters an
additional error when bounding the oracle error for $\Q_n$ in terms of the
oracle error for $W$. Since $\Q_n$ is random, this involves another
concentration error, which is the third term in the bound above.  Finally,
we need to estimate the $\delta_2$ distance between $W$ and $\frac
1\rho\Q_n$, which involves both bounding the distance between $W$ and $\frac
1\rho\min\{1,\rho W\}$, and the distance between $\min\{1,\rho W\}$ and
$\Q_n$.  It turns out that the latter error can be absorbed in the other
terms present above, while the former leads to the term
$\tail_{\rho}^{(2)}(W)$.

Note that the term $\sqrt[4]{\frac{1+\log (1/\kappa)}{\kappa^2\rho n}}$ in
the oracle inequality is larger than the next term $\sqrt[4]{\frac {\log
n}{\kappa n}}$ when $\rho \le 1/\log n$.  We have included both terms to
handle the case in which $\rho$ is large enough that the latter term
dominates, but $\sqrt[4]{\frac{1+\log (1/\kappa)}{\kappa^2\rho n}}$ should be
viewed as the primary term.

For general graphons, our results do not give explicit error bounds, since
all we know is that $\eps_{\geq\kappa}^{(2)}(W)$ and $\tail_{\rho}^{(2)}(W)$
go to $0$ as $\kappa\to 0$ and $\rho\to 0$. But in many applications, one
has additional information on the generating graphon, for example that it is
actually  a stochastic block model with a fixed number of classes, in which
case  both $\eps_{\geq\kappa}^{(2)}(W)$ and $\tail_{\rho}^{(2)}(W)$ become
identically zero once $\kappa$ and $\rho$ are small enough, leaving us only
with the explicit terms in the above bound.

Another class of examples is $\alpha$-H\"older-continuous graphons over
$\R^d$ equipped with a probability measure that decays fast enough to make
the function $|x|^\beta$ integrable. This class encompasses many models of
latent position spaces used in practice.  When $W$ is
$\alpha$-H\"older-continuous and $|x|^\beta$ is integrable with
$\alpha\in (0,1]$ and $\beta>2\alpha$,
we prove that $\eps_{\geq\kappa}^{(2)}(W)=O(\kappa^{\alpha'})$ and
$\tail_{\rho}^{(2)}(W)=O(\rho^{\beta'})$ for some $\alpha',\beta'>0$,
with $\alpha'=\alpha/d$ and $\beta'=\infty$ in the simple case of the
uniform distribution over a box of the form $[-R,R]^d$. See
Propositions~\ref{prop:Hoelder-compact} and \ref{prop:Hoelder} in
Section~\ref{sec:Hoelder} below.

This scaling behavior for the oracle error and tail bounds is typical. We
have stated the oracle inequality in full generality, but when the graphon is
sufficiently well behaved to estimate the oracle error and tail bounds, one
can balance the error terms and derive the scaling rate for $\kappa$ that
optimizes these bounds.  For example, suppose the error bound is
\[
O_p\paren{\kappa^{\alpha'}+\sqrt[4]{\frac{1+\log (1/\kappa)}{\kappa^2\rho n}}
  +\sqrt[4]{\frac {\log n}{\kappa n}}
  +\rho^{\beta'}}.
\]
Choosing $\kappa$ proportional to $\left(\frac {\log(\rho n)}{\rho
n}\right)^{\frac 1{4\alpha'+2}}$ optimizes this bound (assuming $n\rho \to
\infty$ as $n \to \infty$) and yields an error bound of
\[
O_p\paren{\left(\frac {\log(\rho n)}{\rho n}\right)^{\frac {\alpha'}{4\alpha'+2}}
  +\rho^{\beta'}},
\]
which becomes $O_p\left(\Big(\frac {\log(\rho n)}{\rho n}\Big)^{\frac
{\alpha}{4\alpha+2d}}\right)$ in the case of an $\alpha$-H\"older-continuous
graphon over $[-R,R]^d$ equipped with the uniform distribution.

\emph{Least cut norm algorithm.} To give an explicit description of the
least cut norm algorithm, we need the notion of the cut norm, first
introduced in \cite{FK99}.  For an $n\times n$ matrix $A$, it is defined as
\begin{equation}\label{M-cut-norm}
\|A\|_\square=\max_{S,T\subseteq [n]}\frac 1{n^2}\Bigl|\sum_{(i,j)\in S\times T} A_{ij}\Bigr|.
\end{equation}
One way to define the least cut norm algorithm would be to output a block
model defined in terms of the minimizer of $\|A(G)-B^\pi\|_\square$. But
since we now need to minimize the cut norm rather than an $L^2$ norm, this
would involve yet another optimization problem to find the best matrix $B$
for each distribution $\pi$.  To circumvent this issue, we always obtain $B$
by averaging. In other words, we calculate
\begin{equation}
\label{alg:cut}
\hat \pi\in \argmin_{\pi}\|A(G)-(A(G))_\pi\|_\square,
\end{equation}
where the $\argmin$ is again over partitions $\pi\colon [n]\to [k]$ such
that every non-empty partition class has size at least $\lfloor \kappa
n\rfloor$. The least cut norm algorithm then outputs the block average
corresponding to $\hat\pi$; i.e., it outputs the block model $\widehat
W=(\hat{\mathbf p},\hat B)$ where $\hat p_i$ is  again the relative size of
the $i^{\text{th}}$ partition class of $\hat\pi$ and $\hat B=A(G)/\hat\pi$.

We will show that the least cut norm algorithm is consistent in the cut
metric $\delta_\square$ on graphons, defined similar to $\delta_p$, except
that now we use the cut norm instead of the $L^p$ norm $\|\cdot\|_p$; see
\eqref{cut-distance} below for the precise definition. More precisely, we
will show that a.s., the error in the $\delta_\square$ distance goes to zero
for a $W$-random graph $G$ if $\kappa\to 0$ in such a way that
$\kappa^{-1}=o(\frac{\log n}n)$. In addition to consistency, we will again
show a quantitative bound, this time stating that for an arbitrary
normalized $L^1$ graphon $W$ and $\kappa\in (\frac{\log n}n,1]$,
\[
  \delta_\square\Bigl(\frac 1{\rho}\widehat W, W\Bigr)
  = O_p\paren{\eps_{\geq\kappa}^{(1)}(W)
  +\sqrt{\frac{1}{\rho n}}
  +\sqrt{\frac {\log n}{\kappa n } }+\tail_{\rho}^{(1)}(W)};
\]
see Theorem~\ref{thm:cut} in Section~\ref{sec:cut}. The four error terms
have the same explanation as the error terms for the least squares
algorithm: the oracle error for $W$, a concentration error appearing when
estimating the cut norm error with respect to $\Q_n$, a concentration error
stemming from the random nature of the oracle error for $\Q_n$, and a tail
bound stemming from the fact that for unbounded graphons, the matrix $\Q_n$
generating $G_n$ involves a truncation of the entries which are larger than
$1$. For H\"older-continuous graphons over $\R^d$ we can again give explicit
error bounds of the form $\eps_{\geq\kappa}^{(1)}(W)=O(\kappa^{\alpha'})$
and $\tail_{\rho}^{(1)}(W)=O(\rho^{{\beta'}})$; see
Propositions~\ref{prop:Hoelder-compact} and \ref{prop:Hoelder} in
Section~\ref{sec:Hoelder}.

\emph{Degree sorting algorithm.} The last algorithm we consider in this
paper is the degree sorting algorithm, which proceeds as follows.  Given a
degree $G$ on $n$ vertices with vertex degrees $d_1,\dots,d_n$, we sort the
vertices by choosing a permutation $\sigma$ of $[n]$ such that
\[
d_{\sigma(1)} \ge d_{\sigma(n)} \ge \dots \ge d_{\sigma(n)}.
\]
To separate the sorted vertices into $k$ classes of nearly equal size, we
choose integers $0 = n_0 < n_1 < \dots < n_k = n$ such that
\[
\left| n_i - \frac{in}{k}\right| < 1,
\]
and we define $\pi \colon [n] \to [k]$ by $\pi(j) = i$ if $n_{i-1} <
\sigma(j) \le n_i$.  Thus, $\pi$ groups the vertices into $k$ classes,
sorted by degree.  The output of the algorithm is the block model $\widehat
W=(\hat p,\hat B)$ with $\hat p_i=1/k$ and $\hat B=A(G)/\pi$.  In other
words, we simply cluster vertices with similar degrees and then average over
these clusters.

This algorithm has the advantage of being very efficient, but it has no hope
of working unless the degrees suffice to distinguish between the vertices.
More precisely, we need the limiting distribution of normalized degrees to
be atomless (i.e., there should not exist a nonzero fraction of the vertices
with nearly the same degree).  If $G$ is a $W$-random graph, then we can
express the limiting degree distribution as $n \to \infty$ in terms of $W$.
We do so in Section~\ref{sec:degree-distribution}.  If the degree
distribution of $W$ is atomless, then the degree sorting algorithm is
consistent in the sense that $\delta_1(\rho^{-1}\widehat W,W)\to 0$ almost
surely, provided that the number $k$ of classes tends to infinity in such a
way that $\log k = o(n\rho_n)$ and $k = o\big(n\sqrt{\rho_n}\big)$.  See
Theorem~\ref{thm:mon} for a precise statement.

\emph{Graphs with power-law degree distribution.} As an example of random
graphs which require unbounded graphons, we consider two simple models for
graphs with power-law degree distributions.  Both are generated by graphons
over $[0,1]$, with the first one given by $W(x,y)=\frac 12(g(x)+g(y))$,
where $g(x)=(1-\alpha) (1-x)^{-\alpha}$ for some $\alpha\in (0,1)$, and the
second one given by $W(x,y)=g(x)g(y)$. Both can be seen to have a degree
distribution with density function $f(\lambda) =
\Theta\big(\lambda^{-(1+1/\alpha)}\big)$, i.e., a power-law degree
distribution with exponent $1+\frac 1\alpha$. Both graphons are in $L^p$ as
long as $1\leq p<\frac 1\alpha$.

It turns out that the first graphon can be expressed as an equivalent
H\"older-continuous graphon over $\R^d$ equipped with a heavy-tailed
distribution, while this is not possible for the second; see
Section~\ref{sec:power-law} for details. But both fit into our general
theory, implying consistency for all three algorithms without any additional
work, and both allow for explicit bounds similar to the ones obtained for
H\"older-continuous graphons, even though only one of them can actually be
expressed as a H\"older-continuous graphon.  See Lemma~\ref{lem:Power-Block}
for the precise estimates.

\subsection{Comparison with related results}

As discussed above, our primary contribution in this paper is to analyze the
case of unbounded graphons, thus removing the restriction to networks in
which all the degrees are of the same order.  We also formulate our results
over general probability spaces, which increases their applicability.  (One
can always pass to an equivalent graphon over $[0,1]$, but standardizing the
underlying space prevents taking advantage of any smoothness or regularity
the graphon possesses, because these properties are not invariant under
equivalence.)

Least squares estimation is of course not a novel idea.  Motivated by results
of Choi and Wolfe \cite{CW14} on estimating block models, Wolfe and Olhede
proved consistency of least squares estimation for bounded graphons given
sparse graphs (in an
updated version of \cite{WO13} that has not yet, as of this writing, been
circulated publicly), under the additional hypotheses of H\"older continuity
and being bounded away from zero.  Borgs, Chayes, and Smith \cite{BCS15} and
Klopp, Tsybakov, and Verzelen \cite{klopp} proved consistency for bounded
graphons, again given sparse graphs, with no additional assumptions,
but they did not handle the
unbounded case. Our paper thus completes the analysis of this important
algorithm, by extending it to the full range of graphons that describe sparse
networks.

Bounded graphons are automatically square-integrable, but that is not
necessarily true for unbounded graphons.  Least squares estimation is an
appropriate technique only for $L^2$ graphons, and we propose least cut norm
estimation as a substitute that is applicable to arbitrary graphons.

Exact optimization is asymptotically inefficient for both the least squares
and the least cut norm algorithms.  Thus, our consistency results should be
viewed not as a proposal that exact optimization should be carried out in
practice for large networks, but rather as a benchmark for approximate or
heuristic optimization.

Degree sorting has the advantage of efficiency, although it works only for
graphons whose degrees are sufficiently well distributed.  The idea of
clustering vertices according to degree has a long history (see, for example,
\cite{DF89}), as well as connections with the theory of random graphs with a
given degree sequence \cite{Lauritzen2003,CDS11}. Degree sorting has recently
been analyzed as a graphon estimation algorithm by Chan and Airoldi
\cite{CA14}.  They showed that their sorting and smoothing algorithm is
consistent for dense graphs under a two-sided Lipschitz conditions on the
degrees of the underlying graphon. Our analysis accommodates sparse graphs
and even unbounded graphons, while avoiding these Lipschitz conditions.

\subsection{Organization}

The rest of the paper is organized as follows. Section~\ref{sec:prelim}
covers preliminary definitions and results, including equivalence and
identifiability, random graph convergence, and degree distributions.  Our
three main algorithms are analyzed in Sections~\ref{sec:L2}, \ref{sec:cut},
and \ref{sec:degreesort}.  Section~\ref{sec:Hoelder} examines how our bounds
behave given a greater degree of regularity than we assume elsewhere in the
paper (namely, H\"older continuity).  Finally, Section~\ref{sec:power-law}
analyzes two examples of graphons that yield power-law degree distributions.

\section{Preliminaries}

\label{sec:prelim}

\subsection{Notation}

As usual, we use $[n]$ to denote the set $\{1,\dots,n\}$.  The density of an
$n\times n$ matrix $H$  is defined as $\rho(H)=\frac
1{n^2}\sum_{i,j}H_{ij}$, and the density $\rho(G)$ of a graph $G$ is defined
as the density of its adjacency matrix.\footnote{Note that the density of a
simple graph is often defined as the number of non-zero entries in $A(G)$
divided by $\binom{n}{2}$; this definition is related to ours by a
multiplicative factor $\frac{n-1}n$, which becomes irrelevant as
$n\to\infty$.}  We use $\lambda$ to denote the standard Lebesgue measure on
$[0,1]$ (or, when we do not expect this to create confusion, for the
Lebesgue measure on $[0,1]^2$). We use $\Delta_k$ to denote the simplex of
probability measures on $[k]$, i.e., $\Delta_k=\{\mathbf p=(p_i)\in
\R^k_+\colon \sum_i p_i=1\}$.  The notation $O_p$ means big-$O$ in
probability: if $X$ and $Y$ are random variables, then $X = O_p(Y)$ means
for each $\eps>0$, there exists an $M$ such that $|X| \le M |Y|$ with
probability at least $1-\eps$.

Finally, we use the abbreviation a.s.\ for ``almost surely'' or ``almost
sure'' and i.i.d.\ for ``independent and
identically distributed.''

We will also consider general probability spaces $(\Omega,\cF,\pi)$, where
$\cF$ is a $\sigma$-algebra on $\Omega$ and $\pi$ is a probability measure
on $\Omega$ with respect to $\cF$. As usual, a map $\phi\colon
(\Omega,\cF,\pi) \to (\Omega',\cF',\pi')$ is called \emph{measure
preserving} if for all $F'\in\cF'$, $\phi^{-1}(F')\in\cF$ and
$\pi(\phi^{-1}(F'))=\pi'(F')$. We call such a map an \emph{isomorphism} if
it is a bijection and its inverse is measure preserving as well, and an
\emph{isomorphism modulo $0$} if, after removing sets of measure zero from
$\Omega$ and $\Omega'$, it becomes an isomorphism between the resulting
probability spaces.

In addition to the distance $\delta_p$, we also consider the (in general
larger) distance $\hat\delta_p(A,B)$ between two $n\times n$ matrices $A,B$,
defined as
\begin{equation}
\label{hat-delta-p}
\hat\delta_p(A,B)=\min_{\sigma}\|A^\sigma-B\|_p,
\end{equation}
where the minimum is over all bijections $\sigma\colon [n]\to [n]$, the
matrix $A^\sigma$ is defined by $(A^\sigma)_{ij}=A_{\sigma(i)\sigma(j)}$,
and the $L^p$ norm of an $n\times n$ matrix $A$ is defined by $
\|A\|_p^p=\frac 1{n^2}{\sum_{i,j\in [n]} |A_{ij}|^p}. $ Note that by
definition, $\hat\delta_p(A,B)$ is a distance invariant under relabeling;
i.e., it is a distance  on equivalence classes of $n\times n$ matrices with
respect to relabeling of the ``vertices'' in $[n]$.  We will need a similar
version of the cut distance $\|A-B\|_\square$.  It is defined as
\begin{equation}
\label{hat-delta-cut}
\hat\delta_\square(A,B)=\min_{\sigma}\|A^\sigma-B\|_\square,
\end{equation}
where the minimum is again over all bijections $\sigma\colon [n]\to [n]$ and
$\|\cdot\|_\square$ is defined in \eqref{M-cut-norm}.

Note also that the $L^2$ norm is related to a scalar product
$\langle\cdot,\cdot\rangle$ via $\|A\|_2^2=\langle A,A\rangle$, with the
scalar product between two $n\times n$ matrices $A,B$ defined as
\[
\langle A,B\rangle=\frac 1{n^2}\sum_{i,j\in [n]}A_{ij}B_{ij}.
\]

\subsection{Graphons and the cut metric}
\label{sec:graphons}

Given a probability space $(\Omega,\mathcal F,\pi)$, a measurable function
$W\colon \Omega\times\Omega\to\R$ is called symmetric if $W(x,y)=W(y,x)$ for
all $x,y\in \Omega$.  We call such a function a \emph{graphon} if it takes
non-negative values and $\|W\|_1<\infty$, where as usual, the $L^p$ norm of
a function $f\colon \Omega\times\Omega\to \R$ is defined by $
\|f\|_p^p=\int_{\Omega\times\Omega}|f(x,y)|^p\,d\pi(x)\,d\pi(y)$.  We call
$W$ an $L^p$ graphon if $\|W\|_p<\infty$, and we say that $W$ is
\emph{normalized} if $\|W\|_1=1$.

We will refer to $W$ as a graphon over $(\Omega,\mathcal F,\pi)$, or often
just as a graphon over $\Omega$ when the $\sigma$-algebra $\mathcal F$ and
the probability measure $\pi$ are clear from the context. For example, when
we say that $W$ is a graphon over $[0,1]$, we  mean that $W$ is a graphon
over $[0,1]$ equipped with the Borel $\sigma$-algebra and the uniform
measure, unless stated otherwise. Note that graphs are special cases of
graphons: given a graph with vertex set $V$ and adjacency matrix $A$, we
view it as a graphon on $V$ by equipping $V$ with the uniform distribution and
choosing $W(u,v)$ to be $A_{uv}$.

In addition to the $L^p$ norm of a graphon $W$, we will also use the cut
norm $\|W\|_\square$, defined as
\[
\|W\|_\square =\sup_{S,T\subseteq \Omega}\Bigl|\int_{S\times T}W(x,y)\,d\pi(x)\,d\pi(y)\Bigr|,
\]
where the supremum is over measurable subsets of $\Omega$ (i.e., elements of
$\mathcal F$).  The corresponding metric is defined
 for a pair of graphons  $W$ and $W'$
on two probability spaces $(\Omega,\cF,\pi)$ and $(\Omega',\cF',\pi')$ by
\begin{equation}
\label{cut-distance}
\delta_\square(W,W')=
\inf_\nu\sup_{S,T\subseteq \Omega\times\Omega'}
\Bigl|\int_{S\times T}\Bigl(W(x,y)-W'(x',y')\Bigr)\,d\nu(x,x')\,d\nu(y,y')\Bigr|,
\end{equation}
where the infimum is over couplings $\nu$ of the two measures $\pi$ and
$\pi'$ and the supremum is over measurable subsets of $\Omega\times\Omega'$.
Because graphs are special cases of graphons, this in particular defines a
distance between a graph and an arbitrary graphon.

\begin{rem}
\label{rem:delta-over-01}  We will often consider graphons over $[0,1]$
(with the Borel $\sigma$-algebra unless otherwise specified).  For such
graphons, both the cut distance $\delta_\square$ and the $L^p$ distance
$\delta_p$ can be defined in a simpler way.  Specifically,
\begin{equation}
\label{delta-over-01}
\delta_p(W,W')
=\inf_{\Phi}\|W^\Phi-W'\|_p
\qquad\text{and}\qquad
\delta_\square(W,W')
=\inf_{\Phi}\|W^\phi-W'\|_\square,
\end{equation}
where the infima over $\Phi$ are over isomorphisms from $[0,1]$ to itself.
In fact, this simpler definition is equivalent to the definitions
\eqref{del-p} and \eqref{cut-distance} for many spaces used in practice, as
long as they are atomless; see Lemma~\ref{lem:delta-over-01} in
Appendix~\ref{sec:couplings} for the precise setting.
Lemma~\ref{lem:delta-over-01} also shows that for many spaces of interest,
in particular both the unit interval $[0,1]$ with the uniform distribution
and any finite probability space, the infima in the expressions
\eqref{del-p} and \eqref{cut-distance} are actually minima.
\end{rem}

\subsection{$W$-random graphs}
Given a normalized graphon $W$ and a \emph{target density} $\rho$, we define
two random graphs $\Q_n(\rho W)$ and $G_n(\rho W)$ on $[n]$ as follows.
First, we choose i.i.d.\ elements $x_1,\dots,x_n$ from the probability space
$(\Omega,\mathcal F,\pi)$; these elements will index the vertices of the
graphs.  Let $\Q_n=\Q_n(\rho W)$ be the $n\times n$ matrix whose $ij$ entry
is equal to $\min\{1,\rho W(x_i,x_j)\}$ if $i\neq j$ and $0$ if $i=j$.  We
view $\Q_n$ as a weighted graph on $n$ vertices, and we define a
corresponding unweighted graph $G_n$ by including the edge between vertices
$i$ and $j$ with probability $(\Q_n)_{ij}$ (independently for each $i$ and
$j$).  We call $\Q_n$ a \emph{weighted $W$-random graph} at target density
$\rho$, and $G_n$ a \emph{$W$-random graph} at target density $\rho$.

In addition to the graph $G_n(\rho W)$ and the weighted graph $\Q_n(\rho
W)$, we will sometimes also consider the weighted graph $\HH_n(W)$, defined
as weighted graph with entries $(\HH_n(W))_{ij}=W(x_i,x_j)$ for $i\neq j$,
and $(\HH_n(W))_{ii}=0$; in contrast to the definitions of $G_n(\rho W)$ and
$\Q_n(\rho W)$, which we will only use for graphons, the latter notation
will be used even if $W$ takes values in $\R$, rather than in $[0,\infty)$.

Since $G_n$, $\Q_n$, and $H_n$ are trivial for $n=1$, we will always assume
that $n\geq 2$ without explicitly stating this.

\begin{rem}  The expected densities of the graphs $\Q_n$
and $G_n$ are $\|\min\{1,\rho W\}\|_1$, which is $\rho$ when $\rho W$ is
bounded above by $1$ and $W$ is normalized, and which is $(1+o(1))\rho$ if
$\rho=\rho_n\to 0$ as $n\to\infty$. That is why we call $\rho$ the target
density for $\Q_n$ and $G_n$.
\end{rem}

Note that many models of random graphs can be written as $W$-random graphs.

\begin{ex}[Stochastic block model on $k$ blocks]\label{ex:SBM}
Let $\Omega=[k]$, and let the probability distribution $\pi$ on $\Omega$ be
given by a vector $\mathbf p=(p_1,\dots,p_k)\in \Delta_k$.  Setting
$W(i,j)=\beta_{ij}$ for some symmetric matrix $B=(\beta_{ij})$ of
non-negative numbers then describes the standard stochastic block
model.\footnote{We will not restrict the entries to be bounded by $1$, since
we want to consider normalized graphons, which become trivial if all entries
are bounded by $1$.} with parameters $(\mathbf p,B)$  We denote the set of
all block models on $k$ blocks by $\cB_k$ and use $\cB$ to denote the union
$\cB=\bigcup_{k\geq 1}\cB_k$.  For $\kappa\in (0,1/2]$, we use
$\cB_{\geq\kappa}$ to denote all block models $(\mathbf p,B)$ such that
$p_i\geq\kappa$ for all $i$.

Alternatively, we can use the uniform distribution over the interval $[0,1]$
as our probability space.  Then we define $\widetilde W$ by first
partitioning $[0,1]$ into $k$ adjacent intervals of lengths $p_1,\dots,p_k$,
and then setting $\widetilde W$ equal to $\beta_{ij}$ on $I_i\times I_j$.
Note that the random graphs generated by $W$ and $\widetilde W$ are equal in
distribution. We denote the graphon $\widetilde W$ by $\WW{\mathbf p,B}$, or
by $\WW{B}$ if all the probabilities $p_i$ are equal.  (We will also
sometimes abuse notation by identifying it with $W$, when this does not seem
likely to cause confusion.)
\end{ex}

\begin{ex}[Mixed membership stochastic block model]
To express the mixed membership block model of \cite{MMSB08} as a $W$-random
graph, we define $\Omega$ to be the $k$ dimensional simplex $\Delta_k$ and
equip it with a Dirichlet distribution with some parameters
$\mathbf\alpha=(\alpha_1,\dots,\alpha_k)$.  In other words, the probability
density at $(p_1,\dots,p_k)$ is proportional to
\[
\prod_i p_i^{\alpha_i-1}.
\]
Given a symmetric matrix $(\beta_{ij})$ of non-negative numbers, we then
define
\[
W(\mathbf p,\mathbf p')=\sum_{i,j}\beta_{ij}p_ip'_j.
\]
As in the stochastic block model, $\beta_{ij}$ describes the affinity
between  communities $i$ and $j$, but now each vertex is assigned a
probability distribution $\bf p$ over the set of communities (rather than
being assigned a single community).
\end{ex}

\subsection{Equivalence and identifiability}\label{sec:identify}

In this section,
we determine when two different graphons lead to sequences of random graphs
that are indistinguishable, in the sense that they are equal in
distribution. As we will see, this is the case if and only if the two
graphons are equivalent according to the following definition.

\begin{definition}\label{def:equiv}
Let $W$ and $W'$ be graphons over $(\Omega,\cF,\pi)$ and
$(\Omega',\cF',\pi')$. We call $W$ and $W'$  \emph{equivalent}\footnote{Our
notion of equivalence is closely related to the notion of ``weak
isomorphism'' from \cite{BCL10}, the only difference being that in
\cite{BCL10} the maps $\phi$ and $\phi'$ were required to be measure
preserving with respect to the completion of the spaces $(\Omega,\cF,\pi)$
and $(\Omega',\cF',\pi')$.  We will not use the term weak isomorphism since
we want to avoid the impression that it implies that the underlying
probability spaces are isomorphic after removing suitable sets of measure
$0$.  It does not; see Example~\ref{ex:equiv-46} for two equivalent graphons
on non-isomorphic probability spaces.} if there exist two measure-preserving
maps $\phi$ and $\phi'$ from $(\Omega,\cF,\pi)$ and $(\Omega',\cF',\pi')$ to
a third probability space $(\Omega'',\cF'',\pi'')$ and a graphon $U$ on
$(\Omega'',\cF'',\pi'')$, such that $W=U^\phi$ and $W'=U^{\phi'}$ almost
everywhere. We call $W$ and $W'$ \emph{isomorphic modulo $0$} if there
exists a map $\phi\colon \Omega\to\Omega'$ such that $\phi$ is an
isomorphism modulo $0$ and $W=(W')^\phi$ almost everywhere.
\end{definition}

\begin{thm}\label{thm:equiv}
Let $W$ and $W'$ be graphons over $(\Omega,\cF,\pi)$ and
$(\Omega',\cF',\pi')$, respectively. Assume that $n\rho_n\to\infty$ and that
either $\rho_n\max\{\|W\|_\infty,\|W'\|_\infty\}\leq 1$ or $\rho_n\to 0$.
Then the sequences $(G_n(\rho_nW))_{n \ge 0}$ and $(G_n(\rho_nW'))_{n \ge
0}$ are identically distributed if and only if $W$ and $W'$ are equivalent.
\end{thm}

For the dense case and bounded graphons, this follows from the results of
\cite{BCL10} and \cite{BCLSV08} (or from those of \cite{DJ08}, provided we
only consider graphons over $[0,1]$). The sparse case and general (possibly
unbounded) graphons is new, and relies on the theory of graph convergence
for $L^p$ graphons. We prove it in the next section.

The two representations $W$ and $\widetilde W$ for the stochastic block
model from Example~\ref{ex:SBM} are clearly equivalent in the sense of
Definition~\ref{def:equiv}.  In this case, we actually have $\widetilde
W=W^\phi$ for a measure-preserving map $\phi\colon [0,1]\to [k]$ (namely the
map which maps all points in the interval $I_i$ to $i\in [k]$). But in
general, equivalence does not imply the existence of a measure-preserving
map $\phi$ such that $\widetilde W=W^\phi$ or $\widetilde W^\phi=W$.  This
is the content of the next example.

\begin{ex}\label{ex:equiv-46}
Let $\Omega=[4]$ and $\Omega'=[6]$, both equipped with the uniform
distribution. Define  $W$ and $W'$ to be zero if both arguments are even or
both arguments are odd, and set both of them to a constant $p$ otherwise. It
is easy to see that they are equivalent: indeed, let $\Omega''=\{1,2\}$ and
define $\phi\colon [4]\to [2]$ and $\psi\colon [6]\to [2]$ by mapping even
elements to $2$ and odd elements to $1$.  Setting $U$ to $1$ if its two
arguments are different and to $0$ otherwise, we see that $W=U^\phi$ and
$W'=U^\psi$. This shows that in general, we cannot restrict ourselves to a
single, measure-preserving map $\phi\colon \Omega\to\Omega'$, since there is
simply no measure-preserving map between $\Omega$ and $\Omega'$.

But even if both probability spaces are $[0,1]$ equipped with the uniform
measure (in which case there are many measure-preserving maps between the
two), we can in general not find a measure-preserving map such that
$W'=W^\phi$ or the other way around. To see this, let $\phi_k(x)=kx \bmod 1$,
define $W_1(x,y)=xy$, and let $W_k=W_1^{\phi_k}$.  Then there is no
measure-preserving transformation $\phi\colon [0,1]\to [0,1]$ such that
$W_2^\phi=W_{3}$ or $W_3^\phi=W_{2}$; see Example~8.2 in \cite{J13} for the
proof.
\end{ex}

There is however, a special case where it is possible to just use a single
map, namely the case where $W$ and $W'$ are  twin-free Borel graphons. Here
a graphon is called a \emph{Borel graphon} if the underlying probability
space is a \emph{Borel space}, i.e., a space that is isomorphic to a Borel
subset of a complete separable metric space equipped with an arbitrary
probability measure with respect to the Borel $\sigma$-algebra. A graphon
$W$ is called \emph{twin-free} if the set of twins of $W$ has measure zero,
where a \emph{twin} is a point $x$ in the underlying probability space for
which there is another point $y$ such that $W(x,\cdot)$ is equal to
$W(y,\cdot)$ almost everywhere. Note that in Example~\ref{ex:equiv-46}
above, the graphons $U$ and $W_1$ are twin-free, while $W$, $W'$, and
$W_k$ for $k\geq 2$ are not.

\begin{thm}\label{thm:twinfree-equiv}
Let $W$ and $W'$ be twin-free Borel graphons. Then $W$ and $W'$ are
equivalent if and only if they are isomorphic modulo $0$.
\end{thm}

The theorem can easily be deduced from the results of \cite{BCL10}, and is
proved in Appendix~\ref{sec:couplings}.

To state our next theorem, we define a \emph{standard Borel
graphon}\footnote{Note that some authors use the notion of standard graphons
or standard kernels for graphons with values in $[0,1]$; here we don't
require such a condition.} as a graphon over a probability space that is the
disjoint union of an interval $[0,p]$ equipped with the uniform distribution
and the usual Borel $\sigma$-algebra, plus a countable number of isolated
points $\{x_j\}_{j\in J}$ with non-zero mass $p_j$ for each of them,
allowing for the special cases where either the set of atoms or the interval
$[0,p]$ is absent.  The former is the case of graphons over $[0,1]$, while
the latter is the case of block models over $[k]$ equipped with a
probability measure in $\Delta_k$.

\begin{thm}\label{thm:equiv-over-01}
Let $W$ be a graphon over an arbitrary probability space $(\Omega,\cF,\pi)$.

(i) There exists an equivalent graphon over $[0,1]$ equipped with the
uniform distribution.

(ii) There exists a twin-free standard Borel graphon $U$ and a
measure-preserving map $\phi$ from $(\Omega,\cF,\pi)$ to the space on which
$U$ is defined such that $W=U^\phi$ almost everywhere, showing in
particular that $W$ is equivalent to a twin-free standard Borel graphon.
\end{thm}

Again, the theorem follows easily from the results of \cite{BCL10}; see
Appendix~\ref{sec:couplings}.

\begin{rem}\label{rem:delta-p}
(i) The above theorem states that for any graphon $W$, we can find both an
equivalent graphon $U$ over $[0,1]$ and an equivalent twin-free standard
Borel graphon $\tilde U$.  But in general, it is not possible to find a
single equivalent graphon $U$ which is both twin-free and a graphon over
$[0,1]$, as the example of a block model shows, since any representation of
it over $[0,1]$ has uncountably many twins.

(ii) As claimed in the introduction, the metric \eqref{del-p} is indeed a
distance on equivalence classes; in other words, $\delta_p(W,W')=0$ if $W$
and $W'$ are equivalent. To see this, let $W$, $W'$, $\phi$, $\phi'$ and $U$
be as in Definition~\ref{def:equiv}. Define a coupling $d\nu(x,x'')$ of
$\pi''$ and $\pi$ by choosing $x\in\Omega$ according to $\pi$ and then
setting $x''=\phi(x)$.  Using this coupling, it is easy to see that
$\delta_p(U,W)=0$.  Similarly, $\delta_p(U,W')=0$, which together with the
triangle inequality proves the claim.

(iii) When comparing finite graphs to graphons over $[0,1]$, we will
sometimes use a stronger version of the $\delta_p$ distance.  This distance
extends the definition \eqref{hat-delta-p} to a distance between an $n\times
n$ matrix $A$ and a graphon $W$ over $[0,1]$, defined by
\begin{equation}
\label{hat-delta-pW}
\hat\delta_p(A,W)=\min_{\sigma}\|\WW{A^\sigma}-W\|_p,
\end{equation}
where, again, the minimum is over all bijections $\sigma\colon [n]\to [n]$
and $(A^\sigma)_{ij}=A_{\sigma(i)\sigma(j)}$, and where $\WW{\cdot}$ is
defined in Example~\ref{ex:SBM}.
\end{rem}

We close this section with a theorem giving a different characterization of
equivalence in terms of the metrics $\delta_p$ and $\delta_\square$.

\begin{thm}
\label{thm:equiv2} Let $p\geq 1$, and let $W$ and $W'$ be $L^p$ graphons
over two arbitrary probability spaces.  Then the following statements are
equivalent:

(i) $\delta_\square(W,W')=0$;

(ii) $\delta_p(W,W')=0$;

(iii) $W$ and $W'$ are equivalent.
\end{thm}

The theorem follows again from the results of \cite{BCL10}, even though the
details are a little more involved than for the previous two theorems and in
particular make use of the fact that the infimum in \eqref{cut-distance} is
actually a minimum if the underlying space is the unit interval; see
Appendix~\ref{sec:couplings} for the proof.

\subsection{Relation to graph convergence}

As mentioned before, $W$-random graphs arise very naturally as
non-parametric models when considering a given graph as a finite subgraph of
an infinite, exchangeable array, at least in the dense setting.  Indeed, as
the works of Hoover \cite{H79} and Aldous \cite{A81} show, any graph which
is an induced subgraph of an infinite, exchangeable array can be modelled as
a $W$-random graph\footnote{Strictly speaking, the results of \cite{H79,A81}
only imply that the \emph{extremal components} of a infinite, exchangeable
random graph are given by a graphon; see \cite{DJ08} for a review of this
connection. But if we are given only one sample, the difference between an
exchangeable random graph and an ergodic component is unobservable, since by
the results of \cite{Var63}, a single observation of an exchangeable random
graph only reveals one of the ergodic components, just like a single
observation of an infinite set of coin-flips from an exchangeable sequence
looks like a sequence of independent coin flips.} for some graphon $W$.

A different window into the theory of $W$-random graphs is given by the
theory of graph convergence. Here one asks when a sequence of graphs $G_n$
should be considered convergent.  Motivated by extremal combinatorics, one
way to address this question is to define a sequence of graphs to be
convergent if the number of subgraphs isomorphic to a given graph $H$
converges for every finite graph $H$, once suitably normalized.  It turns
out that in the dense setting, this notion is equivalent to many other
natural notions of graph convergence that are relevant in computer science,
statistical physics, and other fields \cite{BCLSV06,BCLSV08,BCLSV12}.

One of these equivalent notions is convergence in metric, defined in terms
of the cut metric \eqref{cut-distance}. We say that a sequence of dense
graphs \emph{converges to a graphon $W$ in metric} if
$\delta_\square(G_n,W)\to 0$ as $n\to \infty$.  Note that the limit $W$ is
not unique, since two graphons $W$ and $W'$ which are equivalent have
distance $\delta_\square(W,W')\leq\delta_1(W,W')=0$.  The results of
\cite{BCL10} imply that this is the only ambiguity: if $W$ and $W'$ are such
that $\delta_\square(G_n,W)\to 0$ and $\delta_\square(G_n,W')\to 0$, then
$W$ and $W'$ are equivalent.

Given this notion of convergence, one may  want to ask whether all sequences
of graph $G_n$ have a limit, or whether they at least have a subsequence
which converges in the metric $\delta_\square$. For dense graphs, the answer
to this question is yes and was given in \cite{LS06}, where it was shown
that every sequence  of dense graphs has a subsequence that is a Cauchy
sequence in the metric $\delta_\square$, and that every Cauchy sequence
converges to a graphon $W$ over $[0,1]$.

Thus the results of \cite{LS06} completely parallel the results on
exchangeable arrays of \cite{H79,A81}: given an ergodic component of an
infinite, exchangeable graph, one can find a graphon over $[0,1]$ that
generates this array, and given an arbitrary sequence of (random or
non-random) dense graphs, one can find a subsequence and a graphon over
$[0,1]$ such that the subsequence converges to that graphon.  In both cases,
the graphon is identifiable only up to equivalence.  Finally, combining
\cite{LS06} with  \cite{BCL10}, we know that if the sequence of graphs
happens to be a sequence of $W$-random graphs, then it converges a.s., and
the generating graphon is a representative from the equivalence class of
limits.

The net result of this theory is that a convergent sequence of dense networks
behaves like a sequence of $W$-random graphs for some graphon $W$ and can
thus be viewed as $W$-quasi-random graphs. Having established this connection
between $W$-random graphs and $W$-quasi-random graphs in the dense setting,
one might ask whether it can be extended to a convergence theory for sparse
graph sequences. It is clear that we cannot just simply consider Cauchy
sequences in the cut metric $\delta_\square$, since all sequences of sparse
graphs have this property. Indeed, by the triangle inequality
\[
\delta_\square(G_n,G_m)\leq \delta_1(G_n,G_m)\leq \delta_1(G_n,0)+
\delta_1(G_m,0)=\rho(G_n)+\rho(G_m).
\]
But if instead of the graphon given by the adjacency matrix  of $G_n$ we
consider the normalized adjacency matrix $\frac 1{\rho(A(G_n))} A(G_n)$,
this argument no longer holds.

This motivates the following definition.  To state it, we define, for an
arbitrary graph $G$ with adjacency matrix $A(G)$ and a constant $c\in\R$,
the graph $cG$ to be the weighted graph with adjacency matrix $cA(G)$.

\begin{definition}
Let $W$ be a graphon over an arbitrary probability space.  A sequence of
graphs $G_n$ \emph{converges to $W$ in metric} if
\[
\delta_\square\Bigl(\frac 1{\rho(G_n)} G_n,W\Bigr)\to 0\qquad\text{as $n\to\infty$.}
\]
In this case, we call $G_n$ a \emph{$W$-quasi-random sequence with target
density $\rho\|W\|_1$.}
\end{definition}

\begin{rem}
This definition is an extension of the one given in \cite{BCCZ14a} for
graphons $W$ over $[0,1]$.  There, as in the earlier literature on graph
convergence for dense graphs, the distance between a graph $G$ and a graphon
$W$ was defined as the distance between $W$ and the embedding $\WW{G}$ of
$G$ into the space of graphons over $[0,1]$, i.e.,
$\delta_\square(G,W)=\delta_\square(\WW{G},W)$, with $\WW{G}$ defined as in
Example~\ref{ex:SBM}, by setting $\WW{G}$ to $A_{ij}(G)$ on $I_i\times I_j$,
where $I_1,\dots,I_n$ is a partition of $[0,1]$ into adjacent intervals of
lengths $1/n$. In our setting, this embedding is not needed, since the cut
distance \eqref{cut-distance} is defined on equivalence classes of graphons,
and $G$ and its embedding $\WW{G}$ are equivalent.
\end{rem}

Given the above definition of convergence for sparse graphs, one might ask
whether this notion is again equivalent to other notions of convergence, and
whether sparse $W$-random graphs converge again to the generating graphon.
The answer to both questions is yes, with one exception: convergence of
subgraph counts is no longer equivalent to convergence in
metric.\footnote{Indeed, it is possible to modify a sparse graph sequence by
very little while greatly changing its subgraph counts: a $W$-random graph
with sufficiently low target density will have far fewer triangles than
edges, so one can eliminate triangles without making any substantial change
in the cut metric. For details, see the discussion after Theorem~2.18 in
\cite{BCCZ14a}.} But all other notions of convergence proved to be
equivalent for dense graphs in \cite{BCLSV12} remain equivalent in the
sparse setting, as shown in \cite{BCCZ14b}. It is also again true that a
sequence of $W$-random graphs  converges to the generating graphon.
This is the content of the following theorem.

\begin{thm}
\label{thm:GW-conv} Let $G_n=G_n(\rho_nW)$ where $W$ is a normalized graphon
over an arbitrary probability space, and $\rho_n$ is such that
$n\rho_n\to\infty$ and either $\limsup\rho_n\|W\|_\infty\leq 1$ or
$\rho_n\to 0$. Then a.s.\ $\rho(G_n)/\rho_n\to 1$ and
\[
\delta_\square\Bigl(\frac 1{\rho(G_n)}G_n,W\Bigr)\to 0.
\]
\end{thm}

\begin{proof}
By Theorem~\ref{thm:equiv-over-01}, we can find a graphon $W'$ over $[0,1]$
that is equivalent to $W$. Since equivalent graphons lead to identically
distributed random graphs, it is enough to prove the theorem for graphons
over $[0,1]$.  But for this case, it has been established in \cite{BCCZ14a}.
\end{proof}

\begin{rem}
The above theorem has many interesting consequence for graphon estimation.
In particular, assume that an algorithm releases an estimator $\widehat W$
for the generating graphon $W$ which is close in  $\delta_p$ for $p\geq 1$.
These distances dominate the invariant $L^1$ distance $\delta_1$, which in
turn dominates the cut distance $\delta_\square$.  Combined with the results
from \cite{BCCZ14b} which state that many other notions of convergence are
equivalent to convergence in metric (see Theorem 2.10), we obtain that
consistent approximation for $W$ leads to consistent approximations for
various quantities of interest, such as minimal energies of graphical models
defined on $G_n$ (see Proposition~5.12 in \cite{BCCZ14b}, which actually
gives a quantitative bound in terms of the cut distance) or collections of
cuts in $G_n$ (see Lemma~5.11 in \cite{BCCZ14b}, which again gives a
quantitative bound). By Theorem~\ref{thm:degree-convergence} below, we also
get good approximations for the empirical distributions of the degrees of
$G_n$.
\end{rem}

Combined with Theorem~\ref{thm:equiv2}, Theorem~\ref{thm:GW-conv}
immediately implies Theorem~\ref{thm:equiv}.

\begin{proof}[Proof of Theorem~\ref{thm:equiv}]
Let $G_n=G_n(\rho_n W)$ and $G_n'=G_n(\rho_n W')$. Since
$\delta_\square(\frac 1{\rho_n}G_n,W)\to 0$ and $\delta_\square(\frac
1{\rho_n}G_n',W')\to 0$ by Theorem~\ref{thm:GW-conv}, we have
$\delta_\square(W,W')=0$ if $G_n$ and $G_n'$ are identically distributed.
Since, on the other hand, $G_n$ and $G_n'$ are clearly identically
distributed if $W$ and $W'$ are equivalent, Theorem~\ref{thm:equiv} follows
from Theorem~\ref{thm:equiv2}.
\end{proof}

\subsection{Convergence of degree distributions}
\label{sec:degree-distribution}

In this subsection we show that convergence in the cut metric
$\delta_\square$ implies convergence of the empirical degree distributions.
We define the \emph{normalized degree} of a vertex $x\in V(G)$ as $d_x/\bar
d$, where $d_x$ is its degree and $\bar d$ is the average degree
\[
\bar d=\frac 1{|V(G)|}\sum_{x\in
V(G)}d_x=\frac{2|E(G)|}{|V(G)|}.
\]
The \emph{normalized degree distribution} of $G$ is the empirical
distribution of the normalized degrees, with cumulative distribution
function
\[
D_G(\lambda)=
\frac 1{|V(G)|}\sum_{x\in V(G)}1_{d_x\leq \lambda\bar d}.
\]
In a similar way, we define the \emph{degrees} of a normalized graphon $W
\colon \Omega \times \Omega \to [0,\infty)$ as the random variable
\[
W_x = \int_\Omega W(x,y) \, d\pi(y),
\]
where $x$ is chosen according to the probability measure $\pi$ on $\Omega$.
This random variable has cumulative distribution function
\[
D_W(\lambda)=\pi(\{x : W_x\leq\lambda\}).
\]
Recalling that convergence in distribution can be formulated as convergence
in the L\'evy-Prokhorov distance, we say that the normalized degree
distributions of a sequence $G_n$ of graphs converge to the degree
distribution of $W$ if $d_{\LP}(D_{G_n},D_W)\to 0$, where as usual, the
L\'evy-Prokhorov distance $d_{\LP}$ between two distribution functions $D$
and $D'$ is defined by
\[
d_{\LP}(D,D')=\inf\{\eps>0:
D'(\lambda-\eps)-\eps\leq D(\lambda)\leq D'(\lambda+\eps)+\eps \text{ for all $\lambda\in \R$}\}.
\]
Our next theorem implies that convergence in the cut metric implies
convergence of the normalized  degree distributions. Combined with
Theorem~\ref{thm:GW-conv}, this gives that a.s., the normalized degree
distributions of a sequence of $W$-random graphs converge to the degree
distribution of $W$ as long as $n\rho_n\to\infty$ and $\rho_n\to 0$. Indeed,
observing that for any graph $G$, the normalized degree distribution $D_G$
is equal to the degree distribution of $\frac 1{\|A(G)\|_1}A(G)$ considered
as a graphon over $V(G)$ equipped with the uniform distribution, both
statements follow immediately from the following theorem.

\begin{thm}\label{thm:degree-convergence}
Let $U$ and $W$ be two normalized graphons.  Then
\[
d_{\LP}(D_U,D_W)\leq \sqrt{2\delta_\square(U,W)}.
\]
\end{thm}

The proof will make use of the following lemma.

\begin{lem}\label{lem:degree-match}
Let $U$ and $W$ be two normalized graphons over the same probability space
$\Omega$. If $x$ is chosen at random from $\Omega$, then
\[
\Pr(|W_x-U_x|\geq\eps) \le \frac 2\eps \|U-W\|_\square.
\]
\end{lem}

\begin{proof}
We have
\begin{align*}
\Pr(|W_x-U_x|\geq\eps)&\leq
\frac 1\eps \E[|W_x-U_x|]\\
&=\frac 1\eps\E[(W_x-U_x)1_{W_x\geq U_x}]+\frac 1\eps\E[(U_x-W_x)1_{W_x\leq U_x}].
\end{align*}
Defining $S$ as the set of points $x\in \Omega$ such that $W_x\geq U_x$ and
$\tilde S$ as the set of points $x\in \Omega$ such that $W_x\leq U_x$, we
write the right side as
\[\frac 1\eps\int_{[0,1]\times S}(W-U)+\frac 1\eps\int_{[0,1]\times\tilde S}(U-W)
\leq \frac 2\eps \|U-W\|_\square,
\]
as desired.
\end{proof}

\begin{proof}[Proof of Theorem~\ref{thm:degree-convergence}]
To prove the theorem, we will prove that for two arbitrary graphons and all
$\lambda\in \R$ and $\eps>0$,
\begin{equation}
\label{d-UW-bd}
D_W(\lambda)\leq D_U(\lambda+\eps) +2\frac{\delta_\square(U,W)}\eps.
\end{equation}
Because the degree distributions of equivalent  graphons are identical, it
will be enough to prove \eqref{d-UW-bd} for two graphons over $[0,1]$, with
an upper bound of $\|U-W\|_\square$ instead of $\delta_\square(U,W)$.

To this end,  we estimate the probability that  $U_x$ and $W_x$ differ by at
least $\eps$ by Lemma~\ref{lem:degree-match}. As a consequence,
\begin{align*}
D_W(\lambda)&=\Pr[W_x\leq\lambda]\\
&\leq \Pr[U_x\leq\lambda+\eps]+\Pr[|U_x-W_x|\geq\eps]\\
&\leq D_U(\lambda+\eps)+\frac 2\eps\|U-W\|_\square,
\end{align*}
which proves \eqref{d-UW-bd} and hence the theorem.
\end{proof}

\subsection{Existence of approximating block models}

Having seen that block models are special cases of $W$-random graphs, one
might wonder how well an arbitrary graphon can be approximated by a
stochastic block model.  The answer is given by the following lemma.  To
state it, we recall the definition of  $\cB_{\geq\kappa}$ as the set of all
block models with minimal block size at least $\kappa$,
$\cB_{\geq\kappa}=\{(\mathbf p, B)\in \cB\colon \min_i p_i\geq\kappa\}$.

\begin{lem}\label{lem:block-model-approx}
Let $1\leq p<\infty$, let $W$ be an $L^p$ graphon, and let
$\eps_{\geq\kappa}^{(p)}(W)$ be as in \eqref{eps_kappa}. Then the infimum in
\eqref{eps_kappa} is achieved for some $W'\in\cB_{\geq \kappa}$ that has
norm $\|W'\|_p\leq 2\|W\|_p$.  Furthermore, $\eps_{\geq\kappa}^{(p)}(W)\to
0$ as $\kappa\to 0$.
\end{lem}

\begin{proof}
We clearly have $\eps_{\geq\kappa}^{(p)}(W)=\inf_{W'}\delta_p(W,W')\leq
\|W\|_p$, so by the triangle inequality, we only need to consider block
models $W'$ with $\|W'\|_p\leq 2\|W\|_p$.  Again by the triangle inequality,
the distance $\delta_p(W,W')$ is continuous in $W'$, which implies that the
infimum is actually a minimum.

To see that $\eps_{\geq\kappa}^{(p)}(W)\to 0$ as $\kappa\to 0$, we first
replace $W$ by an equivalent graphon $U$ over $[0,1]$, and then use the
approximation $U_n$ to $U$ given by averaging over the partition consisting
of consecutive intervals of length $1/n$. This approximation is a block
model with minimal block size $1/n$, and it converges to $U$ by the Lebesgue
differentiation theorem and a truncation argument (see Lemma~5.6 in
\cite{BCCZ14b}).
\end{proof}

When applying the lemma, we will sometimes be constrained to use only block
models whose block sizes are all a multiple of $1/n$, i.e., block models in
\[
\cB_{n,\geq\kappa}
=\{(\mathbf p,B)\in\cB\colon \text{for all $i$, $p_i n\in \Z$ and $p_i n\geq \lfloor n\kappa\rfloor$}\}.
\]
Note that $\cB_{n,\geq\kappa}$ naturally corresponds to the set
$\cA_{n,\geq\kappa}$ of $n\times n$ block matrices $A$ such that each block
in $A$ has size at least $\lfloor n\kappa\rfloor$, via
\begin{equation}
\label{k-block-matrix}
\{ \WW{A} : A \in \cA_{n,\geq\kappa}\} = \{\WW{\mathbf p, B} : (\mathbf p, B) \in \cB_{n,\geq\kappa}\}.
\end{equation}
Our next lemma shows that every block model in $\cB_{\geq\kappa}$ can be
well approximated by a block model in $\cB_{n,\geq\kappa}$, and it also
shows that $\eps_{\geq\kappa}^{(p)}$ can be bounded from above in terms of a
minimum over $\cB_{n,\geq\kappa}$.  It is proved in
Appendix~\ref{app:Auxiliary}.

\begin{lem}\label{lem:block-model-approx-n}
Let $\kappa\in (0,1]$. Then there exists a constant $n_0(\kappa)$ such for
all $p\geq 1$ and all $L^p$ graphons $W$, the following holds:

If $W'\in \cB_{\geq\kappa}$ is a block model on $[k]$, then the labels in
$[k]$ can be reordered in such a way that for each $n\geq 1/\kappa$ there
exists a block model $W''\in\cA_{n,\geq\kappa}$ with
\[
\hat\delta_p(W'',\WW{W'})\leq \sqrt[p]{\frac 4{\kappa n}}\|W'\|_p.
\]
If $n\geq n_0(\kappa)$, then
\[
\eps_{\geq\kappa}^{(p)}(W)
\leq
\min_{W''\in \cB_{\geq\kappa,n}}\delta_p(W'',W)
+2\sqrt[p]{\frac {4}{\kappa^2 n}}\|W\|_p.
\]
\end{lem}

\subsection{Convergence of $W$-weighted graphs}

Recall that by Theorem~\ref{thm:GW-conv}, the sequence $G_n=G_n(\rho_nW)$
converges to $W$ in the cut metric if $W\in L^p$ is  normalized,
$n\rho_n\to\infty$, and either $\limsup\rho_n\|W\|_\infty\leq 1$ or
$\rho_n\to 0$.  Our next lemma, which is a slight strengthening of
Theorem~2.14(a) from \cite{BCCZ14a}, states that for the weighted graphs
$\Q_n(\rho_nW)$, the same holds in the tighter distance $\delta_p$.
Recalling that for any graphon, we can find an equivalent graphon over
$[0,1]$, we will restrict ourselves to the case where $W$ is a graphon over
$[0,1]$, in which case we can use an even tighter distance, the distance
$\hat\delta_p$ defined in \eqref{hat-delta-pW}.

\begin{lem}\label{lem:WH-as}
Let $p\geq 1$, let $W$ be a normalized $L^p$ graphon over $[0,1]$,
let $x_1,x_2,\ldots\in[0,1]$ be chosen i.i.d.\ uniformly at random, and let
$\rho_n$ be a sequence of positive numbers such that $\rho_n\to 0$.  Given
$n\geq 2$, let $\Q_n$ be the $n \times n$ matrix with entries
$\min\{1,\rho_n W(x_i,x_j)\}$, relabelled in such a way that
$x_1<x_2<\dots<x_n$.  Then a.s.\
$\|\frac 1{\rho_n}\WW{\Q_n}-W\|_p\to 0$, so in particular
$\rho(\Q_n)/\rho_n\to 1$ and $\hat\delta_p(\frac 1{\rho_n}\Q_n,W)\to 0$.
\end{lem}

The lemma is proved in Appendix~\ref{app:Auxiliary}.  The next lemma is a
quantitative version of Lemma~\ref{lem:WH-as} for block models, and is also
proved in Appendix~\ref{app:Auxiliary}.

\begin{lem}
\label{lem:WH-as-block} Let $C$ be a positive real number, let
$\kappa\in (0,1)$, and let $W'$ be a block model
with minimal class size at least $\kappa$, represented as a graphon over
$[0,1]$. If $\frac 1{n\kappa}\log n\leq C$, then
\[
\hat\delta_p(H_n(W'),W')=O_p\left(\sqrt[2p]{\frac {\log n}{n\kappa}}\right)\|W'\|_p,
\]
and if $\kappa=\kappa_n$ is such that $\limsup \frac 1{\kappa n}\log n < C$,
then with probability one, there exists a random $n_0$ such that for $n\geq
n_0$,
\[
\hat\delta_p(H_n(W'),W') = O\left(\sqrt[2p]{\frac {\log n}{n\kappa}}\right)\|W'\|_p.
\]
Here the constants implicit in the big-$O$ and $O_p$ symbols depend only on $C$.
\end{lem}

\section{Least squares estimation}
\label{sec:L2}

In this section, we prove the following theorem, which shows that the least
squares estimator is consistent.  To state the theorem, we define
\[
\tail_{\rho}^{(p)}(W)=\|W-\min\{W,\rho^{-1}\}\|_p.
\]
These tail bounds are easy to estimate when $W$ is an $L^{p'}$ graphon for
some $p'>p$, in which case they decay as a power of $\rho$:
\begin{align*}
\tail_{\rho}^{(p)}(W)&=
\|(W -\rho^{-1}) 1_{W \ge \rho^{-1}}\|_p\\
&\le
\|W  1_{W \ge \rho^{-1}}\|_p\\
&\le \big\| W (W\rho)^{p'/p-1}\big\|_p\\
&= \rho^{p'/p-1} \| W\|_{p'}^{p'/p}.
\end{align*}
When $W$ is an $L^{p'}$ graphon for $p'=p$ but not $p'>p$, tail bounds
become more subtle, but it remains the case that
\[
\tail_{\rho}^{(p)}(W) \to 0
\]
as $\rho \to 0$.

\begin{thm}\label{thm:L2} Let $W$ be an $L^2$ graphon, normalized so that $\|W\|_1=1$,
and let $\widehat W=(\hat p,\hat B)$ be the output of the least squares
algorithm \eqref{L2-algorithm} for a $W$-random graph $G$ on $n$ vertices
with target density $\rho$.

(i) If $\kappa\in (n^{-1},1]$ and
$\frac{1+\log (1/\kappa)}{\kappa^2}=O(\rho n)$, then
\[
  \delta_2\Bigl(\frac 1{\rho}\widehat W, W\Bigr)
  = O_p\paren{\eps_{\geq\kappa}^{(2)}(W)
  +\sqrt[4]{\frac{1+\log(1/\kappa)}{\kappa^2\rho n}}
  +\sqrt[4]{\frac {\log n}{\kappa n } }+\tail_{\rho}^{(2)}(W)}.
  \]

(ii) If $\kappa\in (0,1]$ is fixed and $\rho=\rho_n$ is such that $\rho_n\to
0$ and $n\rho_n\to\infty$, then

\[
  \delta_2\Bigl(\frac 1{\rho}\widehat W, W\Bigr)
  \to\eps_{\geq\kappa}^{(2)}(W)
  \quad\text{with probability $1$.}
  \]

(iii) If $\rho=\rho_n$ and $\kappa=\kappa_n$ are such that $\rho_n\to 0$,
$n\rho_n\to\infty$, $\kappa_n\to 0$, and
$\kappa_n^{-2}\log(1/\kappa_n)=o(n\rho_n)$ as $n\to\infty$, then
\[
\delta_2\Bigl(\frac 1{\rho}\widehat W, W\Bigr)\to 0
  \quad\text{with probability $1$.}
\]
\end{thm}

Conceptually, the proof of Theorem~\ref{thm:L2} is based on the following
two observations. First, for any map $\pi\colon [n]\to[k]$ and any $k\times
k$ matrix $B$,
\[
\|A(G)-B^\pi\|_2^2=\|A(G)\|_2^2 - 2\langle A(G),B^\pi\rangle + \|B^\pi\|^2_2.
\]
Therefore, the $\argmin$ of the left side is the same as the $\argmax$ of
$2\langle A(G),B^\pi\rangle - \|B^\pi\|^2_2$. Second, conditioned on the
weighted $W$-random graph $\Q=\Q_n(\rho_nW)$,
\[
\E\Bigl[2\langle A(G),B^\pi\rangle - \|B^\pi\|^2_2\Bigr]=2\langle \Q,B^\pi\rangle - \|B^\pi\|^2_2.
\]
Up to errors stemming from imperfect concentration, we therefore expect that
the argmin $(\hat B,\hat\pi)$ from \eqref{L2-algorithm} is a maximizer for
$2\langle \Q,B^\pi\rangle - \|B^\pi\|^2_2$, and hence a minimizer for $\|\Q -
B^\pi\|_2$. Thus, we would expect that, again up to issues of concentration,
the $L^2$ error is bounded by $\hat\eps_{\geq\kappa}^{(2)}(\Q)$, where for an
arbitrary $n\times n$ matrix $H$,
\[
\hat\eps_{\geq\kappa}^{(2)}(H)=\min_{B\in\cA_{n,\geq\kappa}}\|H-B\|_2.
\]

For bounded graphons, this strategy was implemented in \cite{BCS15}, leading
to (i) a proof of consistency for all bounded graphons $W$ and (ii) a
differentially private algorithm achieving the same goal under slightly less
general conditions (requiring $\rho n$ to grow at least like $\log n$). For
the case of general $L^2$ graphons, the above motivation still lies behind
our proof, but the actual implementation  proceeds along slightly different
lines, and combines elements of the (sparse graph) strategy of \cite{BCS15}
with elements of the (dense graph) strategy developed in \cite{GaoLZ14}. The
resulting estimates are stated in Theorem~\ref{thm:L2-H}, which bounds the
$L^2$ difference between the output of the algorithm \eqref{L2-algorithm}
and the matrix $\Q$ in terms of $\hat\eps_{\geq\kappa}^{(2)}(\Q)$ and an
error term representing errors from imperfect concentration. To obtain
Theorem~\ref{thm:L2} from Theorem~\ref{thm:L2-H}, we will need to transform
an estimate on the $L^2$ error with respect to $\Q$ into an $L^2$ error with
respect to $W$, and we will want to express the result in terms of
$\eps_{\geq\kappa}^{(2)}(W)$ instead of $\hat\eps_{\geq\kappa}^{(2)}(\Q)$.
This leads to two extra error terms, the last two terms in the bound of
statement (i) in Theorem~\ref{thm:L2}.

Before stating Theorem~\ref{thm:L2-H} formally, we recall that any block
model $W\in\cB_{n,\geq\kappa}$ can be represented by an $n\times n$ matrix
$M_n(W)\in\cA_{n,\geq\kappa}$ such that $W$ and $M_n(W)$ are equivalent as
graphons; see \eqref{k-block-matrix} and the discussion preceding it.

\begin{thm}\label{thm:L2-H} Let $W$ be a normalized $L^2$ graphon,
let  $0<\rho,\kappa\leq 1$ and  $n\in \N$, and let $G=G_n(\rho W)$,
$\Q=\Q_n(\rho W)$ and let $\widehat W=(\hat p,\hat B)$ be the output of the
least squares algorithms  \eqref{L2-algorithm} with input $G$. If
$n\kappa>1$ and $\frac{1+\log (1/\kappa)}{\kappa^2}=O(\rho n)$, then
\[
  \hat\delta_2\paren{ M_n(\widehat W),\Q}
  \leq \hat\eps_{\geq\kappa}^{(2)}\paren{\Q}
  +  O_p\paren{\rho\sqrt[4]{{\frac{1+\log(1/\kappa)}{\kappa^2\rho n}
  }}},
\]
where the constant implicit in the $O_p$ symbol depends on the $L^2$ norm of
$W$.

If $\rho=\rho_n$ is such that $n\rho_n\to\infty$ and $\rho_n\to 0$, then
almost surely, for $n$ large enough and all $\kappa$ with $n \kappa
>1$ and $\frac{1+\log (1/\kappa)}{\kappa^2}=O(\rho n)$,
\[
   \hat\delta_2\paren{ M_n(\widehat W),\Q}
  \leq \hat\eps_{\geq\kappa}^{(2)}\paren{\Q}
  +  O\paren{\rho\sqrt[4]{{\frac{1+\log(1/\kappa)}{\kappa^2\rho n}
  }}},
\]
where again the constant implicit in the big-$O$ symbol depends on the $L^2$
norm of $W$.
\end{thm}

\begin{proof}
Let $\hat M=M_n(\widehat W)$, $A=A(G)$, and $k=\lceil\frac n{\lfloor \kappa
n\rfloor}\rceil$.

As a first step, we will prove that
\begin{equation}
\label{hatH-bd}
\hat\delta_2\paren{\hat M,\Q}
  \leq \hat\eps_{\geq\kappa}^{(2)}\paren{\Q}
  +{2k\eps}+ {2\sqrt{k\eps\|\Q\|_2}}
\quad\text{where}\quad
\eps=\max_{\pi\colon [n]\to [k]}\|A_\pi-\Q_\pi\|_1.
\end{equation}
To prove \eqref{hatH-bd} we note that $\hat M=M_n(\widehat W)$ is a minimizer
of $\|A - M\|_2$ over all $M\in\cA_{n,\geq\kappa}$.  As a consequence,
\[
-2\langle A, \hat M\rangle +\|\hat M\|_2^2
\leq
-2\langle A, M\rangle +\|M\|_2^2
\]
for all $M\in\cA_{n,\geq\kappa}$, which in turn implies that
\begin{align*}
  \hat\delta_2\paren{\hat M, \Q}^2
 & \leq \bigl\|\hat M- \Q\bigr\|^2_2\\
 & \leq
  \|M\|_2^2-2\scprod{ \hat M}{ \Q}
  +\bigl\| \Q\bigr\|_2^2
  + 2\scprod{\hat M-M}{A}
  \\
  &= \bigl\|M- \Q\bigr\|^2_2
  + 2\scprod{\hat M-M}{A-\Q}
  .
\end{align*}
Since $M,\hat M\in \cA_{n,\geq\kappa}$, we know that there are partitions
$\pi,\hat\pi\colon [n]\to[k]$ such that $M=M_\pi$, $\hat M=\hat M_{\hat\pi}$,
and all non-empty classes of $\pi$ and $\hat\pi$ have size at least $\lfloor
\kappa n\rfloor$. As a consequence
\begin{align*}
|\langle M,A-\Q\rangle|
&=|\langle M,(A-\Q)_\pi\rangle|
\leq \|M\|_\infty\|(A-\Q)_\pi\|_1\\
&\leq \eps\|M\|_\infty
\leq \eps\frac{n}{\lfloor \kappa n\rfloor}\|M\|_2
\leq k\eps\|M\|_2,
\end{align*}
where in the second to last step we used that $M$ is an $n\times n$ block
matrix such that each block contains at least $\lfloor \kappa n\rfloor^2$
elements.  Bounding $|\langle \hat M,A-\Q\rangle|$ in the same way, we find
that
\[
\hat\delta_2\paren{\hat M, \Q}^2
 \leq \bigl\|M- \Q\bigr\|^2_2 + {2k\eps}(\|M\|_2+\|\hat M\|_2).
\]
Bounding $ \|\hat M\|_2 =\hat\delta_2(0,\hat M) \leq  \|\Q\|_2+
\hat\delta_2\paren{\hat M, \Q} $ and $\|M\|_2\leq  \|\Q\|_2+\|M- \Q\|_2$, a
small calculation then shows that
\[
  \paren{\hat\delta_2\paren{\hat M, \Q}-{k\eps}}^2
  \leq \paren{\bigl\| M- \Q\bigr\|_2+{k\eps}}^2 + {4k\eps}{\|\Q\|_2}.
\]
Choosing $M$ in such a way that
$\hat\eps_{\geq\kappa}^{(2)}\paren{\Q}=\norm{M-\Q}_2$, this proves
\eqref{hatH-bd}.

For all $\pi\colon [n]\to [k]$, we have $\E[A_\pi\mid\Q]=\Q_\pi$. Using this
fact and a concentration argument, one can show that conditioned on $\Q$,
with probability at least $1-e^{-n}$
\begin{equation}
\label{eps-concentration}
\eps
\leq 8\sqrt{\rho(\Q)\paren{\frac{1+\log k}{n }+\frac{k^2}{n^2}}},
\end{equation}
whenever $\rho(\Q)n\geq 1$; see Lemma~\ref{cor:L1-concentration} in
Appendix~\ref{sec:concentration}.  The lemma also gives a bound on the
expectation, implying in particular that conditioned on $\Q$,
\[
\eps=O_p\paren{\sqrt{\rho(\Q)\paren{\frac{1+\log k}{n }+\frac{k^2}{n^2}}}},
\]
whether or not the condition $\rho(\Q)n\geq 1$ holds.

Since $\E[\rho(\Q)]\leq \rho \|W\|_1=\rho$ and $\E[\|\Q\|_2^2]\leq
\rho^2\|W\|_2^2$, this proves that
\[2k\eps+ {4\sqrt{k\eps\|\Q\|_2}}
=
O_p\paren{\rho\sqrt{{\frac{k^2(1+\log k)}{\rho n}
  +\frac {k^4}{\rho n^2}}}}
  +
O_p\paren{\rho\sqrt[4]{{\frac{k^2(1+\log k)}{\rho n}
  +\frac {k^4}{\rho n^2}}}},
\]
with the constant implicit in the $O_p$ symbol depending on $\|W\|_2$. To
transform this bound into the bound in the statement of the theorem, we
observe that for $\kappa=1$, $k=\lceil\frac n{\lfloor \kappa
n\rfloor}\rceil$ is equal to $\frac 1\kappa$, while for $\kappa<1$, the
assumption $n\kappa>1$ implies that $k=\lceil\frac n{\lfloor \kappa
n\rfloor}\rceil\leq\frac 3{2\kappa}$.  In either case,
\[
\frac{k^2(1+\log k)}{ n}=O\paren{{\frac{1+\log(1/\kappa)}{\kappa^2n}}}
\]
and
\[
\frac {k^4}{ n^2}=O\paren{{\frac{1}{\kappa^4n^2}}}
=
O\paren{\paren{\frac{1+\log(1/\kappa)}{\kappa^2n}}^2}
=O\paren{\frac{1+\log(1/\kappa)}{\kappa^2n}},
\]
where in the last step we used that the assumption
$\frac{1+\log(1/\kappa)}{\kappa^2}=O(\rho n)$ implies that $\frac{1+\log(1/\kappa)}{\kappa^2
n}=O(1)$.  Thus,
\begin{align*}
2k\eps+ {4\sqrt{k\eps\|\Q\|_2}}
&=
O_p\paren{\rho\sqrt{\frac{1+\log(1/\kappa)}{\kappa^2n}}
  +
\rho\sqrt[4]{\frac{1+\log(1/\kappa)}{\kappa^2n}}}\\
&=
O_p\paren{\rho\sqrt[4]{\frac{1+\log(1/\kappa)}{\kappa^2n}}},
\end{align*}
again because $\frac{1+\log(1/\kappa)}{\kappa^2 n}=O(1)$.  This completes
the proof of the bound in probability.

To prove the a.s.\ statement, we note that by Lemma~\ref{lem:WH-as},
$\rho(\Q_n)/\rho_n\to 1$, which together with hypothesis that $n \rho_n \to
\infty$ implies that almost surely, $n\rho(\Q_n)\geq 1$ holds for
sufficiently large $n$, which allows us to use the bound
\eqref{eps-concentration}. By a simple union bound, this bound holds for all
$k\leq n$ with probability at least $1-ne^{-n}$. Since the failure
probability is summable, we conclude that there exists a random $n_0$
(depending on $W$ and the sequence $\rho_n$, but not on $k$ or $\kappa$)
such that the bound \eqref{eps-concentration} holds for all $n\geq n_0$ and
all $k\leq n$. Combined with the fact that by the law of large numbers for
$U$-statistics (see Lemma~\ref{lem:U-Stat} in Appendix~\ref{app:Auxiliary}),
$\frac 1{\rho_n}\|\Q\|_2\to \|W\|_2$ a.s.\ as $n\to\infty$, we obtain the
almost sure statement of the theorem.
\end{proof}

\begin{proof}[Proof of Theorem~\ref{thm:L2}]
Let $(\Omega,\cF,\pi)$ be the probability space on which $W$ is defined, and
let $\Q=Q_n(\rho W)$ as before.  Defining $W_\rho=\min\{W,1/\rho\}$, we will
write $\Q$ as $\rho \HH_n(W_\rho)$ and $\tail_{\rho}^{(2)}(W)=\|W-W_\rho\|_2$.

By the triangle inequality and the fact that the $\hat\delta_2$ distance
dominates the $\delta_2$ distance, we have
\begin{equation}
\label{L2-QtoW-bd1}
\begin{split}
  \delta_2\paren{\frac 1{\rho}\widehat W, W}
  &=\delta_2\paren{ M_n\paren{\frac 1\rho\widehat W},W}\\
  &\leq  \hat\delta_2\paren{ M_n\paren{\frac 1\rho\widehat W},\frac 1\rho\Q}
  +\delta_2\paren{\frac 1{\rho}\Q, W}.
\end{split}
\end{equation}
To bound the first term on the right side, we will use
Theorem~\ref{thm:L2-H} and then bound
$\hat\eps_{\geq\kappa}^{(2)}\paren{\frac 1{\rho}\Q}$ in terms of
$\eps_{\geq\kappa}^{(2)}\paren{W}$.

Recall that by Lemma~\ref{lem:block-model-approx} the infimum in the
definition \eqref{eps_kappa} of $\eps_{\geq\kappa}^{(2)}\paren{W}$ is a
minimum, and the minimizer $W'\in\cB_{\geq\kappa}$ satisfies $\|W'\|_2\leq
2\|W\|_2$. As established in Lemma~\ref{lem:block-model-approx-n}, we can
relabel the blocks of the block model $W'$ in such a way that
\[
\hat\delta_2(W'',\WW{W'})
\leq \sqrt{\frac 4{\kappa n}}\|W'\|_2\leq 2\sqrt{\frac 4{\kappa n}}\|W\|_2
=\sqrt{\frac {16}{\kappa n}}\|W\|_2
\]
for some $W''\in\cA_{n,\geq\kappa}$.  Setting $\widetilde W'=\WW{W'}$, we
find that
\begin{align*}
\hat\eps_{\geq\kappa}^{(2)}\paren{\frac 1{\rho}\Q}
&\leq
\hat\delta_2\paren{\frac 1{\rho}\Q,W''}\\
&\leq \hat\delta_2\paren{\frac 1{\rho}\Q,\widetilde W'}+\sqrt{\frac {16}{\kappa n}}\|W\|_2\\
&=\hat\delta_2\paren{\HH_n(W_\rho),\widetilde W'}+\sqrt{\frac {16}{\kappa n}}\|W\|_2.
\end{align*}

Next we would like to choose a coupling $\mu$ of $\mathbf p$ and $\pi$ such
that
\[
  \eps_{\geq\kappa}^{(2)}(W)
  =\delta_2(W',W)
  =\|W'-W\|_{2,\mu},
\]
where $\|\cdot\|_{2,\mu}$ denotes the $L^2$ norm with respect to the
coupling $\mu$.  (This an abuse of notation, but it is more convenient than
writing out the formula, as in \eqref{del-p}.)  Such a coupling needn't
exist, but that is not a significant obstacle.  We could complete the proof
by looking at couplings that come arbitrarily close to the oracle error, but
instead we will switch to equivalent graphons over $[0,1]$, because
Lemma~\ref{lem:delta-over-01} then guarantees the existence of an optimal
coupling. The oracle error and tail bounds are invariant under equivalence,
so we can assume without loss of generality that the coupling $\mu$ exists.

We use this coupling to couple the random graphs $\Q(\rho W)$ and $\Q(\rho
W')$. With the help of the triangle inequality, we then conclude that
\begin{equation}
 \label{hat-eps-bd1}
\begin{split}
\hat\eps_{\geq\kappa}^{(2)}\paren{\frac 1{\rho}\Q} &\leq
    \|\HH_n(W_\rho)-\HH_n( W)\|_2
   +\|\HH_n(W)-\HH_n( W')\|_2\\
   &\quad \phantom{}+\hat\delta_2\paren{\HH_n(W'),\widetilde W'}
   +\sqrt{\frac {16}{\kappa n}}\|W\|_2.
\end{split}
\end{equation}

After these preparations, we start with the proof of (i). To this end, we
first use the triangle inequality and the fact that
$\delta_2(W',W)=\eps_{\geq\kappa}^{(2)}(W)$ to bound
\[
\begin{split}
\delta_2\paren{\frac 1{\rho}\Q, W}&=
\delta_2\paren{\HH_n(W_\rho), W}
\\
&\leq
\|\HH_n(W_\rho)-\HH_n(W)\|_2
+\|\HH_n(W)-\HH_n(W')\|_2\\
&\quad\phantom{}+
\delta_2(\HH_n(W'),W') + \eps_{\geq \kappa}^{(2)}(W).
\end{split}
\]
Next we estimate
\begin{align*}
\E\left[\|\HH_n(W_\rho)-\HH_n(W)\|_2\right]
&=
\E\left[\|\HH_n(W_\rho-W)\|_2\right]
\leq\sqrt {\E\left[\|\HH_n(W_\rho-W)\|_2^2\right]}\\
&=\|W_\rho-W\|_2= \tail_{\rho}^{(2)}(W)
\end{align*}
and
\[
\E\left[\|\HH_n(W)-\HH_n(W')\|_2\right]\leq
\|W-W'\|_{2,\mu}=\eps_{\geq\kappa}^{(2)}(W).
\]
Since $\hat\delta_2\paren{\HH_n(W'),\widetilde W'}$ has the same distribution
as $\hat\delta_2\paren{\HH_n(\widetilde W'),\widetilde W'}$, we may then use
Lemma~\ref{lem:WH-as-block} and the fact that $\|W'\|_2\leq 2\|W\|_2$ to
conclude that
\[
\hat\eps_{\geq\kappa}^{(2)}\paren{\frac 1{\rho}\Q}
  =
    O_p\paren{\tail_{\rho}^{(2)}(W)
   +\eps_{\geq\kappa}^{(2)}(W)+\sqrt[4]{\frac {\log n}{n\kappa}}\|W\|_2
  }.
\]
(Note that $(1+\log(1/\kappa))\kappa^{-2}=O(\rho n)$ implies that
$1/\sqrt{n}=O(\kappa)$ and hence $\log n=O(\kappa n)$, as required for the
application of Lemma~\ref{lem:WH-as-block}.) In a similar way, we use the fact
that $\delta_2(\HH_n(W'),W')$ has the same distribution as
$\delta_2(\HH_n(\widetilde W'),\widetilde W')$
 to conclude that
\[
\delta_2\paren{\frac 1{\rho}\Q, W}=
    O_p\paren{\tail_{\rho}^{(2)}(W)
   +\eps_{\geq\kappa}^{(2)}(W)+\sqrt[4]{\frac {\log n}{n\kappa}}\|W\|_2}.
\]
With the help of \eqref{L2-QtoW-bd1} and Theorem~\ref{thm:L2-H}, this implies
that
\[
  \delta_2\paren{\frac 1{\rho}\widehat W, W}
  = O_p\paren{
  \tail_{\rho}^{(2)}(W)
   +\eps_{\geq\kappa}^{(2)}(W)+\sqrt[4]{\frac {\log n}{n\kappa}}\|W\|_2
   +\sqrt[4]{{\frac{1+\log(1/\kappa)}{\kappa^2\rho n}
   }}},
\]
which concludes the proof of (i).

Next we prove (ii). Since $W$ is square integrable,
$\|W-W_\rho\|_2\to 0$ as $\rho\to 0$, so by combining the law of
large numbers for $U$-statistics (see Lemma~\ref{lem:U-Stat} in
Appendix~\ref{app:Auxiliary}) with a simple two $\eps$ argument,  we
conclude that a.s., the first term in \eqref{hat-eps-bd1} goes to zero.
Again by the law of large numbers for $U$-statistics, the second term goes
to $\|W'-W\|_{2,\mu}= \eps_{\geq\kappa}^{(2)}(W)$, and by
Lemma~\ref{lem:WH-as-block} and the fact that $H_n(W')$ and $H_n(\widetilde
W')$ have the same distribution, the third term goes to zero as well. Thus
a.s., the right side of \eqref{hat-eps-bd1} goes to
$\eps_{\geq\kappa}^{(2)}(W)$. Combined with \eqref{L2-QtoW-bd1},
Lemma~\ref{lem:WH-as}, and Theorem~\ref{thm:L2-H}, we see that for fixed
$\kappa$,
\[
  \limsup_{n\to\infty}\delta_2\paren{\frac 1{\rho}\widehat W, W}
  \leq  \eps_{\geq\kappa}^{(2)}(W)\quad\text{with probability $1$.}
\]
On the other hand, by the second bound in
Lemma~\ref{lem:block-model-approx-n},
\[
\eps_{\geq\kappa}^{(p)}(W)
\leq\liminf_{n\to\infty}\min_{W''\in \cB_{\geq\kappa,n}}\delta_p(W'',W).
\]
Since $\frac 1{\rho}\widehat W\in \cB_{\geq\kappa,n}$, this gives $
\eps_{\geq\kappa}^{(2)}(W) \leq \liminf_{n\to\infty}\delta_2\paren{\frac
1{\rho}\widehat W, W} $, completing the proof of (ii).

To prove (iii), note that the condition
$\kappa_n^{-2}\log(1/\kappa_n)=o(n\rho_n)$ implies in particular that
$\kappa_n\sqrt n\to\infty$, which in turn implies that $\frac 1{\kappa_n
n}\log n\to 0$. We may therefore again use Lemma~\ref{lem:WH-as-block} to
show that the third term in \eqref{hat-eps-bd1} goes to zero a.s.  The first
term does not depend on $\kappa$, and hence goes to zero just as before, but
now the second term goes to zero as well, by a two $\eps$ argument invoking
now the fact that $\|W'-W\|_{2,\mu}=  \eps_{\geq\kappa_n}^{(2)}(W)\to 0$.
Since the condition $\kappa_n^{-2}\log(1/\kappa_n)=o(n\rho_n)$ clearly
implies that $n\kappa_n\to\infty$, we conclude that a.s.,
$\hat\eps_{\geq\kappa_n}^{(2)}\paren{\frac 1{\rho}\Q}\to 0$. Combined with
\eqref{L2-QtoW-bd1}, Lemma~\ref{lem:WH-as}, and Theorem~\ref{thm:L2-H}, this
implies (iii).
\end{proof}

\section{Cut norm estimation for general $L^1$ graphons}
\label{sec:cut}

In this section, we prove the following theorem, which shows that the least
cut norm estimator is consistent.

\begin{thm}\label{thm:cut} Let $W$ be an $L^1$ graphon, normalized so that $\|W\|_1=1$,
and let $\widehat W=(\hat p,\hat B)$ be the output of the least cut norm
algorithms  \eqref{alg:cut}.

(i) If $\kappa\in [\frac {\log n}n,1]$, then
\[
  \delta_\square\Bigl(\frac 1{\rho}\widehat W, W\Bigr)
  = O_p\paren{\eps_{\geq\kappa}^{(1)}(W)
  +\sqrt{\frac{1}{\rho n}}
  +\sqrt{\frac {\log n}{\kappa n } }+\tail_{\rho}^{(1)}(W)}.
\]

(ii) If $\kappa\in (0,1]$ is fixed and $\rho=\rho_n$ is such that $\rho_n\to
0$ and $n\rho_n\to\infty$, then
\[
 \limsup\delta_\square\Bigl(\frac 1{\rho}\widehat W, W\Bigr)
  \leq 2\eps_{\geq\kappa}^{(1)}(W)
  \quad\text{with probability $1$.}
\]

(iii) If $\rho=\rho_n$ and $\kappa=\kappa_n$ are such that $\rho_n\to 0$,
$n\rho_n\to\infty$, $\kappa_n\to 0$, and $\frac{\log n}{n\kappa_n}\to 0$,
then
\[
\delta_\square\Bigl(\frac 1{\rho}\widehat W, W\Bigr)\to 0
  \quad\text{with probability $1$.}
\]
\end{thm}

The proof relies again on a concentration argument, this time starting from
the observation that for all $S,T\subseteq [n]$,
\begin{equation}
\label{cut-in-expectation}
\E\Bigl[\sum_{(x,y)\in S\times T} A_{xy}(G)\Bigr]=\sum_{(x,y)\in S\times T}\Q_{x,y}.
\end{equation}
Therefore, up to issues of concentration, minimizing the cut distance
between $A(G)$ and a block model in $\cB_{\geq\kappa,n}$ is  the same as
minimizing the cut distance between $\Q$ and a block model in
$\cB_{\geq\kappa,n}$.  In other words, up to issues of concentration, one
might hope that the distance between $\Q$ and the output $\widehat W$ of the
algorithm \eqref{alg:cut} is just $\hat\eps_{\geq\kappa,\square}(\Q)$, where
for an arbitrary $n\times n$ matrix $H$,
\[
\hat\eps_{\geq\kappa,\square}(H)=\min_{B\in\cA_{n,\geq\kappa}}\|H-B\|_\square.
\]
It turns out that we lose a factor of two with respect to this optimum, due
to the fact that in \eqref{alg:cut}, we optimize over all block matrices of
the form $A(G)_\pi$, rather than all block matrices that are constant on the
blocks determined by $\pi$.  While these two minimizations are equivalent in
the least squares case, they are not here, leading to the loss of a factor
of two.\footnote{At the cost of an even slower algorithm, this could be
cured by redefining the algorithm \eqref{alg:cut} to optimize over all block
matrices that are constant on the blocks determined by $\pi$.}

The following theorem states our approximation guarantees with respect to
$\Q$.  Theorem~\ref{thm:cut} follows from it in essentially the same way as
Theorem~\ref{thm:L2} follows from Theorem~\ref{thm:L2-H}.  To state it, we
recall the definition \eqref{hat-delta-cut} of the distance $\hat
\delta_\square$.

\begin{thm}\label{thm:cut-H} Let $W$ be a normalized $L^1$ graphon,
let  $0<\rho\leq 1$ and  $n\in \N$, and let $G=G_n(\rho W)$ and
$\Q=\Q_n(\rho W)$. If $\kappa\in (n^{-1},1]$ and $\widehat W=(\hat p,\hat
B)$ is the output of the least cut norm algorithm \eqref{alg:cut} with input
$G$, then
\[
  \hat\delta_\square\paren{ M_n(\widehat W),\Q}
  \leq 2\hat\eps_{\geq\kappa,\square}\paren{\Q}
  +  O_p\paren{\rho\sqrt{\frac{1}{\rho n}
  }}.
\]

If $\rho=\rho_n$ is such that $n\rho_n\to\infty$ and $\rho_n\to 0$, then
almost surely, for $n$ large enough and all $\kappa\in(n^{-1},1]$,
\[
   \hat\delta_\square\paren{ M_n(\widehat W),\Q}
  \leq 2\hat\eps_{\geq\kappa,\square}\paren{\Q}
  +  O\paren{\rho\sqrt{\frac{1}{\rho n}
    }}.
\]
\end{thm}

\begin{proof}
Let $A=A(G)$ and $k=\lceil\frac n{\lfloor \kappa n\rfloor}\rceil$. We will
show that
\begin{equation}
\label{hatH-cut-bd}
\hat\delta_\square\paren{ M_n(\widehat W),\Q}
  \leq 2\hat\eps_{\geq\kappa,\square}\paren{\Q}+3\|\Q-A\|_\square.
\end{equation}
To this end, we first prove that
\begin{equation}
\label{hat-eps-cut-bd}
\| M_n(\widehat W)-A\|_\square\leq
2\min_{M\in\cA_{n,\geq\kappa}}\|M-A\|_\square.
\end{equation}
To see this, we note that $\cA_{n,\geq\kappa}$ consists of all $n\times n$
matrices $M$ such that $M=M_\pi$ for some $\pi\colon [n]\to [k]$ such that
the smallest non-empty class of $\pi$ has at least size $\lfloor\kappa
n\rfloor$.  Next we observe that for all $\pi\colon [n]\to [k]$, the map
$H\mapsto H_\pi$ is a contraction in the cut norm.  As a consequence, for all
$n\times n$ matrices $M$ with $M=M_\pi$,
\[
\|A_\pi-A\|_\square
\leq
\|A_\pi-M_\pi\|_\square+\|M-A\|_\square
\leq 2\|M-A\|_\square.
\]
Because $M_n(\widehat W)=A_{\hat \pi}$ for some $\hat\pi\colon [n]\to [k]$
that minimizes $\|A-A_\pi\|_\square$ over all $\pi$ whose smallest non-empty
class has size at least $\lfloor\kappa n\rfloor$, the bound
\eqref{hat-eps-cut-bd} now follows.

After this preparation, the proof of \eqref{hatH-cut-bd} is straightforward.
Indeed,
\begin{align*}
\hat\delta_\square\paren{ M_n(\widehat W),\Q}
&\leq
\| M_n(\widehat W)-A\|_\square+\|A-\Q\|_\square
\\
&\leq
2\min_{M\in\cA_{n,\geq\kappa}}\|M-A\|_\square +\|A-\Q\|_\square
\\
&\leq
2\min_{M\in\cA_{n,\geq\kappa}}\|M-\Q\|_\square +3\|A-\Q\|_\square
\\
&=2\hat\eps_{\geq\kappa,\square}\paren{\Q}+3\|\Q-A\|_\square.
\end{align*}
From here on, the proof proceeds along the same lines as that of
Theorem~\ref{thm:L2-H}, this time starting from the observation
\eqref{cut-in-expectation}.  Using this fact and a concentration argument,
we now can show that conditioned on $\Q$, if $\rho(\Q)n\geq 1$ then
\[
\|\Q-A\|_\square \leq
15\sqrt{\frac{\rho(\Q)}{n}}
\]
holds with probability at least $1-e^{-n}$, and
\[
\|\Q-A\|_\square=O_p\paren{\sqrt{\frac{\rho(\Q)}{n}}},
\]
independently of the condition $\rho(\Q)n\geq 1$; see
Lemma~\ref{lem:cut-concentration} in Appendix~\ref{sec:concentration} for
details. The assertions of the theorem now follow.
\end{proof}

\begin{proof}[Proof of Theorem~\ref{thm:cut}]
Keeping the notation from the proof of Theorem~\ref{thm:L2}, and using the
fact that the distance $\hat\delta_\square$ is dominated by the distance
$\hat\delta_1$, we now bound
\begin{align}
\label{cut-QtoW-bd1}
  \delta_\square\paren{\frac 1{\rho}\widehat W, W}
  \leq \hat\delta_\square\paren{ M_n\paren{\frac 1\rho\widehat W},\frac 1\rho\Q}
  +\delta_1\paren{\frac 1{\rho}\Q, W}.
\end{align}
Using Lemma~\ref{lem:block-model-approx} and
Lemma~\ref{lem:block-model-approx-n} for $p=1$, we now bound
\[
\hat\eps_{\geq\kappa,\square}\paren{\frac 1{\rho}\Q}
\leq \hat\delta_1\paren{\frac 1{\rho}\Q,\widetilde W'}+{\frac 8{\kappa n}}
=\hat\delta_1\paren{\HH_n(W_\rho),\widetilde W'}+{\frac 8{\kappa n}},
\]
where $W'$ is a minimizer for \eqref{eps_kappa} for $p=1$,  with
$\|W'\|_1\leq 2\|W\|_1=2$, and $\widetilde W'$ again stands for $\WW{W'}$.
Writing $\eps_{\geq\kappa}^{(1)}(W)$ as
$\eps_{\geq\kappa}^{(1)}(W)= \delta_1(W',W)=\|W'-W\|_{1,\mu}$ for some coupling $\mu$ of
$\mathbf p$ and $\pi$ (which we can assume exists without loss of generality
by passing to equivalent graphons over $[0,1]$, as in the proof of
Theorem~\ref{thm:L2}), we then get
\begin{align*}
\hat\eps_{\geq\kappa,\square}\paren{\frac 1{\rho}\Q}
&\leq
    \|\HH_n(W_\rho)-\HH_n( W)\|_1
   +\|\HH_n(W)-\HH_n( W')\|_1\\
&\quad\phantom{}   +\hat\delta_1\paren{\HH_n(W'),\widetilde W'}
   +{\frac 8{\kappa n}}
\end{align*}
and
\begin{align*}
\delta_1\paren{\frac 1{\rho}\Q, W}&\leq
\|\HH_n(W_\rho)-\HH_n(W)\|_1
+\|\HH_n(W)-\HH_n(W')\|_1\\
&\quad\phantom{}+\delta_1(\HH_n(W'),W')+ \eps_{\geq\kappa}^{(1)}(W),
\end{align*}
where as before $\HH_n(W)$ and $\HH_n(W')$ are coupled with the help of $\mu$.
From here on the proof of Theorem~\ref{thm:cut} proceeds exactly as the
proof of Theorem~\ref{thm:L2}, with the condition $ \frac 1{n\kappa}\log
n =O(1)$ that is needed to apply Lemma~\ref{lem:WH-as-block} guaranteed by
the hypotheses of the theorem. We finally arrive at
\[
\hat\eps_{\geq\kappa,\square}\paren{\frac 1{\rho}\Q}
  =
    O_p\paren{\tail_{\rho}^{(1)}(W)
   +\eps_{\geq\kappa}^{(1)}(W)+\sqrt{\frac {\log n}{n\kappa}}
  }
\]
and
\[
\delta_1\paren{\frac 1{\rho}\Q, W}=
    O_p\paren{\tail_{\rho}^{(1)}(W)
   +\eps_{\geq\kappa}^{(1)}(W)+\sqrt{\frac {\log n}{n\kappa}}}.
\]
With the help of \eqref{cut-QtoW-bd1} and Theorem~\ref{thm:cut-H}, this
gives the bound in probability.

The almost sure statements are proved similarly.
\end{proof}

\section{Graphon estimation via degree sorting}
\label{sec:degreesort}

In this section we analyze the behavior of the degree sorting algorithm
described in the introduction. We will use the notation from
Section~\ref{sec:degree-distribution} for degrees and the degree
distribution.

\begin{thm} \label{thm:mon}
Let $W$ be a graphon whose degree distribution function $D_W\colon
[0,\infty) \to [0,1]$ is continuous, let $G_n$ be a $W$-random graph on $n$
vertices with target density $\rho_n$, and let $\widehat{W}_n$ be the output
of the degree sorting algorithm with $k_n$ parts and input $G_n$.

Suppose $\rho_n \to 0$, $n \rho_n \to \infty$, $k_n \to \infty$, $\log k_n =
o(n\rho_n)$, and $k_n = o\big(n\sqrt{\rho_n}\big)$ as $n \to \infty$. Then
$\rho_n^{-1} \widehat{W}_n$ converges a.s.\ to $W$ under $\delta_1$.
\end{thm}

Note that $D_W$ is continuous if and only if the degree distribution of $W$
is atomless.  Graphons with this property have a useful characterization as
graphons over $[0,1]$:

\begin{lem} \label{lem:continuous-equiv}
The degree distribution function $D_W$ of a graphon $W$ is continuous if and
only if $W$ is equivalent to a graphon $U$ over $[0,1]$ whose degrees $U_x$
are strictly decreasing in $x$.
\end{lem}

\begin{proof}
Every graphon $W$ is equivalent to a graphon $U$ over $[0,1]$, and via
monotone rearrangement we can furthermore assume that $U_x$ is weakly
decreasing in $x$ (see \cite{Ryff70} for a thorough discussion of the
measure-theoretic technicalities). Then $D_U=D_W$, while $D_U$ is continuous
if and only if $U_x$ is strictly decreasing.
\end{proof}

If $W$ is a graphon over $(\Omega,\cF,\pi)$ and $\cP$ is a partition of $\Omega$ into
finitely many measurable pieces, then $W_{\cP}$ denotes the step function
defined by
\[
W_{\cP}(x,y) = \frac 1{\pi(I)\pi(J)}\int_{I \times J} W(u,v)
\, d\pi(u) \, d\pi(v)
\]
whenever $x$ is in the part $I$ of $\cP$ and $y$ is in the part $J$.  (This
is not well defined for parts of measure zero, but they can be ignored.) We
will need the following sufficient condition for when averaging over
partitions converges under the $L^1$ norm.

\begin{lem} \label{lemma:averaging}
Let $W$ be an $L^1$ graphon over $[0,1]$, and let $\cP_1,\cP_2,\dots$ be
partitions of $[0,1]$ into finitely many measurable pieces.  Let
$p_{n,\eps}$ be the probability that independent random elements $x,y \in
[0,1]$ satisfy $|x-y|\ge\eps$, conditioned on $x$ and $y$ lying in the same
part of $\cP_n$. If
\[
\lim_{n \to \infty} p_{n,\eps} = 0
\]
for each $\eps>0$, then
\[
\lim_{n \to \infty} ||W_{\cP_n} - W||_1 = 0.
\]
\end{lem}

\begin{proof}
Without loss of generality we can assume that $W$ is continuous, because
continuous functions are dense in $L^1$ and $||W_{\cP_n} - W'_{\cP_n}||_1
\le ||W-W'||_1$.

Let $J_1,\dots,J_N$ be the parts of $\cP_n$.  Then for $(x,y) \in J_i \times
J_j$,
\[
W_{\cP_n}(x,y) = \frac{1}{\lambda(J_i)\lambda(J_j)} \int_{J_i \times J_j} W(u,v) \, du \, dv.
\]
By combining this formula with
\[
W(x,y) = \frac{1}{\lambda(J_i)\lambda(J_j)} \int_{J_i \times J_j} W(x,y) \, du \, dv,
\]
we find that
\[
\|W_{\cP_n} - W\|_1 \le
\sum_{i,j=1}^N \frac{1}{\lambda(J_i)\lambda(J_j)}
\int_{J_i \times J_j} \int_{J_i \times J_j} |W(u,v)-W(x,y)| \, du \, dv \, dx \, dy.
\]

Because $W$ is continuous on $[0,1]^2$ (and hence uniformly continuous), for
each $\delta>0$, there exists $\varepsilon>0$ such that $|W(x,y)-W(u,v)| <
\delta$ whenever $|x-u|<\varepsilon$ and $|y-v| < \varepsilon$.  Then
\begin{align*}
||W_{\cP_n} - W||_1 &\le \delta + \sum_{i,j=1}^N \frac{2 \|W\|_\infty}{\lambda(J_i)\lambda(J_j)}
\int_{J_i \times J_j} \int_{J_i \times J_j} 1_{\text{$|x-u| \ge \eps$ or $|y-v| \ge \eps$}} \, du \, dv \, dx \, dy\\
&\le \delta + \sum_{i,j=1}^N \frac{4 \|W\|_\infty}{\lambda(J_i)\lambda(J_j)}
\int_{J_i \times J_j} \int_{J_i \times J_j} 1_{\text{$|x-u| \ge \eps$}} \, du \, dv \, dx \, dy\\
&= \delta + 4 \|W\|_\infty \sum_{i=1}^N \frac{1}{\lambda(J_i)} \int_{J_i \times J_i} 1_{\text{$|x-u| \ge \eps$}} \, du \, dx\\
&= \delta + 4 \|W\|_\infty p_{n,\eps}.
\end{align*}
It follows that
\[
\limsup_{n \to \infty} ||W_{\cP_n} - W||_1 \le \delta
\]
for each $\delta>0$, as desired.
\end{proof}

\begin{proof}[Proof of Theorem~\ref{thm:mon}]
By Lemma~\ref{lem:continuous-equiv}, we can assume that $W$ is a graphon
over $[0,1]$ for which the degrees $W_x$ are strictly decreasing in $x$.

Let $I_{i,n} = [(i-1)/n,i/n]$, so that $I_{1,n},I_{2,n},\dots,I_{n,n}$ form a
partition of $[0,1]$ (up to the measure-zero set of their endpoints, which
we will ignore).  We will assume the vertices of $G_n$ are ordered so that
the corresponding sample points in $[0,1]$ satisfy $x_1<x_2<\dots<x_n$, and
we view $G_n$ as a graphon over $[0,1]$ via the blocks $I_{i,n}$ and this
vertex ordering.

Let $d_1,\dots,d_n$ be the vertex degrees, and set $\bar d =
(d_1+\dots+d_n)/n$.  Recall that the degree sorting algorithm works as
follows.  We choose a permutation $\sigma$ of $[n]$ such that
\[
d_{\sigma(1)} \ge d_{\sigma(n)} \ge \dots \ge d_{\sigma(n)}
\]
and integers $0 = n_0 < n_1 < \dots < n_k = n$ such that
\[
\left| n_i - \frac{in}{k}\right| < 1.
\]
Then we define $\pi \colon [n] \to [k]$ by $\pi(j) = i$ if $n_{i-1} <
\sigma(j) \le n_i$.  The output of the algorithm is the block model
$\widehat W=(\hat p,\hat B)$ with $\hat p_i=1/k$ and $\hat B=A(G)/\pi$.

Let $V_1,\dots,V_k$ be the preimages of $1,\dots,k$ under $\pi$, and set
\[
J_i = \bigcup_{j \in V_i} I_{j,n}.
\]
Then $J_1,\dots,J_k$ form a partition $\cP_n$ of $[0,1]$, and
$\widehat{W}_n$ is equivalent to $(G_n)_{\cP_n}$. (Recall that we view $G_n$
as a graphon over $[0,1]$.) We wish to prove that
\[
\delta_1\big(\rho_n^{-1}(G_n)_{\cP_n},W\big) \to 0.
\]
In fact, we will prove that $\big\|\rho_n^{-1}(G_n)_{\cP_n} - W\big\|_1 \to
0$, given our ordering of the vertices of $G_n$.

We will use the notation established in previous sections, such as $\Q_n$
for the weighted random graph used to generate $G_n$.  Recall from
Lemma~\ref{lem:WH-as} that a.s.\ $\rho(\Q_n)/\rho_n \to 1$ and
$\|\rho_n^{-1}\Q_n-W\|_1\to 0$.

We begin with the inequality
\begin{align*}
\big\| \rho_n^{-1} (G_n)_{\cP_n}-W\|_1 &\le
\big\|\rho_n^{-1} (G_n)_{\cP_n} - \rho_n^{-1} (\Q_n)_{\cP_n}\big\|_1 +
\big\| \rho_n^{-1}(\Q_n)_{\cP_n}-\rho_n^{-1}\Q_n\big\|_1\\
&\quad\phantom{} + \big\|\rho_n^{-1}\Q_n-W \big\|_1.
\end{align*}
The third term on the right tends to zero a.s. For the first term, we have
\[
\big\|\rho_n^{-1} (G_n)_{\cP_n} - \rho_n^{-1}
(\Q_n)_{\cP_n}\big\|_1 =
\rho_n^{-1}\big\| (G_n)_{\cP_n} - (\Q_n)_{\cP_n}\big\|_1.
\]
By Lemma~\ref{cor:L1-concentration} and the fact that $\rho(\Q_n)/\rho_n\to
1$ a.s., we can bound $\big\| (G_n)_{\cP_n} - (\Q_n)_{\cP_n}\big\|_1$ by
$O\paren{\sqrt{\rho\paren{\frac{1+\log k}{n} + \frac{k^2}{n^2}}}}$ a.s., and
thus the hypotheses that $\log k_n = o(n\rho_n)$ and $k_n =
o\big(n\sqrt{\rho_n}\big)$ imply that
\[
\big\|\rho_n^{-1} (G_n)_{\cP_n} - \rho_n^{-1}
(\Q_n)_{\cP_n}\big\|_1 \to 0.
\]
All that remains is to handle the second term, namely $\big\|
\rho_n^{-1}(\Q_n)_{\cP_n}-\rho_n^{-1}\Q_n\big\|_1$.  Because
$\big\|\rho_n^{-1}\Q_n-W\big\|_1\to 0$, it will suffice to show that
$\big\|W_{\cP_n} - W\big\|_1 \to 0$.  We will do so using
Lemma~\ref{lemma:averaging}.

Fix $\eps > 0$, and let $p_{n,\eps}$ be the probability that independent
random elements $x,y \in [0,1]$ satisfy $|x-y|\ge\eps$, conditioned on $x$
and $y$ lying in the same part of $\cP_n$.  By contrast, let $p'_{n,\eps}$
be the probability that $|x-y|\ge\eps$ and both points lie in the same part
of $\cP_n$, without the conditioning. Because each part $J_i$ of $\cP_n$
satisfies $\lambda(J_i) = (1+o(1))/k_n$, proving that $p_{n,\eps} \to 0$ is
equivalent to proving that $k_n p'_{n,\eps} \to 0$.  Thus, to apply
Lemma~\ref{lemma:averaging}, we must show that $k_n p'_{n,\eps} \to 0$.

Instead of analyzing the points $x$ and $y$, it will be convenient to
consider the intervals $I_{\ell,n}$ and $I_{m,n}$ containing them. We will
use the bound
\begin{equation} \label{eq:pneps}
\begin{split}
p'_{n,\eps} & \le \Pr_{\ell,m\in[n]} \big(\text{$\pi(\ell)=\pi(m)$ and $\max \{|x-y| : x \in I_{\ell,n}, y \in I_{m,n}\} \ge \eps$}\big)\\
& = \Pr_{\ell,m\in[n]} \big(\text{$\pi(\ell)=\pi(m)$ and $|\ell/n-m/n| \ge \eps-1/n$}\big),
\end{split}
\end{equation}
where of course $\Pr_{\ell,m \in [n]}$ denotes the probability if $\ell$ and
$m$ are chosen uniformly at random from $[n]$.

To analyze these probabilities, we need to bound how close the degrees in
$G_n$ are to those in $W$.  Lemma~\ref{lem:degree-match} will provide
suitable bounds. To apply this lemma, we must quantify how quickly the
degrees in $W$ change as a function of distance. Let
\[
\delta = \inf_{|x-y| \ge \eps/4-1/n} |W_x-W_y|.
\]
Because $x \mapsto W_x$ is strictly decreasing, $\delta > 0$.  Call an
element $i \in [n]$ \emph{good} if the normalized degree $d_i/\bar d$ is
within $\delta/3$ of $W_x$ for some $x \in I_{i,n}$.  Taking $U =
\rho_n^{-1} G_n$ in Lemma~\ref{lem:degree-match} shows that the fraction of
bad elements is at most
\[
\frac{2}{\delta/3} \|\rho_n^{-1} G_n - W\|_\square,
\]
which tends to zero as $n \to \infty$.  If $i$ and $j$ are good and
$|i/n-j/n| \ge \eps/4$, then
\[
\left|\frac{d_i}{\bar d} - \frac{d_j}{\bar d}\right| \ge \delta/3.
\]
If follows that if $i$ and $j$ are good and $|i/n-j/n| \ge 3\eps/4$, then at
least the middle $\lfloor n \eps/4\rfloor$ vertices between $i$ and $j$ have
degrees strictly between $d_i$ and $d_j$.  When $n$ is large enough, this is
much larger than the number of vertices in any part of $\cP_n$.  In
particular, if $n$ is large enough then good $i$ and $j$ with $|i/n-j/n| \ge
3\eps/4$ cannot possibly end up in the same part after the degrees are
sorted.

Thus, by \eqref{eq:pneps},
\begin{align*}
p'_{n,\eps}
&\le \Pr_{\ell,m \in [n]} \big(\text{$\ell$ or $m$ is bad and $\pi(\ell)=\pi(m)$}\big)
\\
&\le 2\Pr_{\ell,m \in [n]} \big(\text{$\ell$ is bad and $\pi(\ell)=\pi(m)$}\big)
\\
&\le 2\Pr_{m \in [n]} \big(\text{$\ell$ is bad}\big)\max_i\lambda(J_i)\\
&\leq\frac{4}{\delta/3} \|\rho_n^{-1} G_n - W\|_\square\frac{1+o(1)}{k_n}.
\end{align*}
It now follows from $ \|\rho_n^{-1} G_n - W\|_\square \to 0$ that $k_n
p'_{n,\eps} \to 0$, as desired.
\end{proof}

\section{H\"older-continuous graphons}
\label{sec:Hoelder}

In this section, we analyze the least squares and the least cut norm
algorithms for the case of H\"older-continuous graphons. As discussed in the
introduction, our approach allows us to reduce this to the analysis of the
two error terms $\tail_\rho^{(p)}(W)$ and $\eps^{(p)}_{\geq\kappa}(W)$ for
$p=2$ and $p=1$, respectively, which reduces the analysis to pure
approximation theory.

Throughout this section, we consider graphons $W$ over $\R^d$ (equipped with
the standard Borel $\sigma$-algebra and some probability measure $\pi$) that
are $\alpha$-H\"older-continuous for some $\alpha\in (0,1]$, i.e., graphons
$W$ for which there exists a constant $C$ such that
\[
|W(x,y)-W(x',y)|\leq C |x-x'|_\infty^\alpha\quad\text{for all $x,x',y\in \R^d$,}
\]
with $|\cdot|_\infty$ denoting the $L^\infty$ distance on $\R^d$ (note that
we only require this for one of the two coordinates of $W$, since for the
other one it follows from the fact that $W$ is symmetric).  We denote the set
of graphons obeying this bound  by $\HOLD$.  If we restrict ourselves to
graphons on a subset $\Lambda$ of $\R^d$, we use the notation
$\HOLD(\Lambda)$.

Our first proposition concerns the case when the support of the underlying
measure $\pi$ is compact, in which case we may assume without loss of
generality that $\pi$ is a measure on $\Lambda_R=[-R,R]^d$ for some $R\in
[0,\infty)$.  Note that many examples of $W$-random graphs considered in the
statistics and machine learning literature fit into this setting, e.g., the
mixed membership block model of \cite{MMSB08}.  Note also that while these
models can be mapped onto $W$-random graphs over $[0,1]$ with the uniform
distribution by a measure-preserving map, such a map will typically not do
this in a continuous way.  So if one wants to use continuity properties of
the generating graphon $W$, one has to analyze it on the original space on
which it was defined, not on $[0,1]$.

\begin{prop}
\label{prop:Hoelder-compact} Let $d\geq 1$, $R\in [0,\infty)$, $\alpha\in
(0,1]$, and $C<\infty$, let $\pi$ be a probability measure on
$\Lambda_R\subseteq \R^d$, and let $W$ be a normalized graphon in
$\HOLD(\Lambda_R)$. Then there exists a constant $D$ depending only on $R$,
$C$, and $\alpha$ such that the following holds:

(i) We have $\|W\|_\infty\leq D$. So in particular
\[
\tail_\rho^{(p)}(W)=0 \qquad\text{if $\rho\leq \frac 1D$}.
\]

(ii) For  $p\geq 1$ and  $\kappa>0$,
\begin{equation}
\label{Holder-bd2}
\eps^{(p)}_{\geq\kappa}(W)
\leq 4D\kappa^{\alpha'}.
\end{equation}
where $ \alpha'=\frac \alpha{p\alpha+d} $. If $\pi$ is the uniform measure,
then the bound \eqref{Holder-bd2} holds for $\alpha'=\alpha/d$.
\end{prop}

\begin{proof}
We will prove the proposition for $D=1+2C(2R)^\alpha$.

To prove the first statement, let $C_0=\min_{x,y\in \Lambda_L} W(x,y)$.
Since $C_0=\int C_0\leq \|W\|_1=1$, H\"older continuity implies that
$\|W\|_\infty\leq 1+2C(2R)^\alpha =D$.

To prove the second statement, consider $k\in\N$, and let $\cP$ be the
partition of $\Lambda_R$ into $k^d$ cubes of side-length $a=2R/k$.  For a
given class $Y\in \cP$, two points $x,x'\in Y$ have distance
$|x-x'|_\infty\leq a$. Thus, if $Y$ and $Y'$ are two classes in $\cP$, then
$|W(x,y)-W(x',y')|\leq 2C a^\alpha=2C(2Rk^{-1})^\alpha$ whenever $x,y\in Y$
and $x',y'\in Y'$.  As a consequence $\|W-W_\cP\|_\infty\leq
2C(2Rk^{-1})^\alpha\leq Dk^{-\alpha}$.

If $\pi$ is the uniform measure over $\Lambda_R$, then each class $Y$ of
$\cP$ has measure $\pi(Y)=k^{-d}$, so setting
$k=\lfloor\kappa^{-1/d}\rfloor$, we obtain $k^{-1}\leq 2\kappa^{1/d}$ and
thus
\[
\eps_{\geq\kappa}^{(p)}(W)
\leq 2D \kappa^{\alpha/d}
,
\]
which proves the proposition for the case of the uniform measure.  (Recall
that $\delta_p$ and hence $\eps_{\geq\kappa}^{(p)}(W)$ are decreasing
functions of $p$.)

But for general measures, some of the classes of $\cP$ might have tiny
measure.  To fix this, we merge all classes of measure less than $\kappa$
(where $\kappa$ will now be smaller than $k^{-d}$) with the smallest of
those which have measure at least $\kappa$. Lemma~\ref{lem:small-kappa}
below will show that for $\kappa$ small enough, this actually works. To
apply the lemma, we set $N=k^d$ and observe that $\|W_\cP\|_\infty\leq
\|W\|_\infty\leq D$ and $\|W-W_\cP\|_\infty\leq  D N^{-\alpha/d}$.
Lemma~\ref{lem:small-kappa} then implies that
\[
\eps^{(p)}_{\geq\kappa}(W)\leq 2D N^{-\alpha/d}
\]
provided $ 2\kappa\leq N^{-\frac{p \alpha +d}d} $. Thus for $\kappa\leq
1/2$, we may choose $k=\lfloor({2\kappa})^{-1/(p \alpha  +d)} \rfloor $ to
show that \eqref{Holder-bd2} holds for $\kappa\leq 1/2$. For $\kappa\geq
1/2$, that would amount to $k=0$, but fortunately this case is trivial: the
right side of \eqref{Holder-bd2} is at least $2D$ and hence at least $2$,
while $\eps^{(p)}_{\geq\kappa}(W)\leq 1$ for a normalized graphon, showing
that \eqref{Holder-bd2} holds for $\kappa\geq 1/2$ as well.
\end{proof}

\begin{lem}\label{lem:small-kappa}
Let $W$ be a bounded graphon over some probability space $(\Omega,\cF,\pi)$,
and let $W'$ be a graphon over $(\Omega,\cF,\pi)$ such that $\|W-W'\|_p\leq
\eps$ and $W'$ is a block model with $N$ classes. Then
\[
\eps^{(p)}_{\geq\kappa}(W)\leq 2\eps
\qquad\text{whenever}\qquad
\kappa\leq \frac 1{2N}\Bigl( \frac\eps{\|W'\|_\infty}\Bigr)^p.
\]
\end{lem}

\begin{proof}
Suppose $W'$ is based on the partition $(Y_1,\dots,Y_N)$ of $\Omega$.
Arranging the classes $Y_i$ in $\cP$ in order of decreasing measure,  let
$Y_\ell$ be the last class of measure $\kappa$ or more.  We then define
$Y_\ell'=\bigcup_{i\geq\ell} Y_i$, and $Y_i'=Y_i$ for all $i<\ell$.  Let
$W''$ be a block model with blocks $Y_1',\dots,Y_\ell'$ and the same values
as $W'$ on $Y_i \times Y_j$ when $i,j < \ell$ but the value $0$ when $i$ or
$j$ equals $\ell$.  Clearly $W''\in \cB_{\geq \kappa}$. To prove the
proposition, we will have to show that $\|W'-W''\|_p\leq \eps$. To this end,
we  note that $W''$ and $W'$ agree on $\Omega_0\times\Omega_0$, where
$\Omega_0=Y_1\cup\dots\cup Y_{\ell-1}$, and that $\|W'-W''\|_\infty \leq
\|W'\|_\infty$. As a consequence,
\[
\|W'-W''\|_p=\|(W'-W'')(1-1_{\Omega_0\times\Omega_0})\|_p
\leq \|W'\|_\infty \bigl(1-\pi(\Omega_0)^2\bigr)^{1/p}.
\]
But because the classes $Y_{\ell+1},\dots, Y_N$ have measure smaller than
$\kappa$,
\[
\pi(\Omega_0)\geq 1-\ell\kappa\geq 1-N\kappa,
\]
showing that
\[
\|W'-W''\|_p
\leq \|W'\|_\infty \bigl(2N\kappa\bigr)^{1/p},
\]
which is bounded by $\eps$ if $\kappa\leq \frac 1{2N}(
\eps/{\|W'\|_\infty})^p$.
\end{proof}

In many applications, the underlying measure on the latent position space
$\Omega$ does not have compact support.  Gaussians are a noteworthy case, as
are distributions with heavier tails (such as Student distributions).
Another reason to consider measures without compact support comes from the
desire to model graphs with power-law degree distributions. As discussed
already in Section~\ref{sec:intro-dense-sparse},
bounded graphons do not allow for power-law degree
distributions, showing in particular that H\"older-continuous graphons  over
$\R^d$  equipped with a measure with compact support do not lead to graphs
that exhibit power-law degree distributions.\footnote{Once the assumption of
compact support is removed, this reasoning no longer applies, and as shown
in Section~\ref{sec:power-law}, there are indeed H\"older-continuous
graphons over $\R^d$ which generate graphs with power-law degree
distributions.} For all these reasons, we aim for a generalization of
Proposition~\ref{prop:Hoelder-compact} to measures whose supports are not
necessarily compact.

Since we want graphons to be integrable (in fact, for the least squares
algorithm to be consistent, we need them to be square integrable) we will
restrict ourselves to probability distributions $\pi$ over $\R^d$ in
\[
\cM_\beta=\Big\{\pi \,\Big|\,\int_{\R^d}|x|_\infty^{\beta} \, d\pi(x)<\infty\Big \},
\]
where $\beta>0$ is a parameter which we will choose to be at least $\alpha$
(or at least $2\alpha$ when we want to guarantee that the graphons in
$\HOLD$ are in $L^2$).

\begin{prop}\label{prop:Hoelder}
Let $d\geq 1$ and $\beta\geq\alpha>0$, let $\pi\in \cM_\beta$, and let $W$
be an $\alpha$-H\"older-continuous graphon over $\R^d$ equipped with the
probability distribution $\pi$, normalized in such a way that $\|W\|_1=1$.
If $1\leq p<\beta/\alpha$ and $\kappa \le 1/2$, then
\[
\eps_{\geq\kappa}^{(p)}(W)=O\big(\kappa^{\alpha'}\big)
\qquad\text{and}\qquad
\tail_\rho^{(p)}(W)=O\big(\rho^{{\beta'}}\big),
\]
where ${\beta'}=\frac\beta{p\alpha}-1$ and $\alpha'=\frac {\alpha}{p\alpha
+d}\frac{\beta'}{1+{\beta'}}$, and the constants implicit in the big-$O$
symbols depend on the distribution $\pi$ and the constants $\alpha$,
$\beta$, $p$, and $C$.
\end{prop}

\begin{proof}
Let $R_0\geq 1$ be such that $\pi(\Lambda_{R_0})\geq 1/2$, and let $D_0 =4
+2CR_0^\alpha$.  Then
\[
\min_{x,y\in \Lambda_{R_0}}W(x,y) \leq \frac
1{\pi(\Lambda_{R_0})^2}\int_{\Lambda_{R_0}\times\Lambda_{R_0}}W \leq \frac
{\|W\|_1}{\pi(\Lambda_{R_0})^2}\leq 4.
\]
Denoting the minimizer of $W$ in $\Lambda_{R_0} \times \Lambda_{R_0}$ by $(x_0,y_0)$, we then
have $W(0,0)\leq 4+C|x_0|_\infty^\alpha+C|y_0|_\infty^\alpha\leq 4 +2CR_0^\alpha$,
implying that
\[
W(x,y)\leq D_0+C|x|_\infty^\alpha+C|y|_\infty^\alpha
\]
for all $x,y\in \R^d$. It will be convenient to introduce the functions
$f(x,y)=C|x|_\infty^\alpha$ and $g(x,y)=C|y|_\infty^\alpha$ and write this inequality as
\[
W\leq D_0+f+g.
\]
By our definition of ${\beta'}$ and our assumption on $\pi$,
\[
\|f\|_{p(1+{\beta'})}=C\Bigl(\int |x|_\infty^\beta\, d\pi(x)\Bigr)^{\frac 1{p(1+{\beta'})}}<\infty.
\]
To prove the bound on $ \tail_\rho^{(p)}(W)$, we observe that $0\leq
W-W_\rho\leq W 1_{W\geq 1/\rho}$.  As a consequence,
\begin{align*}
\tail_\rho^{(p)}(W)
&\leq
\| W 1_{W\geq 1/\rho}\|_p
\leq
\rho^{{\beta'}}\| W^{1+{\beta'}} \|_p
=\rho^{{\beta'}}\| W \|_{p(1+{\beta'})}^{1/(1+{\beta'})}
\\
&\leq
\rho^{{\beta'}}\Bigl(\| D_0+f+g \|_{p(1+{\beta'})}\Bigr)^{1/(1+{\beta'})}
\\
&\leq
\rho^{{\beta'}}\Bigl(D_0+\| f\|_{p(1+{\beta'})}+\| g\|_{p(1+{\beta'})}\Bigr)^{1/(1+{\beta'})}
\leq D\rho^{\beta'}
\end{align*}
for some constant $D$ depending on  $\alpha$, $\beta$, $p$, and $C$, as well
as the measure $\pi$ (via $R_{0}$ and the norm $\|f\|_{p(1+{\beta'})}$).

To prove the bound on the oracle error, we want to construct a good block
model approximation to $W$. To this end, we first bound the contributions to
$\|W\|_p$ that come from points $x,y$ outside a box $\Lambda_R$, where
$R\geq 1$ will be chosen later.  If we set $r=CR^\alpha$, then the condition
$(x,y)\notin \Lambda_R\times\Lambda_R$ implies $|x|_\infty>R$ or $|y|_\infty>R$ and hence
$f+g>r$.  But
\[
\|W 1_{f+g>r}\|_p\leq r^{-{\beta'}}\|(f+g)^{\beta'} W\|_p
\leq r^{-{\beta'}}\|(D_0+f+g)^{{\beta'}+1}\|_p\leq Dr^{-{\beta'}},
\]
and hence
\begin{equation}
\label{Lp-Holder-bd}
\|W 1_{(\R^d \times \R^d)\setminus(\Lambda_R\times\Lambda_R)}\|_p\leq D_1R^{-{\beta'}\alpha},
\end{equation}
as long as $D_1$ is chosen so that $D_1\geq D C^{-{\beta'}}$.

Next we consider a partition $\cP=(Y_1,\dots,Y_N)$ of $\Lambda_R$ into cubes
of side length $2R/k$, with $N=k^d$.  We define $\beta_{ij}$ to be the
average of $W$ over ${Y_i\times Y_j}$, and
\[
W'=\sum_{i,j=1}^N \beta_{ij} 1_{Y_i\times Y_j}.
\]
Since $W'$ is composed of parts obtained by averaging over subsets in
$\Lambda_R$, where $W$ is bounded by $D_0+2CR^\alpha\leq D_0(1+R^\alpha)\leq
2D_0R^\alpha$, we have
\[
\|W'\|_\infty\leq 2D_1R^\alpha,
\]
provided $D_1$ is chosen to be at least $D_0$.

Inside $\Lambda_R\times\Lambda_R$, we bound $|W(x,y)-W'(x,y)|$ by
\[
2C (2R/k)^\alpha=2C(2R)^\alpha N^{-\alpha/d} \leq D_1R^\alpha
N^{-\alpha/d},
\]
where $D_1=\max\{D_0, DC^{-{\beta'}},2C2^\alpha\}$. Finally $W-W'=W$ outside
$\Lambda_R$. Combined with the bound \eqref{Lp-Holder-bd}, we conclude that
\[
\|W-W'\|_p\leq \eps,
\]
where $\eps=D_1(R^\alpha N^{-\alpha/d}+R^{-{\beta'}\alpha})$. With the help
of Lemma~\ref{lem:small-kappa} we conclude that
\[
\eps^{(p)}_{\geq\kappa}(W)\leq 2D_1(R^\alpha N^{-\alpha/d}+R^{-{\beta'}\alpha}),
\]
provided that
\[
\kappa\leq \frac 1{2N}\Bigl( \frac{R^\alpha N^{-\alpha/d}+R^{-{\beta'}\alpha}}{2R^\alpha}\Bigr)^p
=\frac 1{2N}\Bigl( \frac{N^{-\alpha/d}+R^{-({\beta'}+1)\alpha}}2\Bigr)^p
\]
and $R\geq 1$. Choosing $R=N^{\frac 1{d({\beta'}+1)}}$, we find that
\[
\eps^{(p)}_{\geq\kappa}(W)\leq 4D_1N^{-\frac{{\beta'}\alpha}{d(1+{\beta'})}},
\]
provided that $\kappa\leq \frac 1{2}N^{-\frac{p\alpha +d}d}$.

Because $\kappa\leq 1/2$, we can choose $k=\Bigl\lfloor\Bigl(\frac
1{2\kappa}\Bigr)^{\frac 1{p\alpha +d}}\Bigr\rfloor$.  Then $N=k^d$ implies
\[
\frac{1}{2^d}\left(\frac{1}{2\kappa}\right)^{\frac{d}{p\alpha+d}} \le N \le \left(\frac{1}{2\kappa}\right)^{\frac{d}{p\alpha+d}}.
\]
This yields a bound of
\[
\eps^{(p)}_{\geq\kappa}(W)\leq 4 D_1 \left(2^d (2\kappa)^{\frac{d}{p\alpha+d}}\right)^{\frac{\beta'\alpha}{d(1+\beta')}},
\]
which is $O\big(\kappa^{\alpha'}\big)$.  Again the implicit constant depends
only on $\alpha$, $\beta$, $p$, $C$, and $\pi$.
\end{proof}

\section{Power-law graphs}
\label{sec:power-law}

Recall that the normalized degree distribution of a graph $G$ on $[n]$ is
defined as the empirical distribution of the normalized degrees $d_i/\bar
d$, where $\bar d$ is the average degree.  We say that a sequence $(G_n)_{n
\ge 0}$ has \emph{convergent degree sequences} if the cumulative
distribution functions $D_{G_n}$ of the normalized degrees converge to some
distribution function\footnote{That is, a non-decreasing, right-continuous
function $D\colon \R\to [0,1]$ such that
$\lim_{\lambda\to-\infty}D(\lambda)=0$ and
$\lim_{\lambda\to\infty}D(\lambda)=1$.} $D$ in the L\'evy-Prokhorov distance
$d_{\LP}$ or, equivalently, if $D_{G_n}(\lambda)\to D(\lambda)$ for all
$\lambda$ at which $D$ is continuous.

We say that the sequence $(G_n)_{n \ge 0}$ has a \emph{power-law degree
distribution with exponent $\gamma$} if its degree distributions converge to
$D$ satisfying
\[
D(\lambda)=1-\Theta\big(\lambda^{-(\gamma-1)}\big)\qquad\text{as $\lambda\to\infty$},
\]
and we say that a graphon $W$ has a \emph{power-law degree distribution with
exponent $\gamma$} if $D_W=1-\Theta\big(\lambda^{-(\gamma-1)}\big)$ as
$\lambda\to\infty$.

Note that it is $\gamma-1$ that appears in the exponent, not $\gamma$.  The
naming conventions in the above definitions are based on density functions,
rather than distribution functions: if the degree distribution is absolutely
continuous with respect to Lebesgue measure and thus has a density function
$f(\lambda)$, and if $f(\lambda) = \Theta\big(\lambda^{-\gamma}\big)$ as
$\lambda \to \infty$, then the distribution function $D$ satisfies
\[
1-D(\lambda) = \int_\lambda^\infty f(\lambda) \, d\lambda = \Theta\big(\lambda^{-(\gamma-1)}\big).
\]

In this section, we give two examples of $W$-random graphs with power-law
degree distributions and establish bounds on the convergence rate of our
estimation procedures for these graphons.

We start with an example that can be expressed as a H\"older-continuous
graphon over $\R^d$, even though we will first define it as a graphon over
$[0,1]$. It is the graphon
\begin{equation}
\label{ex:power_x+y}
W(x,y)=\frac 12(g(x)+g(y))\quad\text{where $g(x)=(1-\alpha) (1-x)^{-\alpha}$}.
\end{equation}
for some $\alpha \in (0,1)$.  Note that the degrees of this graphon are
$W_x=\frac 12 +\frac 12 g(x)$, with a distribution function $D_W(\lambda)$
that goes to $1$ like $1-\Theta\big(\lambda^{-1/\alpha}\big)$ as
$\lambda\to\infty$, showing that the graphs $G_n(\rho_nW)$ have a power-law
degree distribution with exponent $\gamma=1+\frac 1\alpha$.

As a graphon over $[0,1]$ equipped with the uniform measure, this graphon is
not continuous, but it turns out that it can be expressed as an equivalent
graphon over $\R^d$ that is H\"older-continuous.  To see this, let us
consider a probability distribution $\pi$ on $\R^d$ such that the
distribution of the $L^2$ norm $r=|x|_2$ of $x\in \R^d$ is absolutely
continuous with respect to the Lebesgue measure on $[0,\infty)$, with a
strictly positive density function $f(r)$. We will want to construct a
measure-preserving map $\phi\colon \R^d\to [0,1)$ to obtain an equivalent
graphon $W^\phi$ over $\R^d$. Requiring $\phi$ to be measure preserving is
equivalent to requiring that $\pi(\phi^{-1}([0,a]))=\pi(\{x\colon
\phi(x)\leq a\})=a$. We will construct $\phi$ radially, via a map $F$ such
that $\phi(x) = F(|x|_2)$, and we will make sure that $F$ is strictly
increasing, in which case $\phi(x)\leq a$ is equivalent to $|x|_2\leq
F^{-1}(a)$.  Thus, our condition for $\phi$ to be measure preserving becomes
$a=\int 1_{|x|_2\leq F^{-1}(a)}d\pi(x)$, or equivalently, $\int 1_{|x|_2\leq
r}d\pi(x)=F(r)$, showing that $F(r)$ is the cumulative distribution function
of $|x|_2$ (which is strictly monotone by our assumption that $f(r)>0$ for
all $r\in [0,\infty)$). Taking $F(r)=1-\frac 1{r+1}$, we get
\begin{align*}
W^\phi(x,y)&=\frac {1-\alpha}2\Bigl(\frac1{1-F(|x|_2)}\Bigr)^\alpha
+\frac {1-\alpha}2\Bigl(\frac1{1-F(|y|_2)}\Bigr)^\alpha\\
&=\frac {1-\alpha}2\bigl((1+|x|_2)^\alpha+(1+|y|_2)^\alpha\bigr),
\end{align*}
showing that $W$ is equivalent to an $\alpha$-H\"older-continuous graphon
over $\R^d$ equipped with any measure for which the cumulative distribution
function of $|x|_2$ is equal to $F$.  As a consequence, we may use the
results of Section~\ref{sec:Hoelder} to give explicit bounds on the
estimation errors for the least squares and least cut algorithms.  We will
not give these bounds here, since for $W$ of the form \eqref{ex:power_x+y},
one can obtain slightly better bounds using the actual form of $W$; see
Lemma~\ref{lem:Power-Block} below.

The second example we consider in this section is the graphon $W$ over
$[0,1]$ that is defined by
\begin{equation}
\label{ex:power_xy}
W(x,y)=g(x)g(y)\quad\text{where again $g(x)=(1-\alpha) (1-x)^{-\alpha}$.}
\end{equation}
As before, we equip $[0,1]$ with the uniform measure.  Now the degrees of
$W$ are equal to $g(x)$, which shows that again, the $W$-random graphs
obtained from $W$ have power-law degrees with exponent $\gamma=1+\frac
1\alpha$.

Note that the second graphon cannot be expressed as a H\"older-continuous
graphon over $\R^d$ in the sense of Section~\ref{sec:Hoelder}.  Indeed,
suppose $\tilde W$ were such a graphon. By Theorem~\ref{thm:equiv-over-01},
there would exist a standard Borel twin-free graphon $U$ such that $\tilde
W=U^\phi$ for some measure-preserving map $\phi$ from $\R^d$ to the space on
which $U$ is defined. Since $W$ is twin-free as well we may without loss of
generality assume that $U=W$ (use Theorem~\ref{thm:twinfree-equiv}).  But
this means that $\tilde W$ would be of the form $\tilde
W(x,y)=W(\phi(x),\phi(y)) =g(\phi(x))g(\phi(y))$ for some measure-preserving
map $\phi\colon \R^d\to [0,1]$. Since $g(\phi(x))$ is unbounded, this cannot
be a H\"older-continuous function of the argument $y$.

Nevertheless, we can give explicit bounds on our estimation error since for
$W$ of the form \eqref{ex:power_x+y} or \eqref{ex:power_xy}, we can estimate
$\eps_{\geq\kappa}^{(p)}(W)$ and $\tail_\rho^{(p)}(W)$ directly.

\begin{lem}
\label{lem:Power-Block} Let $\alpha\in (0,1)$, let $1\leq p<1/\alpha$, and
define $\alpha'=\frac 1p -\alpha$ and $\beta'=\frac{1-p\alpha}{p\alpha}$. If
$W$ is the power-law graphon \eqref{ex:power_x+y}, then
\[
\eps_{\geq\kappa}^{(p)}(W)=O\big(\kappa^{\alpha'}\big)
\qquad\text{and}\qquad
\tail_\rho^{(p)}(W)=O\big(\rho^{{\beta'}}\big),
\]
while if $W$ is the power-law graphon \eqref{ex:power_xy}, then
\[
\eps_{\geq\kappa}^{(p)}(W)=O\big(\kappa^{\alpha'}\big)
\qquad\text{and}\qquad
\tail_\rho^{(p)}(W)=O\big(\rho^{{\beta'}}|\log \rho|^{1/p}\big).
\]
\end{lem}

\begin{proof}
We start with the proof of the tail bounds.  Defining $g_1,g_2\colon
[0,1]^2\to[0,\infty)$ by $g_1(x,y)= g(x)$ and $g_2(x,y)=g(y)$, we write the
first graphon as $\frac 12(g_1+g_2)$.  Noting that $W\geq \rho^{-1}$ implies
that either $g_1\geq 1/\rho$ or $g_2\geq 1/\rho$, we bound
\begin{align*}
\|W-W_\rho\|_p
&\leq \|W1_{W\geq 1/\rho}\|_p
\leq \|W(1_{\rho g_1\geq 1}+1_{\rho g_1\geq 1})\|_p
\\
&=\frac 12\|g_11_{\rho g_1\geq 1} +g_21_{\rho g_1\geq 1}
+g_11_{\rho g_2\geq 1} +g_21_{\rho g_2\geq 1}\|_p\\
&\leq \|g1_{\rho g\geq 1}\|_p+\|1_{\rho g\geq 1}\|_p.
\end{align*}
The two terms can easily be calculated explicitly, giving a term of order
$O\big(\rho^{\frac{1-p\alpha}{p\alpha}}\big)$ for the first and a term of
order $O\big(\rho^{\frac{1}{p\alpha}}\big)$ for the second. For the second
graphon, we note that the condition $W(x,y)\geq 1/\rho$ is equivalent to
$(1-x)(1-y)\leq \bigl(\rho(1-\alpha)^2\bigr)^{1/\alpha}$. Changing to the
variables $1-x$ and $1-y$, we have to estimate the integral
\[
\int_0^1\int_0^1 (xy)^{-p\alpha}1_{xy\leq \rho^{1/\alpha}} \, dx \, dy.
\]
The integral can again be calculated explicitly, giving an error term of
order $O\big(\rho^{\frac{1-p\alpha}\alpha}|\log\rho|\big)$. Taking the
$p^{\text{th}}$ root, we obtain the claimed tail bound for the second
graphon.

All that remains is to estimate the oracle errors. Let $I_1,\dots,I_k$ be a
partition of $[0,1]$ into $k$ adjacent intervals of size $\eps=\frac 1k$
(ordered from left to right), let $g'$ be the function obtained by averaging
$g$ over these intervals on $I_{1}\cup I_2\dots\cup I_{k_0}$ (where $k_0$
will be determined later), and let $g'=0$ on the remaining intervals. Define
$g_1,g_2\colon [0,1]^2\to[0,\infty)$ as above, define $g_1'$ and $g_2'$
analogously, and set $W'=\frac 12(g_1'+g_2')$ for the graphon
\eqref{ex:power_x+y} and $W'=g_1'g_2'$ for the graphon \eqref{ex:power_xy}.
With this notation,
\[
\|W-W'\|_p=\frac 12 \|g_1+g_2-g_{1}'-g_{2}'\|_p
=\|g-g'\|_p
\]
for the graphon \eqref{ex:power_x+y}, and
\[
\|W-W'\|_p= \|g_1g_2-g_{1}'g_{2}'\|_p
\leq\|(g_1-g_{1}')g_2\|_p +\|g_{1}'(g_2-g_{2}')\|_p
\leq \|g\|_p\|g-g'\|_p
\]
for the graphon \eqref{ex:power_xy}.  So all we need to show is that
$\|g-g'\|_p=O\big(\eps^{\alpha'}\big)$.

For $i\leq k_0$, let $\bar x_i\in I_i$ be defined by $\frac
1\eps\int_{I_i}g=g(\bar x_i)$. For $x\in I_i$,  we bound $|g(x)-g(\bar
x_i)|\leq \max_{y\in I_i} \Bigl|\frac {dg(y)}{dy} \Bigr| |x-\bar x_i|$,
implying that the integral of $|g(x)-g(\bar x_i)|^{p}$ over $I_i$ can be
bounded by $\eps^{{p}+1}\max_{y\in I_i} \Bigl|\frac
{dg(y)}{dy}\Bigr|^{p}\leq \eps^{{p}+1}(1-i\eps)^{-{p}(1+\alpha)}$. Summing
over $i=1,\dots, k_0$, we get a contribution of
$O\big(\eps^p(1-k_0\eps)^{1-p(1+\alpha)}\big)$ to $\|g-g'\|_p^p$.  The
integral of $g^p$ from $k_0\eps$ to $1$ will contribute
$O\big((1-k_0\eps)^{1-\alpha p}\big)$.  As a consequence, the choice
$k_0=k-1$ (which gives $1-k_0\eps=\eps$) leads to the estimate
\[
\|g-g'\|_p^p=O\big(\eps^{1-\alpha p}\big),
\]
as desired.
\end{proof}

\section*{Acknowledgments}

We thank David Choi, Sofia Olhede, and Patrick Wolfe for initially
introducing us to applications of graphons in machine learning of networks,
and, in particular, to the problem of graphon estimation. We are indebted to
Sofia Olhede and Patrick Wolfe for numerous helpful discussions in the early
stages of this work, and to Alessandro Rinaldo for providing valuable feedback
on our paper.

\bibliographystyle{amsalpha}

\appendix

\section{Couplings, metrics, and equivalence}
\label{sec:couplings}

We start this appendix by reformulating Remark~\ref{rem:delta-over-01} in
the more general setting of Borel spaces.

\begin{lem}\label{lem:delta-over-01}
Let $p\geq 1$ and let $W$ and $W'$ be $L^p$ graphons over two Borel spaces
$(\Omega,\cF,\pi)$ and $(\Omega',\cF',\pi')$. Then the following hold:

(i) The infima in \eqref{del-p} and \eqref{cut-distance} are attained for
some couplings $\nu$.

(ii) If $(\Omega,\cF,\pi)$ and $(\Omega',\cF',\pi')$ are atomless, then the
distances $\delta_p(W,W')$ and $\delta_{\square}(W,W')$ can be expressed as
\[
\delta_p(W,W')=\inf_{\phi}\|W-(W')^\phi\|_p
=\inf_{\Phi}\|W-(W')^\Phi\|_p
\]
and
\[
\delta_\square(W,W')=\inf_{\phi}\|W-(W')^\phi\|_\square
=\inf_{\Phi}\|W-(W')^\Phi\|_\square,
\]
where the infima over $\phi$  re over measure-preserving maps from $\Omega$
to $\Omega'$ and the infima over $\Phi$ are over isomorphisms from $\Omega$
to $\Omega'$.
\end{lem}

For the cut metric, the first statement is a special case of Theorem~6.16 in
\cite{J13} (see also Lemma 2.6 in \cite{BR}, which proves the statement for
bounded graphons over $[0,1]$), while the second is
essentially\footnote{While Lemma~3.5 in \cite{BCLSV08} was only stated for
bounded graphons over $[0,1]$, the generalization to unbounded graphons over
an atomless Borel space is straightforward.} given in Lemma~3.5 in
\cite{BCLSV08}. The proofs for the distance $\delta_p$ are virtually
identical. For the convenience of the reader, we sketch them below.

Note that the first statement does not hold without the assumption that
$(\Omega,\cF,\pi)$ and $(\Omega',\cF',\pi')$ are Borel spaces; see, for
example, Example~8.13 in \cite{J13} for a counterexample.  Similarly, the
assumption that $(\Omega,\cF,\pi)$ and $(\Omega',\cF',\pi')$ are atomless is
needed for the second statement to hold; see Remark~6.10 in \cite{J13}.
(Indeed, the condition involving $\Phi$ does not even make sense unless
$\Omega$ and $\Omega'$ are isomorphic, but all atomless Borel spaces are
isomorphic by Theorem~A.7 in \cite{J13}.  For arbitrary probability spaces
there may not even be any measure-preserving maps from $\Omega$ to
$\Omega'$.)

\begin{proof}
We begin with part (i). For the cut metric, this is a special case of
Theorem~6.16 in \cite{J13}. The proof for the metric $\delta_p$ is very
similar.  For the convenience of the reader, we give the proof below,
combining proof techniques from \cite{J13} and \cite{BR}.

Let $\cM$ be the set of all probability measures on $\Omega \times \Omega'$
for which the marginals are $\pi$ and $\pi'$.  We first observe that $\cM$
is compact in the weak* topology.  To see why, first note that by Theorem
A.4(iv) in \cite{J13}, the measurable spaces $(\Omega,\cF)$ and
$(\Omega',\cF')$ are either countable (with all subsets measurable) or
isomorphic to $[0,1]$ with the Borel $\sigma$-algebra. Let $\cA_0$ be the
set of all $A\subseteq \Omega \times \Omega'$ that are products of intervals
with rational endpoints in the $[0,1]$ case and finite sets in the countable
case. Since $\cA_0$ is countable, any sequence of measures $\nu_n\in\cM$ has
a subsequence $\nu_n'$ such that $\nu_n'(A)$ converges for all $A\in\cA_0$.
Since $\cA_0$ generates the product $\sigma$-algebra on $\Omega \times
\Omega'$, the limit can be extended to a probability measure $\mu$ on
$\Omega \times \Omega'$, which can easily by checked to have $\pi$ and
$\pi'$ as marginals, implying that $\mu\in \cM$.

Consider a sequence of couplings $\nu_n$ such that
\begin{equation}
\label{del-p-limit}
\delta_p(W,W')=\lim_{n\to\infty}
\paren{\int \Bigl|W(x,y)-W'(x',y')\Bigr|^p \,d\nu_n(x,x')\,d\nu_n(y,y')}^{1/p}.
\end{equation}
By the compactness of $\cM$, we may pass to a subsequence (which we again
denote by $\nu_n$) for which there is a limit $\nu\in\cM$ such that
$\nu_n(A)\to \nu(A)$ for all $A\in\cA_0$.  Since $\nu\in\cM$,
\[
\delta_p(W,W')
\leq
\paren{\int \Bigl|W(x,y)-W'(x',y')\Bigr|^p \,d\nu(x,x')\,d\nu(y,y')}^{1/p}.
\]

To prove a matching lower bound we fix $\eps>0$ to be sent to zero later. By
\eqref{del-p-limit}, we can find an $n_0$ such that
\[
\delta_p(W,W')\geq
\paren{\int \Bigl|W(x,y)-W'(x',y')\Bigr|^p \,d\nu_n(x,x')\,d\nu_n(y,y')}^{1/p}-\eps.
\]
for all $n\geq n_0$. Since $W\in L^p$, we can find an $M$ such that
$\|W1_{W\geq M}\|_p\leq \eps$, and since $W1_{W<M}$ is bounded, we can find
a graphon $\tilde W$ which is a finite sum of the form $\tilde
W=\sum_{i,j}\beta_{i,j}1_{A_i\times A_j}$ with $A_i\in\cA_0$ such that
$\|W1_{W<M}-\tilde W\|_p\leq\eps$, implying in particular $\|W-\tilde
W\|_p\leq 2\eps$.  In a similar way, we can find $\tilde W'$ of the form
$\tilde W'=\sum_{k,\ell}\beta'_{k,\ell}1_{B_k\times B_\ell}$ with
$B_i\in\cA_0$ and $\|W'-\tilde W'\|_p\leq 2\eps$.  As a consequence
\begin{align*}
\delta_p(W,W')&\geq
\paren{\int \Bigl|\tilde W(x,y)-\tilde W'(x',y')\Bigr|^p \,d\nu_n(x,x')\,d\nu_n(y,y')}^{1/p}-5\eps
\\
&=
\paren{\sum_{i,j,k,\ell}|\beta_{ij}-\beta'_{k\ell}|^p
\nu_n(A_i\times B_k)\,\nu_n(A_j\times B_\ell)}^{1/p}-5\eps
\end{align*}
for all $n\geq n_0$.  We can take the limit as $n\to\infty$ on the right
side, to obtain the bound
\begin{align*}
\delta_p(W,W')&\geq
\paren{\sum_{i,j,k,\ell}|\beta_{ij}-\beta'_{k\ell}|^p
\nu(A_i\times B_k)\,\nu(A_j\times B_\ell)}^{1/p}-5\eps
\\
&=
\paren{\int \Bigl|\tilde W(x,y)-\tilde W'(x',y')\Bigr|^p \,d\nu(x,x')\,d\nu(y,y')}^{1/p}-5\eps
\\
&\geq \paren{\int \Bigl|W(x,y)- W'(x',y')\Bigr|^p \,d\nu(x,x')\,d\nu(y,y')}^{1/p}-9\eps.
\end{align*}
Since $\eps$ was arbitrary, this proves part (i) of the lemma.

We now turn to part (ii).  All atomless Borel spaces are isomorphic to
$[0,1]$ (with the Borel $\sigma$-algebra and uniform distribution), by
Theorem~A.7 in \cite{J13}.  Thus, we can assume without loss of generality
that $\Omega$ and $\Omega'$ are both $[0,1]$.

Choosing $z$ uniform at random from $[0,1]$, the map $z\mapsto (z,\phi(z))$
provides a coupling showing that
$\delta_p(W,W')\leq\inf_{\phi}\|W-(W')^\phi\|_p$ and
$\delta_\square(W,W')\leq\inf_{\phi}\|W-(W')^\phi\|_\square$. It is also
obvious that $\inf_{\phi}\|W-(W')^\phi\|_p \leq\inf_{\Phi}\|W-(W')^\Phi\|_p$
and $\inf_{\phi}\|W-(W')^\phi\|_\square
\leq\inf_{\Phi}\|W-(W')^\Phi\|_\square$.

To prove equality, one first approximates $W$ and $W'$ by piecewise constant
functions (more precisely, graphons on $[n]$ equipped with the uniform
measure), and then approximates an arbitrary coupling of two uniform
measures on $[n]$ by a bijection on a ``blow-up'' $[nk]$ of $[n]$.  Mapping
this bijection back to an isomorphism $\Phi\colon[0,1]\to [0,1]$  then gives
a lower bound on $\delta_p(W,W')$ in terms of $\inf_{\Phi}\|W^\Phi-W'\|_p$,
minus some error which can be taken to be arbitrarily small. The details are
very similar to the proof of Lemma~3.5 in \cite{BCLSV08}, which proves
equality for the cut norm when $W$ and $W'$ are bounded, and we leave them
to the reader. Note that the generalization to unbounded graphons is
straightforward, given that  $\|W 1_{W\geq M}\|_p\to 0$ as $M\to\infty$ and
$\|W 1_{W\geq M}\|_\square\leq \|W 1_{W\geq M}\|_1$.
\end{proof}

In the remainder of this appendix, we prove most of the theorems from
Section~\ref{sec:identify}.  We rely heavily on both the results and the
techniques of \cite{BCL10} and \cite{J13}; see also \cite{BR}. Before
turning to these proofs, we relate the notion of equivalence from
Definition~\ref{def:equiv} to the notion of ``weak isomorphism'' from
\cite{BCL10}, which requires the maps $\phi$ and $\phi'$ to be measure
preserving with respect to the completion of the spaces $(\Omega,\cF,\pi)$
and $(\Omega',\cF',\pi')$. It is clear that equivalence implies weak
isomorphism, since maps that are measurable with respect to
$(\Omega,\cF,\pi)$ and $(\Omega',\cF',\pi')$ are clearly measurable with
respect to their completions. We can also turn this around, at least when
the third space is a \emph{Lebesgue space}, i.e., the completion of a Borel
space. This follows from part (i) of the following technical lemma.

\begin{lem}
\label{lem:equiv-weak-iso} Let $W$ and $W'$ be  graphons over two
probability spaces $(\Omega,\cF,\pi)$ and $(\Omega',\cF',\pi')$,
respectively.

(i) Assume that there exist measure-preserving maps $\phi$ and $\phi'$ from
the completions of $(\Omega,\cF,\pi)$ and $(\Omega',\cF',\pi')$ to a
Lebesgue space $(\Omega'',\cF'',\pi'')$ and a graphon $U$ over
$(\Omega'',\cF'',\pi'')$ such that $W=U^\phi$ and $W'=U^{\phi'}$ almost
everywhere. Then there exists a standard Borel graphon $\tilde U$ and
measure-preserving maps $\tilde \phi$ and $\tilde\phi'$ from
$(\Omega,\cF,\pi)$ and $(\Omega',\cF',\pi')$ to the Borel space on which
$\tilde U$ is defined such that $W=\tilde U^{\tilde \phi}$ and $W'=\tilde
U^{\tilde\phi'}$ almost everywhere.  If $U$ is twin-free, then $\tilde U$
can be chosen to be twin-free as well.

(ii) If $(\Omega,\cF,\pi)$ and $(\Omega',\cF',\pi')$ are Borel spaces and
$W$ and $W'$ are isomorphic modulo $0$ when considered as graphons over the
completion of $(\Omega,\cF,\pi)$ and $(\Omega',\cF',\pi')$, then they are
also isomorphic modulo $0$ as graphons over $(\Omega,\cF,\pi)$ and
$(\Omega',\cF',\pi')$.
\end{lem}

\begin{proof}
(i) Since every Lebesgue space is isomorphic modulo $0$ to the union of an
interval $[0,p]$ and a collection of atoms $x_i$ (see Theorem~A.10 in
\cite{J13}), we may without loss of generality assume that
$(\Omega'',\cF'',\pi'')$ is of this form. Assume without loss of generality
that the atoms are represented as points $x_i\in (p,1]$, so that $\phi$
takes values in $[0,1]$.  Noting that $\cF''$ is the completion of a Borel
$\sigma$-algebra $\cB''$, define $\tilde U$ as the conditional expectation
$\E[U\mid \cB''\times\cB'']$. Then $\tilde U$ is a Borel graphon  such that
$U=\tilde U$ almost everywhere. Since $\phi$ is measure preserving from the
completion $(\Omega,\bar\cF,\pi)$ of $(\Omega,\cF,\pi)$ to
$(\Omega'',\cF'',\pi'')$, it is also measure preserving from
$(\Omega,\bar\cF,\pi)$ to $(\Omega'',\cB'',\pi'')$. Replacing $\phi$ by the
conditional expectation $\tilde\phi=\E[\phi\mid \cF]$, we obtain a
measure-preserving map $\tilde\phi$ from $(\Omega,\cF,\pi)$ to
$(\Omega'',\cB'',\pi'')$ such that $W=\tilde U^{\tilde\phi}$ almost
everywhere.  If $U$ is twin-free, then so is $\tilde U$.

(ii) The completions of $(\Omega,\cF,\pi)$ and $(\Omega',\cF',\pi')$ are
Lebesgue spaces.  Since every Lebesgue space is isomorphic modulo $0$ to the
disjoint union of an interval $[0,p]$ (equipped with the Lebesgue
$\sigma$-algebra and the uniform measure) and countably many atoms $x_i$, we
have that as graphons over the completion of $(\Omega,\cF,\pi)$ and
$(\Omega',\cF',\pi')$, both $W$ and $W'$ are isomorphic modulo $0$ to a
graphon $U$ over such a space. Proceeding as in the proof of (i), we can
then replace $U$ by a Borel graphon $\tilde U$ such that $W$ and $W'$ are
isomorphic modulo $0$ to the graphon $\tilde U$, which in particular implies
that $W$ and $W'$ are isomorphic modulo $0$.
\end{proof}

\begin{proof}[Proof of Theorem~\ref{thm:twinfree-equiv}]
If $W$ and $W'$ are isomorphic modulo $0$, they are clearly equivalent.
Assume on the other hand that $W$ and $W'$ are equivalent.  Moving from
$(\Omega,\cF,\pi)$ and $(\Omega',\cF',\pi')$ to their completion, we obtain
graphons which are  defined on a Lebesgue space and are weakly isomorphic in
the sense of \cite{BCL10}. For bounded graphons, we can then  use Theorem
2.1 of  \cite{BCL10} to conclude that $W$ and $W'$ are isomorphic modulo $0$
as graphons over the completion of $(\Omega,\cF,\pi)$ and
$(\Omega',\cF',\pi')$. By Lemma~\ref{lem:equiv-weak-iso}, this implies that
they are also isomorphic modulo $0$ as graphons over $(\Omega,\cF,\pi)$ and
$(\Omega',\cF',\pi')$.

If $W$ and $W'$ are unbounded, let $\widetilde W=\tanh W$ and $\widetilde
W'=\tanh W'$. Clearly,  $W$ and $W'$ are equivalent if and only if
$\widetilde W$ and $\widetilde W'$ are equivalent, and $W$ and $W'$ are
isomorphic modulo $0$ if and only if $\widetilde W$ and $\widetilde W'$ are
isomorphic modulo $0$.  Therefore the unbounded case follows from the
bounded case.
\end{proof}

\begin{proof}[Proof of Theorem~\ref{thm:equiv-over-01}]
For bounded graphons, the analogous statement for graphons over a Lebesgue
space was proven in \cite{BCL10}; in particular, by Corollary 3.3 from
\cite{BCL10}, we can find a twin-free graphon $U$ over a Lebesgue space
$(\Omega',\cF',\pi')$ and a measure-preserving map $\phi$ from the
completion of $(\Omega,\cF,\pi)$ to $(\Omega',\cF',\pi')$ such that
$W=U^\phi$ almost everywhere. By Lemma~\ref{lem:equiv-weak-iso}, this
implies the existence of a twin-free standard Borel graphon $\tilde U$ on a
Borel space $(\tilde\Omega,\tilde \cF,\tilde\pi)$ and a measure-preserving
map from $(\Omega,\cF,\pi)$ to $(\tilde\Omega,\tilde \cF,\tilde\pi)$ such
that $W=\tilde U^{\tilde \phi}$ almost everywhere, which proves (ii) for
bounded graphons. Statement (i) follows from (ii) by expanding the atoms
$x_i$ in $\tilde\Omega$ into intervals of widths $p_i=\tilde\pi(x_i)$.

To reduce the case of unbounded graphons to the case of bounded graphons, we
again use the transformation $W\mapsto \tanh W$, which maps arbitrary
graphons to bounded graphons.
\end{proof}

\begin{proof}[Proof of Theorem~\ref{thm:equiv2}]
We first note that the implications (iii) $\Rightarrow$ (ii) $\Rightarrow$
(i) are trivial.  So all that remains to prove is that (i) $\Rightarrow$
(iii), and by Theorem~\ref{thm:equiv-over-01}, it will be enough to prove
this for graphons $W$ and $W'$ over $[0,1]$ equipped with the uniform
distribution.

Assume thus that $W$ and $W'$ are graphons over $[0,1]$ with
$\delta_\square(W,W')=0$. By Lemma~\ref{lem:delta-over-01} this implies that
$W$ and $W'$ can be coupled in such a way that $\|W-W'\|_\square=0$, which
in turn implies that $W(x,y)=W'(x',y')$ almost surely with respect to this
coupling.  As a consequence, $\delta_\square(\tanh W,\tanh W')=0$.  By the
results of \cite{BCL10}, this implies that $\tanh W$ and $\tanh W'$ are
equivalent, which in turn gives that $W$ and $W'$ are equivalent, as
required.
\end{proof}

\section{Concentration bounds}
\label{sec:concentration}

We start with a slight generalization of the multiplicative Chernoff bound.

\begin{lem}\label{lem:Chernoff}
Let $X_1,\dots,X_n$ be independent random variables with values in $\R$, let
$X=\sum_{i=1}^nX_i$, and suppose there exists $X_0\in [0,\infty)$ such that
\[
\sum_i\E[X_i^m]\leq X_0 \quad\text{for all $m\geq 2$.}
\]
Then
\[
\Pr(X-\E[X]\geq X_0 t) \leq \exp\left(-\min\{t,t^2\}\frac{X_0}3\right)
\]
for $t \ge 0$.
\end{lem}

\begin{proof}
As in the proof of the standard Chernoff bound, we estimate the expectation
of $e^{\alpha X}$ for a constant $\alpha\geq 0$ to be determined later.  To
this end, we first bound
\begin{align*}
\E[e^{\alpha X_i}]&=
1+ \alpha \E[X_i] +\sum_{m\geq 2} \frac{\alpha^m \E[X_i^m]}{m!}
\\
&\leq\exp\Bigl(\alpha \E[X_i] +\sum_{m\geq 2} \frac{\alpha^m \E[X_i^m]}{m!}\Bigr),
\end{align*}
which together with the assumption of the lemma proves that
\[
\E[e^{\alpha X}]
\leq
\exp\Bigl(\alpha \E[X] +\sum_{m\geq 2} \frac{\alpha^m }{m!}\sum_i\E[X_i^m]\Bigr)
\leq e^{\alpha \E[X]+(e^\alpha-\alpha-1)X_0}.
\]
As a consequence,
\begin{align*}
\Pr\bigl(X\geq E(X)+t X_0)
&=
\Pr\Bigl(e^{\alpha X-\alpha \E[X]-t\alpha X_0}\geq 1\Bigr)
\\
&\leq
\E[e^{\alpha X}]e^{-\alpha \E[X]-t\alpha X_0}
\\
&\leq e^{(e^\alpha-\alpha-1)X_0-t\alpha X_0}.
\end{align*}
Choosing $\alpha=\log(1+t)$ gives $ e^\alpha-1-\alpha-t\alpha
=t-(t+1)\log(t+1) $ and hence
\[
\Pr\bigl(X\geq E(X)+t X_0)
\leq e^{-X_0((t+1)\log(t+1)-t)}\leq \exp\left(-\frac {X_0}3 \min\{t,t^2\}\right).
\]
\end{proof}

Lemma~\ref{lem:Chernoff} immediately implies the following lemma. To state
it, we define, for an arbitrary symmetric  matrix $\Q\in [0,1]^{n\times n}$
with empty diagonal, the random symmetric matrix
$A=\Bern(\Q)\in\{0,1\}^{n\times n}$ obtained by setting $A_{ij}=A_{ji}=1$ with
probability $\Q_{ij}$, independently for all $i<j$, and $A_{ij}=0$ whenever
$i=j$. Note that with this notation, $\E[A_\pi]=\Q_\pi$ for all $\pi\colon
[n]\to[k]$. The following lemma states that $A_\pi$ is tightly concentrated
around its expectation.

\begin{lem}
\label{cor:L1-concentration} Let $1\leq k\leq n$, let $\Q$ be a symmetric
$n\times n$ matrix with entries in $[0,1]$ and empty diagonal, and let
$A=\Bern(\Q)$. Let $\eps$ be the random variable $ \eps=\max_{\pi\colon
[n]\to [k]}\|A_\pi-\Q_\pi\|_1$. Then
\begin{equation}
\label{eq:L1-1}
\E[\eps]\leq 9\sqrt{\rho(\Q)\paren{\frac{1+\log k}{n} + \frac{k^2}{n^2}}}.
\end{equation}
If $n\rho(\Q)\geq 1$, then with probability at least $1-e^{-n}$
\begin{equation}
\label{eq:L1-2}
\eps\leq 8\sqrt{\rho(\Q)\paren{\frac{1+\log k}{n} + \frac{k^2}{n^2}}}.
\end{equation}
\end{lem}

Recall that $\rho(Q)$ means $\frac{1}{n^2} \sum_{i,j} Q_{ij}$.

\begin{proof}
We begin with the proof of \eqref{eq:L1-2}. We distinguish two cases:

If $\frac{1+\log k}{n} + \frac{k^2}{n^2}\geq \rho(\Q)$, all we need to show
is that with probability at least $1-e^{-n}$, the left side is at most
$8\rho(\Q)$.  To prove this, we bound
\[
\|A_\pi-\Q_\pi\|_1\leq \|A_\pi\|_1+\|\Q_\pi\|_1
=\|A\|_1+\|\Q\|_1.
\]
Now we apply Lemma~\ref{lem:Chernoff} to the random variable
$X=\sum_{i<j}A_{ij}$.  Because $\E[\sum_{i<j} A_{ij}^m]=\sum_{i<j} Q_{ij}
=\frac{n^2}2\rho(\Q)$, we can take $X_0 = \frac{n^2}2\rho(\Q)$.  Taking
$t=6$, we see that with probability at least $1-e^{-n^2\rho(\Q)}\geq
1-e^{-n}$,
\[
\|A_\pi-\Q_\pi\|_1 \leq 2\|\Q\|_1 + 6\rho(\Q)= 8\rho(\Q).
\]

If $\frac{1+\log k}{n} + \frac{k^2}{n^2}\leq \rho(\Q)$, we will use a union
bound over all $\pi\colon [n]\to[k]$.  Considering first a fixed $\pi\colon
[n]\to[k]$, we rewrite
\begin{align*}
\|A_\pi-\Q_\pi\|_1&=\frac 2{{n^2}}\sum_{u<v}(\Q_{uv}-A_{uv})\text{sign}((\Q_\pi)_{uv}-(A_\pi)_{uv})\\
&=\max_{B\in\cB_\pi}\frac 2{{n^2}}\sum_{u<v}B_{uv}(\Q_{uv}-A_{uv}),
\end{align*}
where $\cB_\pi$ consists of all matrices with entries $\pm1$ that are
constant on the partition classes of $\pi$ (note that  $\cB_\pi$ has size
$2^{k_0^2}$, where $k_0\leq k$ is the number of non-empty classes in $\pi$).
Applying Lemma~\ref{lem:Chernoff} again, this time to the random variables
$B_{uv}A_{uv}$, noting that
$\sum_{u<v}\E[(B_{uv}A_{uv})^\alpha]\leq\sum_{u<v}\E[A_{uv}] \leq \frac
{n^2}2\rho(\Q)$, and using the union bound to deal with the maximum over
$B'\in \cB_\pi$, we find that
\[
\Pr\left(\|A_\pi-\Q_\pi\|_1 \geq t\rho(\Q)\right)
\leq 2^{k^2}\exp\left(-\frac {\min\{t,t^2\}}6 n^2\rho(\Q)\right).
\]
Setting
\[
t=6 \sqrt{{\frac{1+\log k}{n\rho(\Q)} + \frac{k^2}{n^2\rho(\Q)}}},
\]
our case assumption implies that $t\leq 6$, which in turn implies that
\[
\min\{t,t^2\}\geq \frac {t^2}{6}= 6\paren{\frac{1+\log k}{n\rho(\Q)} + \frac{k^2}{n^2\rho(\Q)}}.
\]
As a consequence, for each partition $\pi\colon [n]\to [k]$,
\begin{align*}
\Pr(\|A_\pi-\Q_\pi\|_1\geq t\rho(\Q))
&\leq  \exp\left(k^2\log 2-n(1+\log k) - k^2\right)\\
&\leq e^{- n(1+\log k)}.
\end{align*}
Taking the union bound over all partitions $\pi\colon [n]\to[k]$, we obtain
the desired bound.

All that remains is to prove \eqref{eq:L1-1}. If $n\rho(\Q)\leq 1$, we bound
\[
\E[\eps]\leq \E[\|A-\Q\|_1]\leq \|\Q\|_1+\E[\|A\|_1]=2\rho(\Q)\leq 2\sqrt{\rho(\Q)/n}.
\]
If $n\rho(\Q)\geq 1$, we use \eqref{eq:L1-2} and the fact that $\eps\leq
\|\Q\|_1+\|A\|_1\leq 2$ to bound
\[
\E[\eps]\leq
8\sqrt{\rho(\Q)\paren{\frac{1+\log k}{n} + \frac{k^2}{n^2}}}+2e^{-n}.
\]
Because $2e^{-n}\leq 1/n\leq \sqrt{n\rho(G)}/n=\sqrt{\rho(G)/n}$, this
completes the proof.
\end{proof}

Our next lemma states that a similar bound holds for the cut norm of $A-\Q$.

\begin{lem}
\label{lem:cut-concentration} Let $n\geq 2$, let $\Q$ be a symmetric
$n\times n$ matrix with entries in $[0,1]$ and empty diagonal, and let
$A=\Bern(\Q)$. Then
\[
\E[\|A-\Q\|_\square]\leq 16\sqrt{\frac{\rho(\Q)}n}.
\]
If $n\rho(\Q)\geq 1$, then with probability at least $1-e^{-n}$,
\begin{equation} \label{eq:cut-2}
\|A-\Q\|_\square\leq 15\sqrt{\frac{\rho(\Q)}n}.
\end{equation}
\end{lem}

\begin{proof}
A bound of the form \eqref{eq:cut-2} can easily be inferred from Lemma~7.2 in
\cite{BCCZ14a}. For the convenience of the reader, we given an independent,
slightly simpler proof here.

Let $\cF_n$ be the set of functions $f\colon [n]\to \{-1,+1\}$. It is not
hard to check that
\begin{align*}
\|A-\Q\|_\square&\leq\max_{f,g\in \cF_n}\frac 1{n^2}\sum_{i,j}f(i)g(j) (A_{ij}-\Q_{ij})\\
&\le\max_{f,g\in \cF_n}\frac 2{n^2}\sum_{i<j}f(i)g(j) (A_{ij}-\Q_{ij}).
\end{align*}
Proceeding as in the proof of Lemma~\ref{cor:L1-concentration}, a union
bound and Lemma~\ref{lem:Chernoff} now imply that
\[
\Pr\left(\|A-\Q\|_\square \geq t\rho(\Q)\right)
\leq 4^n\exp\left(-\frac {\min\{t,t^2\}}6 n^2\rho(\Q)\right).
\]
Choosing $t=6\log (4e)/\sqrt{n\rho(\Q)}$ and observing that $6\log (4e)\leq
15$ then gives the high probability bound.  The bound in expectation follows
from this bound and the observation that $\|A-\Q\|_\square\leq 2\rho(\Q)$.
Indeed, if $n\rho(\Q)\geq 1$, then
\begin{align*}
15\sqrt{\rho(\Q)/n} + 2e^{-n}\rho(\Q)& \leq15\sqrt{\rho(\Q)/n}
+2\rho(\Q)/(en)\\
&\leq 16\sqrt{\rho(\Q)/n}
\end{align*}
(for the final step recall that $\rho(\Q)\leq 1$), and if $n\rho(\Q)\leq 1$,
then $2\rho(\Q) \leq 2\sqrt{\rho(\Q)/n}$.
\end{proof}

\section{Proofs of Lemmas~\ref{lem:block-model-approx-n}, \ref{lem:WH-as}, and \ref{lem:WH-as-block}}
\label{app:Auxiliary}

We start with the following lemma, which is an easy consequence of the law
of large numbers for $U$-statistics.

\begin{lem}\label{lem:U-Stat}
Let $(\Omega,\cF,\pi)$ be a probability space, and let $W\colon
\Omega\times\Omega\to \R$ be in $L^p$ for some $p\geq 1$.  Then
$\|H_n(W)\|_p\to \|W\|_p$ a.s.
\end{lem}

\begin{proof}
Define $U=|W|^p$, and choose $x_1,\dots,x_n$ i.i.d.\ with distribution
$\pi$. Then
\[
\|{\HH_n}(W)\|_p^p=
\frac 1{n^2}\sum_{i\neq j}|W(x_i,x_j)|^p=
\frac 1{n^2}\sum_{i\neq j}U(x_i,x_j).
\]
By the strong law of large numbers for $U$-statistics (see, for example,
\cite{H}), the right side converges to $\|U\|_1=\|W\|_p^p$ as claimed.
\end{proof}

Next we prove Lemma~\ref{lem:WH-as}.

\begin{proof}[Proof of Lemma~\ref{lem:WH-as}]
We first note that the statement clearly holds if $W$ is replaced by the
block model $W^{(k)}=W_{\mathcal P_k}$, where $\mathcal P_k$ is the
partition of $[0,1]$ into consecutive intervals of length $1/k$.  To see
this, one just needs to use the fact that as $n\to\infty$, the fraction of
points $x_i$ which fall into the $j^{\text{th}}$ interval  converges a.s.\
to $1/k$.

To prove the statement of the lemma for general $W$, we will use
Lemma~\ref{lem:U-Stat}. Let $\rho=\rho_n$, fix $\eps>0$, choose $k$ so that
$\|W-W^{(k)}\|_p\leq \eps$, and let $M$ be large enough that $\|W1_{W\geq
M}\|_p\leq \eps$. Also, define $W_\rho=\min\{W,1/\rho\}$. Noting that $\frac
1\rho \Q_n=\HH_n(W_{\rho})$, we then bound
\begin{align*}
\|W-\frac 1{\rho}\Q_n\|_p &=\|W-\HH_n( W_{\rho})\|_p\\
&\leq
\|W-W^{(k)}\|_p
+\|W^{(k)}-\HH_n(W^{(k)})\|_p\\
& \quad \phantom{}
+\|\HH_n(W^{(k)})-\HH_n(W)\|_p
+\|\HH_n(W)-\HH_n(W_{\rho})\|_p.
\end{align*}
Assuming $n$ is large enough to ensure that $\rho^{-1}\geq M$ (which in turn
implies that $|W-W_{\rho}|=W-W_{\rho}\leq W1_{W\geq M}$), we then bound the
right side by
\[
\eps+\|W^{(k)}-\HH_n(W^{(k)})\|_p+\|\HH_n(W^{(k)}-W)\|_p+\|\HH_n(W1_{W\geq M})\|_p.
\]
As $n\to\infty$, the second term goes to zero with probability one, and the
third and the fourth both converge to quantities which are at most $\eps$ by
Lemma~\ref{lem:U-Stat}.  Thus, with probability one, the limit superior of
$\|W-\frac 1{\rho}\Q_n\|_p$ is at most $3\eps$.  Since $\eps$ was arbitrary,
this proves the claim.
\end{proof}

Next we prove Lemma~\ref{lem:block-model-approx-n}.  To this end, we start
with a simple technical lemma.  We use $\lambda$ to denote the Lebesgue
measure on $[0,1]$ or $[0,1]^2$ (depending on the context), and, as usual,
we use $A\symmdiff B$ to denote the symmetric difference of two sets $A,B$,
i.e., $A\symmdiff B=(A\setminus B)\cup (B\setminus A)$.

\begin{lem}
\label{lem:step-function-distance} Let $W$ and $W'$ be of the form
$W=\sum_{i,j}B_{ij}1_{Y_i\times Y_j}$ and $W'=\sum_{i,j}B_{ij}1_{Y_i'\times
Y_j'}$, where $B$ is a $k\times k$ matrix, and $(Y_1,\dots,Y_k)$,
$(Y'_1,\dots,Y_k')$ are partitions of $[0,1]$. If $\lambda(Y_i\symmdiff
Y'_i)\leq \eps \lambda(Y_i)$ for all $i$, then
\[
\|W-W'\|_p\leq \sqrt[p]{2\eps(1+\eps)}\|W\|_p.
\]
\end{lem}

\begin{proof}
We begin with the bound
\begin{align*}
\|W-W'\|_p^p &= \Big\|\sum_{i,j}(W 1_{Y_i \times Y_j} - W' 1_{Y_i' \times Y_j'})\Big\|_p^p\\
&\le \sum_{i,j} |B_{ij}|^p \lambda\big((Y_i \times Y_j) \symmdiff (Y_i' \times Y_j')\big).
\end{align*}
We have
\[
(Y_i \times Y_j) \symmdiff (Y_i' \times Y_j') \subseteq \big((Y_i \cup Y_i') \times (Y_j \symmdiff Y_j')\big)
\cup \big((Y_i \symmdiff Y_i') \times (Y_j \cup Y_j')\big).
\]
Combining this containment with $\lambda(Y_i\symmdiff Y'_i)\leq \eps
\lambda(Y_i)$ and $\lambda(Y_i \cup Y_i') \le (1+\eps) \lambda(Y_i)$ yields
\[
\|W-W'\|_p^p \le 2 \eps (1+\eps) \sum_{i,j} |B_{ij}|^p \lambda(Y_i \cup Y_j) = 2\eps (1+\eps) \|W\|_p^p,
\]
as desired.
\end{proof}

\begin{rem} \label{rem:W-W'-bd1}
A slight variation of the above proof also shows that
\[
\|W-W'\|_p\leq
\max_i\frac 1{\lambda(Y_i)^{2/p}}
\|W\|_p,
\]
no matter how large the measure of the symmetric differences $Y_i\symmdiff
Y_i'$ is.  To see this, just bound
\begin{align*}
\|W-W'\|_p^p &\leq \sum_{i,j} |B_{ij}|^p \max\{\lambda(Y_i \times Y_j),\lambda (Y_i' \times Y_j')\}\\
&\leq \|W\|_p^p \max_i\left(\frac{\lambda(Y_i')}{\lambda(Y_i)}\right)^2\\
&\leq \|W\|_p^p \max_i\frac{1}{\lambda(Y_i)^2}.
\end{align*}
\end{rem}

\begin{proof}[Proof of  Lemma~\ref{lem:block-model-approx-n}]
If $\kappa=1$, $\cB_{n,\geq\kappa}(W)=\cB_{\geq\kappa}(W)$ and there is
nothing to prove. We may therefore assume without loss of generality that
$\kappa\in (0,1)$.

To prove the first bound, we write $W'$ as $(\pp,B)$ and reorder the
elements of $[k]$ so that $p_1\leq p_2\leq\dots\leq p_k$.  Also, without
loss of generality, we may remove all labels with $p_i=0$, so that
$p_i\geq\kappa$ for all $i$. Define $W''=({\mathbf p}'', B)$, where $\mathbf
p''$ is obtained from $\mathbf p$ so that for each $i$, $p_1'' + \dots +
p_i''$ equals $p_1+\dots+p_i$ rounded to the nearest multiple of $1/n$ (with
the convention that in the case of ties, we choose the point to the left).
After embedding both $W'$ and $W''$ into the space of graphons on $[0,1]$,
we can write the resulting graphons $\widetilde W''=\WW{W''}$ and
$\widetilde W'=\WW{W'}$ in the form $\widetilde
W''=\sum_{i,j}B_{ij}1_{Y_i''\times Y_i''}$ and $\widetilde
W'=\sum_{i,j}B_{ij}1_{Y_i\times Y_i}$, where $Y_i$ and $Y_i''$ are intervals
whose endpoints differ by at most $1/(2n)$. As a consequence
$\lambda(Y_i\symmdiff Y_i'')\leq \frac 1n\leq \frac 1{\kappa
n}\lambda(Y_i)$. By  Lemma~\ref{lem:step-function-distance} and the fact
that $\frac 1{\kappa n}\leq 1$, this  implies that
\begin{equation}
\label{block-rounding-bd}
\|\WW{W'}-\WW{W''}\|_p^p
\leq \frac 4{\kappa n}\|W'\|_p^p.
\end{equation}
To complete the proof of the first bound, all we need to show is that
$W''\in\cB_{n,\geq\kappa}$, which means we need to show that $n
p_i''=n\lambda(Y_i'')\geq \lfloor\kappa n\rfloor$ for all $i$.  Let $i_0$ be
the first $i$ such that $np_i$ is not an integer. For $i<i_0$, we then have
$np_i''=np_i\geq\kappa n\geq \lfloor\kappa n\rfloor$.  On the other hand,
for $i\geq i_0$, we can use $|np_i-np_{i}''|\leq 1$, which follows from
$|n(p_1+\dots+p_i) - n (p_1''+\dots+p_i'')| \le 1/2$.  We then conclude that
$np_i''\geq np_i-1\geq np_{i_0}-1>\lfloor np_{i_0}\rfloor -1\geq
\lfloor\kappa n\rfloor-1$, where we used that $np_{i_0}$ is not an integer.
Since $np_i''$ is an integer, this implies $np_i''\geq \lfloor
n\kappa\rfloor$, which shows that $W''\in\cB_{n,\geq\kappa}$.  Identifying
$W''$ with the corresponding matrix in $\cA_{n,\geq\kappa}$, this proves the
first bound.

To prove the second bound we first observe that the minimizer  $W''=(\mathbf
p'',B)\in\cB_{n,\geq\kappa}$ obeys the bound $\|W''\|_p\leq{2}\|W\|_p$. Our
task is now to find a block model $W'\in\cB_{\geq\kappa}$ that approximates
$W''$ in the norm $\delta_p$. Let $k''$ be the number of classes in $W''$;
again, we assume without loss of generality that they are all non-empty,
which means we have that $p_i''\geq\kappa_n$ for all $i\in [k'']$, where
$\kappa_n:=\frac 1n\lfloor n\kappa\rfloor$.

We would like to increase $p_i''$ to $\kappa$ whenever it is smaller than
that, while compensating for this by decreasing those probabilities that are
larger than $\kappa$.  However, there is a potential obstruction, namely
that $k'' \kappa$ could be greater than $1$, in which case it is clearly
impossible to increase all $k''$ probabilities to at least $\kappa$. For
comparison, we know that $k'' \kappa_n \le 1$, but that is a slightly weaker
assertion.

To deal with this difficulty, we will show that there exist some $n_0$
depending on $\kappa$ such that for $n\geq n_0$, we do have $\kappa k''\leq
1$.  First, note that $\kappa_n > \kappa - \frac{1}{n}$.  Thus,
\[
k'' \le \left\lfloor \frac{1}{\kappa - 1/n}\right\rfloor.
\]
As $n \to \infty$, $1/(\kappa - 1/n)$ approaches $1/\kappa$ from above, and
thus
\[
\left\lfloor \frac{1}{\kappa - 1/n}\right\rfloor = \left\lfloor \frac{1}{\kappa}\right\rfloor
\]
for all sufficiently large $n$.  If we take $n_0$ to be sufficiently large,
then for $n \ge n_0$ we have
\[
k'' \kappa \le \left\lfloor \frac{1}{\kappa}\right\rfloor \kappa \le 1.
\]

Given this, we  now define $W'=(\pp,B)$ as follows: let $I_-$ be the set of
indices $i\in [k'']$ such that $p_i''<\kappa$, and let $\delta=\sum_{i\in
I_-}(\kappa-p_i'')$. For $i\in I_-$, we then set $p_i=\kappa$, while for
$i\notin I_-$ we first decrease the largest  $p_i''$ until we either hit
$\kappa$ or have used up the excess $\delta$.  If we stop because we hit
$\kappa$, then we move to the next largest $p_i''$, etc.  Since in the
second step, we will eventually use up the excess $\delta$, this process
constructs a distribution $\pp$ such that $p_i\geq\kappa$ for all $i\in
[k'']$, and such that $\sum_i|p_i''-p_i|=2\delta$.  Note for future
reference that $\delta \le k''/n$.

Writing the embedding $\WW{W''}$ of $W''$ into the set of graphons over
$[0,1]$ as $\sum_{i,j}B_{ij}1_{Y_i''\times Y_j''}$, we construct
corresponding measurable sets $Y_i$ such that $Y_1,\dots,Y_{k''}$ forms a
partition of $[0,1]$ with $\lambda(Y_i)=p_i$ and $\lambda(Y_i\symmdiff
Y_i'') \le |p_i-p''_i|$.  (Each set $Y_i$ will be either a superset or a
subset of $Y''_i$, according to whether $p''_i$ was increased or decreased.)

For $i \in I_-$,
\[
\lambda(Y_i\symmdiff
Y_i'') \le |p_i-p''_i| \le \frac{1}{n} \le \frac{1}{\kappa_n n} \lambda(Y_i'').
\]
For $i \not\in I_-$,
\[
\lambda(Y_i\symmdiff
Y_i'') \le |p_i-p''_i| \le \delta \le \frac{k''}{n} \le \frac{1}{\kappa^2 n} \lambda(Y''_i).
\]
When $n$ is sufficiently large, $\kappa_n \ge \kappa^2$.  Increase $n_0$
enough for this to hold, as well as $n_0 \ge 1/\kappa^2$.  Then for $n \ge
n_0$,
\[
\delta_p(W',W'')\leq  \sqrt[p]{\frac {4}{\kappa^2 n}}\|W''\|_p
\leq {2 }\sqrt[p]{\frac {4}{\kappa^2 n}}\|W\|_p
\]
by Lemma~\ref{lem:step-function-distance}, as in the proof of the first
bound.  This concludes the proof of the second bound.
\end{proof}

For bounded graphons, the next lemma was proved in \cite{BCS15}.

\begin{lem}
Let $\cP=(Y_1,\dots,Y_k)$ be a partition of $[0,1]$ into consecutive
intervals, and let $W$ be a graphon over $[0,1]$ that is constant on sets of
the form $Y_i\times Y_j$.  If $x_1,\dots,x_n$ are chosen i.i.d.\ uniformly
at random from $[0,1]$ and $H_n$ is the $n\times n$ matrix with entries
$W(x_i,x_j)$, then
\[
\hat\delta_p(H_n,W)\leq \sqrt[p]{2\eps(1+\eps)}\|W\|_p,
\]
where $\eps$ is the random variable
\[
\eps=\max_{i\in[k]}\frac 1{\lambda(Y_i)}\Bigl(\frac 1n+\left|\frac {n_i}n-\lambda(Y_i)\right|\Bigr),
\]
with $n_i$ denoting the number of points $x_\ell$ that lie in $Y_i$.
\end{lem}

\begin{proof}
Let $I_1,\dots,I_n$ be a partition of $[0,1]$ into adjacent intervals of
lengths $1/n$. Then $\WW{H_n}$ is of the form $\sum_{i,j}
B_{ij}1_{Y_i'\times Y_j'}$, where $Y_i'$ is the union of $n_i$ of the
intervals $I_1,\dots,I_n$ (which particular $n_i$ intervals depends on the
labeling of the vertices of $H_n$). In fact, given a map $\pi\colon
[n]\to[k]$, define $Y_i'=Y_i'(\pi)$ to be the union of all intervals
$I_\ell$ such that $\pi(\ell)=i$, and let
$W(\pi)=\sum_{i,j}B_{ij}1_{Y_i'(\pi)\times Y_j'(\pi)}$. Then
\[
\hat\delta_2(H_n,W)=\min_{\pi}\|W(\pi)-W\|_2,
\]
where the minimum is over all $\pi$ such that $|\pi^{-1}(\{i\})|=n_i$ for
all $i$. In view of Lemma~\ref{lem:step-function-distance}, we will want to
keep the Lebesgue measure of $Y_i\symmdiff Y_i'$ small for all $i$.  We
claim that this is indeed possible, and that $\pi$ can be chosen in such a
way that
\begin{equation}
\label{alignment-of-i}
\lambda(Y_i\symmdiff Y_i')\leq\left|\frac {n_i}n-\lambda(Y_i)\right|+\frac 1n
\quad\text{for all $i$}.
\end{equation}
To prove this, we note that choosing $\pi$ is equivalent to choosing, for
all $i$, $n_i$ of the intervals $I_1,\dots,I_n$ to make up $Y_i'$.

Let $\tilde Y_1,\dots,\tilde Y_k$ be obtained from $Y_1,\dots,Y_k$ by
rounding the endpoints to the nearest integer multiples of $1/n$, choosing
the multiple to the left in case of a tie. With this convention,
\[
\lfloor \lambda(Y_i)n \rfloor \le \lambda(\tilde Y_i)n\le \lceil \lambda(Y_i)n\rceil.
\]
Thus, if $n_i\leq\lambda(Y_i)n$, then $n_i\leq n\lambda(\tilde Y_i)$, while
if $n_i\ge\lambda(Y_i)n$, then $n_i\ge n\lambda(\tilde Y_i)$. Keeping this
in mind, we see that for $n_i\leq\lambda(Y_i)n$, we can find at least $n_i$
intervals $I_\ell$ that, except possibly for their endpoints, are subsets of
$\tilde Y_i$.  We will define $Y_i'$ to be the union of these intervals. In
a similar way, if $n_i>\lambda(Y_i)n$, we choose $n\lambda(\tilde Y_i)\leq
n_i$ intervals (namely, those forming $\tilde Y_i$) to build a preliminary
set $Y_i^{(0)}$.  Having done this for all $i$, we take a second run through
all $i$ with $n_i>\lambda(Y_i)n$, choosing an arbitrary set of
$n_i-\lambda(\tilde Y_i)n$ intervals $I_\ell$ from those not yet assigned at
this point.  At the end of this round, we end up with sets $Y_i'$ such that
$Y_i'$ is the union of $n_i$ intervals from $I_1,\dots,I_n$, with the
additional property that
\[
\text{either} \quad Y_i'\subseteq\tilde Y_i\quad\text{or}\quad Y_i'\subseteq\tilde Y_i.
\]
But this implies that $\lambda(Y_i'\symmdiff\tilde Y_i)=|\frac{n_i}n
-\lambda(\tilde Y_i)|$ for all $i$. Since the endpoints of $Y_i$ get shifted
by at most $1/(2n)$ in order to obtain $\tilde Y_i$, the additional error in
going from $\tilde Y_i$ to $Y_i$ is at most $1/n$, proving
\eqref{alignment-of-i}. Combined with
Lemma~\ref{lem:step-function-distance}, this concludes the proof.
\end{proof}

Finally, the following lemma implies Lemma~\ref{lem:WH-as-block}.

\begin{lem}
Let $\eps$ and the other notation be as in the previous lemma, suppose that
all sizes of $\cP$ have measure at least $\kappa$, and let $\eta\in (0,1)$.
Then
\[
\eps\leq \frac 1{\kappa n}+\max\left\{\frac3{ n\kappa}\log \frac 2{\kappa \eta},\sqrt{\frac3{ n\kappa}\log \frac 2{\kappa \eta}}\right\}
\]
with probability at least $1-\eta$. As a consequence, if $C$ is a positive real number, then
\[
\hat\delta_p(H_n,W)=O_p\left(\sqrt[2p]{\frac {\log n}{n\kappa}}\right)\|W\|_p
\]
whenever $\frac{\log n}{n\kappa}\leq  C$, with the constant implicit in the
$O_p$ symbol depending on $C$.  In addition, if $\kappa=\kappa_n$ is such
that $\limsup \frac 1{\kappa n}\log n< C$, then with probability one, there
exists a random $n_0$ such that for $n\geq n_0$,
\[
\hat\delta_p(H_n,W)= O\left({\sqrt[2p]{\frac {\log n}{n\kappa}}}\right)\|W\|_p,
\]
with the constant implicit in the big-$O$ symbol again depending on $C$.
\end{lem}

\begin{proof}
By the multiplicative Chernoff bound,
\begin{align*}
\Pr\left(\Bigl|\frac {n_i}n-\lambda(Y_i)\Bigr|\geq t\lambda(Y_i)\right)
&\leq 2\exp\left(-\frac{ n\lambda(Y_i)}3\min\{t,t^2\}\right)\\
&\leq 2 \exp\left(-\frac{ n\kappa}3\min\{t,t^2\}\right),
\end{align*}
so by the union bound and the fact that the number $k$ of classes is at most
$1/\kappa$, we get
\[\eps\leq t+\frac 1{\kappa n}
\quad\text{with probability at least}\quad
1-\frac 2\kappa\exp\left(-\frac{ n\kappa}3\min\{t,t^2\}\right).
\]
Setting $y=\frac3{ n\kappa}\log \frac 2{\kappa \eta}$ we see that with
probability at least $1-\eta$, $\eps\leq t+\frac 1{\kappa n}$ whenever
$\min\{t,t^2\}\geq y$. This implies the  bound on $\eps$.

For the remaining part of the proof, choose $\eta=2n^{-2}$.  Then with probability at
least $1-2n^{-2}$,
\begin{align*}
\eps&\leq \frac 1{\kappa n}+\max\left\{\frac3{ n\kappa}\log \frac{n^2}{\kappa},\sqrt{\frac3{ n\kappa}\log \frac{n^2}{\kappa}}\right\} &\\
&\leq \frac 1{\kappa n}+\max\left\{\frac9{ n\kappa}\log 2Cn,\sqrt{\frac9{ n\kappa}\log 2Cn}\right\}&\text{(because $\frac 1{n \kappa}\leq \frac C{\log n}\leq \frac C{\log 2}\leq 2C$)}\\
&\leq\sqrt{\frac{C''\log n}{n \kappa}}\leq\sqrt {C C''},
\end{align*}
for some $C''$ depending on $C$.  This
implies $2\eps (1+\eps)\leq 2(1+\sqrt {CC''})\sqrt{\frac{C''\log n}{ n\kappa}}=:C'\sqrt{\frac{\log n}{ n\kappa}}$
and hence
\[
\hat\delta_p(H_n,W)
\leq \sqrt[2p]{\frac{C'\log n}{ n\kappa}}\|W\|_p.
\]
Since the failure probability $2n^{-2}$ is summable, this implies the a.s.\
statement. To prove the statement in probability, we note that by
Remark~\ref{rem:W-W'-bd1}, $\hat\delta_p(H_n,W)\leq \kappa^{-2/p}\|W\|_p$,
which shows that
\begin{align*}
\E\left[(\hat\delta_p(H_n,W))^p\right]
&\leq \left(\sqrt{\frac{C'\log n}{ n\kappa}}
+\frac {\eta}{\kappa^2}\right)\|W\|_p^p\\
&=\left(\sqrt{\frac{C'\log n}{ n\kappa}}
+\frac {2}{\kappa^2n^2}\right)\|W\|_p^p.
\end{align*}
This implies the statement in probability.
\end{proof}

\end{document}